\documentclass[12pt]{article}

\usepackage[a4paper,left=20mm,top=20mm,right=20mm,bottom=30mm]{geometry}

\usepackage{graphicx}
\usepackage{algorithm}
\usepackage[noend]{algpseudocode}
\usepackage{amsmath, amsfonts, amssymb, amsthm}
\usepackage{mathtools}
\usepackage{enumitem}

\usepackage{authblk}
\usepackage[numbers, sort&compress]{natbib}

\usepackage[colorlinks]{hyperref}
\hypersetup{
	citecolor=blue,
   linkcolor=red,
}
\usepackage[font=footnotesize,width=.85\textwidth,labelfont=bf]{caption}
\usepackage[mathcal]{euscript}
\usepackage{mathptmx} 

\newcommand{\rev}[1]{\begingroup#1\endgroup}

\DeclareMathAlphabet{\mathpzc}{OT1}{pzc}{m}{it}

  \newenvironment{owndesc}%
    {\begin{description}[leftmargin = 0.2cm, labelsep = 0.2cm]}
    {\end{description}}

\newcommand{\bul}{\mathfrak{s}}
\newcommand{\squ}{\mathfrak{d}}

\newcommand{\fV}{f_{\mathpzc{V}}}
\newcommand{\fE}{f_{\mathpzc{E}}}

\newcommand{\mc}{\mathcal}
\newcommand{\parent}{\ensuremath{\operatorname{par}}}
\newcommand{\lca}{\ensuremath{\operatorname{lca}}}
\newcommand{\Gen}{\ensuremath{\mathbb{G}}}
\newcommand{\Spe}{\ensuremath{\mathbb{S}}}
\newcommand{\V}{\ensuremath{\mathcal{V}}}
\newcommand{\W}{\ensuremath{\mathcal{W}}}

\newtheorem{defi}{Definition}
\newtheorem{observe}{Observation}

\providecommand{\keywords}[1]{\textbf{\textit{Keywords: }} #1}

\newtheorem{theorem}{Theorem}
\newtheorem{lemma}{Lemma}
\newtheorem{proposition}{Proposition}
\newtheorem{corollary}{Corollary}

\makeindex             

\title{Reconciling Event-Labeled Gene Trees with MUL-trees and Species Networks}

\author[1,2]{Marc Hellmuth}
\author[3]{Katharina T.\ Huber} 
\author[3]{Vincent Moulton}

\affil[1]{Institute	 of Mathematics and Computer Science, University of Greifswald, Walther-
  Rathenau-Strasse 47, D-17487 Greifswald, Germany  \\ 	
	Email: \texttt{mhellmuth@mailbox.org}}

\affil[2]{
	Saarland University, Center for Bioinformatics, Building E 2.1, P.O.\ Box 151150, D-66041 Saarbr{\"u}cken, Germany
	  }

\affil[3]{School of Computing Sciences, University of East Anglia, Norwich, UK \\ 
	Email: \texttt{K.Huber@uea.ac.uk} and  \texttt{v.moulton@uea.ac.uk} 
 }

\date{}

\begin{document}

\maketitle

\abstract{ 
Phylogenomics commonly aims to construct evolutionary trees from genomic
sequence information. One way to approach this problem is to first estimate
event-labeled gene trees (i.e., rooted trees whose non-leaf vertices are
labeled by speciation or gene duplication events), and to then look for a
species tree which can be reconciled with this tree through a
\emph{reconciliation map} between the trees. In practice, however, it can
happen that there is no such map from a given event-labeled tree to
\emph{any} species tree. An important situation where this might arise is
where the species evolution is better represented by a 
\emph{network} instead of a
tree. In this paper, we therefore consider the
problem of reconciling event-labeled trees with species networks. In
particular, we prove that any event-labeled gene tree can be reconciled
with some network \rev{and that, under certain mild assumptions
on the gene tree, the network can even be assumed to be multi-arc free.} 
To prove this result, we show that we
can always reconcile the gene tree with some multi-labeled (MUL-)tree,
which can then be ``folded up" to produce the desired reconciliation and
network. In addition, we study the interplay between reconciliation maps
from event-labeled \rev{gene} trees to MUL-trees and networks.
Our results could be
useful for understanding how genomes have evolved after undergoing complex
\rev{evolutionary} events such as polyploidy.
}

\smallskip
\noindent

\sloppy
\sloppy

\section{Introduction}
\label{sec:intro}

Phylogenomics aims to find plausible hypotheses about the evolutionary
history of species based on genomic sequence information.
Such hypotheses often take the form of an evolutionary tree
whose leaves are labeled by the species in question, or a \emph{species tree}.
There are various ways to construct species trees from genomics data, many of 
which involve estimating the evolutionary history of the 
underlying genes, and then using the resulting \emph{gene trees} to construct 
a species tree \cite{posada2016phylogenomics}. 

A recent example of such an approach relies on using \emph{event-labeled gene trees} \cite{HHH+12,Hellmuth2017,HW:16b}. 
These are trees in which the  leaves correspond to the genes, 
inner vertices to ancestral genes, and the label of an inner vertex corresponds to 
the divergence event that led to the offspring. Such
events include \emph{speciation} and \emph{duplication} events
which correspond to genes that are orthologous or paralogous \cite{Fitch:70,Fitch:00},
respectively, that is,  they diverged after a speciation event or from one another within a 
species after a duplication event. To estimate event-labeled gene trees, 
sequence similarities and synteny information is first used to determine 
which genes are orthologous \cite{Altenhoff:09,AGGD:13,ASGD:11,CMSR:06,Lechner:11a,Lechner:14,inparanoid:10,TG+00,T+11},
and the trees are then estimated from the resulting orthology relations 
\cite{HHH+13,dondi2017approximating,lafond2015orthology,DEML:16,DONDI17,LDEM:16} 
using an underlying ``cograph''-structure \cite{HW:16,HHH+13}. 

Once an event-labeled gene tree has been estimated, the
task then becomes finding a species tree which accommodates
the gene history. This is essentially done by looking for a mapping
(a so-called \emph{reconciliation}) from the gene tree into some 
species tree which respects ancestral relationships. Note that the problem of 
reconciling gene trees with species trees
has been studied for some time (mainly for non-event-labeled gene trees)
\cite{BLZ:00,GJ:06,RLG+14,DCH:09,DES:14,VSGD:08,lafond2012optimal,SLX+12,huson2011survey,page1998genetree,
STDB:15,SD:12,MMZ:00,DRDB11,Doyon2010,EHL10,GCMRM:79,Tofigh2011}.
However, one issue with this overall process for constructing species trees is that it may not be possible to
reconcile an event-labeled gene tree with \emph{any} species tree \cite{HHH+12,Hellmuth2017}. 
This may be because of inaccuracies in estimating the trees, but a more fundamental 
problem can arise due to the fact that the species tree is not an appropriate way
to represent the evolution of the species in question. 

More specifically, it is well-known that species can come together with each 
other to form new ones through 
processes such as hybridization or recombination \cite{gontier2015reticulate} (a process
sometimes called \emph{reticulate  evolution}).
In this case it can be more appropriate to represent 
species evolution using a \emph{species network} instead
of a species tree \cite{bapteste2013networks}. 
Interestingly, as we shall see, it may be possible to reconcile 
an event-labeled gene tree with a species network \rev{even though}
it is not possible
to reconcile it with any species tree. The aim of this
paper is to better understand why this is the case, and to 
develop new theory and techniques for reconciling event-labeled gene trees with
\rev{species} networks.
Note that some work has appeared on
the problem of reconciling \emph{non event-labeled} gene trees
\rev{with}
\emph{given} species networks 
\cite{SMC:17,To2015}. However, to our best knowledge, 
the problem of reconciling event-labeled gene trees with \emph{unknown} species 
networks has not yet been considered.

The rest of this paper is organized as follows. 
We start with basic definitions concerning 
phylogenetic trees and networks  in Section~\ref{sec:prelim}. 
In Section~\ref{sec:rec-network}, we then give a new definition of a reconciliation map
between arbitrary, possibly non-binary, event-labeled gene trees and 
species networks, called  \emph{TreeNet-reconciliation maps}. 
We show that this definition is a natural generalization of the reconciliation
maps defined in \cite{To2015},
between \emph{binary} gene trees and \emph{binary} species networks. 
In Section~\ref{sec:rec-tree}, we also prove that TreeNet-reconciliation maps are equivalent to
reconciliation maps between trees \cite{HHH+12,Hellmuth2017,DCH:09}
in case the considered species network is a tree. Continuing with this theme, we also
examine the problem of determining when there exists a TreeNet-reconciliation  from 
an event-labeled tree to some species tree in Section~\ref{sec:triples}.

In Section~\ref{sec:existence} we show that for
every event-labeled gene tree $T$, there always exists a TreeNet-reconciliation map from 
$T$ to some species network $N$.  
To prove this result we employ the concept of so-called MUL-trees  \cite{cui2012polynomial,HES:14,lott2009inferring,huber2006reconstructing,scornavacca2009gene,czabarka2013generating,HM:06}.
Our proof essentially relies on first defining a reconciliation map between 
an event-labeled tree and a multiple-labeled tree, or \emph{MUL-tree}. Once
we have a reconciliation between $T$ and $M$, 
we then apply a so-called \emph{folding map} \cite{huber2016folding} 
to $M$ to produce a species
network $N$ (as explained in Section~\ref{sec:fold}), which induces 
the desired TreeNet-reconciliation map between $T$ and $N$.
This whole process allows us to determine 
the TreeNet-reconciliation map and network for the gene tree in polynomial-time 
in the number of genes.

The use of folding maps can result in species networks which 
contain multi-arcs, which may be somewhat unrealistic for applications.
In Section~\ref{sec:existence-multiarcFree} we therefore strengthen
the main result in Section~\ref{sec:existence}, showing that given an event-labeled gene tree
\rev{that satisfies a simple biological constraint on its speciation vertices (cf.\ Def.\ \ref{def:well-behaved}), there 
	always exists a TreeNet-reconciliation map from $T$ to some  \emph{multi-arc free} species network.}
Finally, in Section \ref{sec:rec-MUL}, we study the interplay between 
TreeNet-reconciliation maps from event-labeled gene trees to MUL-trees and networks.
\rev{This topic  has recently become of interest in the phylogenomics literature, where 
it has been applied to understand polyploid evolution \cite{gregg2017gene}.
In particular, in this work the authors use the interplay between MUL-trees and phylogenetic
networks  to try to distinguish between different types of ploidy events.}
We conclude with a discussion of some open problems in Section~\ref{sec:rec-MUL}.

\section{Preliminaries}
\label{sec:prelim}

\subsubsection*{Phylogenetic Networks and Trees}

In what follows, the set $X$ always denotes a finite set of size at least two.
Moreover, $\Gen$ and $\Spe$ will denote a set of genes and species,
respectively.

\rev{In general, in this paper a directed graph or \emph{digraph}, for short,
  may contain multi-arcs, that is, 
	two or more arcs that connect two vertices.}
We consider rooted, not necessarily 
binary phylogenetic trees and networks, called trees or networks for short
(see e.g. \cite{steel2016phylogeny} for an overview of phylogenetic trees and networks). To be more precise:
\begin{defi}
A {\em network} $N=(V,E)$ on $X$ is a directed acyclic graph (DAG)
with leaf set $L(N)=X$, multi-arcs allowed  \rev{that satisfies}  the following properties 
\begin{enumerate}
	\item[(N1)] There is a single root $\rho_N$ with indegree 0 and outdegree 1
						such that its unique child has indegree 1 and outdegree at least 2;
	\item[(N2)] $x\in X$ if and only if $x$ is an outdegree-0 \rev{and indegree-1} vertex. 
	\item[(N3)] Each vertex $v\in V^0\coloneqq V\setminus X$ with $v\neq \rho_N$ has either 
		
						indegree 1 and outdegree greater than 1 (\emph{tree vertex})						or 

						indegree greater than 1 and outdegree 1 (\emph{hybrid vertex})						
\end{enumerate}
If $X$ is a set of genes $\Gen$ (resp.\ species $\Spe$), then $N$ is called a \emph{gene} (resp.\ \emph{species})
\emph{network}. 
A network that has no multi-arcs is called {\em multi-arc free} and
a multi-arc free network that has no hybrid vertices is
called a (phylogenetic) tree.

We call a tree \emph{reduced} if it is obtained from a phylogenetic tree $T$ by removing 
the root $\rho_T$ and its unique incident arc from $T$. Hence, the root of a reduced 
tree always has degree greater or equal to 2. 
\label{def:network}
\end{defi}

Note, Property (N1) differs slightly from the usual notion as used e.g.\ in \cite{To2015, SMC:17}, 
where $\rho_N$ has indegree 0 and outdegree 2. 
We need this extra condition since our notion of a reconciliation map 
between a gene tree and a species network 
allows for the possibility that an event occurred before the first speciation event in the network.

Now, suppose that $N=(V,E)$ is a phylogenetic network with leaf set $X$.
All vertices within $V\setminus X$ are called \emph{inner} vertices. 
Given an arc $e=(x,y)$ in  $N$, $y$ is the \emph{head} of $e$, denoted by $h_N(e)$, and
$x$ is the \emph{tail} of $e$, denoted by $t_N(e)$.  In this case, we also
say that $x$ is the \emph{parent} $\parent(y)$ of $y$.
In addition, for any DAG $G=(V,E)$ and any vertex $v\in V$ that has a 
unique incoming arc, we denote this arc by \emph{$e^v$}. 

A \emph{directed path} from a vertex $x_1$ to another vertex $x_l$ in $N$ is a non-empty 
sequence $P=(x_1, \dots, x_l)$ of pairwise disjoint vertices such that $(x_i,x_{i+1})\in E$, $1\leq i\leq l-1$. 
We often denote the directed path $P=(x_1, \dots, x_l)$ by $P(x_1,x_l)$ and
also write  $P' = Pv$ (resp.\ $P' = Pe$ with $e = (x_l,v)$) for the directed path $P'=(x_1, \dots, x_l,v)$ that is obtained from 
the directed path $P(x_1,x_l)$ by adding the vertex $v$. 
Moreover, for $x\in V$ and $e\in E$, we define the \emph{directed path from $x$ to  $e$ }
in $N$ as the directed path $P$ from $x$ to the head $h_N(e)$, 
if $P$ exists. 	Thus, the directed path from $x$ to the arc $e= (x,y)$ coincides with $e$. 	

Let $u,v\in V$. Then,  
a vertex $v\in V$ is called a \emph{descendant} of  $u$ (in symbols, $v \preceq_N u$),
if there is a directed path (possibly reduced to a single vertex) 
in $N$ from $u$ to $v$. In this case, we also call $u$ an \emph{ancestor} of $v$, 
denoted by $u \succeq_N v$. 
If $u \preceq_N v$ or $v \preceq_N u$ then $u$ and $v$
are \emph{comparable} and otherwise, \emph{incomparable}.  
Moreover, if $v \preceq_N u$ and $u\neq v$, then 
$v$ is \emph{below} $u$ and $u$ \emph{above} $v$. 
Note, since $N$ is a DAG, an arc $e=(u,v)$ always implies that $u\succ_N v$. 
For a vertex $x\in V$, we write $L_N(x)\coloneqq\{ y\in X \mid y\preceq_N x\}$ for the
set of leaves in $X$ that are below or equal to $x$.

For our discussion below we need to extend the definition of $\preceq_N$ to $V\cup E$. 
More precisely, for the
arc $e=(u,v)\in E$ we put $x \prec_N e$ if  $x\preceq_N v$ and $e
\prec_N x$ if $u\preceq_N x$. In this case, 
the vertex $x$ and the arc $e$ are \emph{comparable}, and \emph{incomparable}
otherwise.   If $e=(u,v)$ and $f=(a,b)$ are arcs in $N$, then
we define $e\preceq_N f$ to hold if $v\prec u\preceq b\prec a$ or $e=f$.
In this case, the arcs $e$ and $f$ are also  \emph{comparable}, and \emph{incomparable}
otherwise. 

We say that a network $N'$ is a \emph{subdivision} of $N$, if $N'$ can be obtained from $N$
by replacing arcs $(u,v)$ of $N$ by directed paths from $u$ to $v$. Hence, a network $N$ is also
a subdivision of itself.  Let $W\subseteq V$.  The subgraph of $N$ with vertex set $W$ that contains all arcs $(x,y)\in E$
for which $x,y\in W$ is called \emph{induced subgraph} of $N$ and is denoted by $N[W]$.

\subsubsection*{Common Ancestors}

For a non-empty subset of leaves $A\subseteq X$ of a
phylogenetic tree $T=(V,E)$ on $X$, we define $\lca_T(A)$, the
\emph{least common ancestor of $A$}, to be the unique $\preceq_T$-minimal vertex
of $T$ that is an ancestor of every vertex in $A$. In case $A=\{x,y \}$, we put
$\lca_T(x,y)\coloneqq\lca_T(\{x,y\})$ and if $A=\{x,y,z \}$, we put
$\lca_T(x,y,z)\coloneqq \lca_T(\{x,y,z\})$. 
If $e,f\in E$ and $x\in V$, then we define $\lca_T(x,e) \coloneqq \lca_T(x,t_T(e))$
and $\lca_T(e,f) \coloneqq \lca_T(t_T(e),t_T(f))$.

Given a tree vertex $z$ in a network $N = (V,E)$, two (not necessarily disjoint) directed paths in $N$
that start from $z$ are said to be \emph{separated (by $z$)} if each path contains a different
child of $z$. Given (not necessarily distinct) $x,y \in V$, we denote 
by $Q_N(x,y)$ the set of vertices $z$ of $N$ such that there exists a directed path $P(z,x)$
and a directed path $P(z,y)$ such that $P(z,x)$ and $P(z,y)$ are separated by $z$. 
Note, $Q_N(x,y) = Q_N(y,x)$. If $z\in Q_N(x,y)$, then we also say that \emph{$x$ and $y$ are separated by $z$}. 
By way of example, we have $Q_N(x,x) = \{a,w\}$ and $Q_N(x,y) =\{b,w\}$ for the network $N$ in 
Fig.\ \ref{fig:MULsimple-foldN-2} (right).
We generalize the latter also to arcs and put 
$Q_N(x,e)\coloneqq Q_N(x,h_N(e))$ and $Q_N(e,f)\coloneqq Q_N(h_N(f),h_N(e))$
\rev{for all arcs} $e,f\in E$.

The definition of $Q_N(x,y)$ was presented in  \cite{To2015}, to generalize
the concept of least common ancestors in binary trees to binary networks.  We now generalize this definition 
to cope with networks that are not necessarily binary. More specifically, given a collection
      $x_1\dots,x_k$ of not necessarily distinct vertices in $N$, we let
     \[Q^2_N(x_1,\dots,x_k)\coloneqq \bigcup_{1\leq i \leq j\leq k} Q_N(x_i,x_j)\]
			be the set of vertices that separate any $x_i$ and $x_j$, $1 \le i \leq j \le k$.

\subsubsection*{Event-labeled Gene Trees}

	An \emph{event-labeled gene tree} $(T;t,\sigma)$ on $\Gen$ is a reduced  tree $T=(V,E)$
	with leaf set $\Gen$ (a set of genes)	a labeling map of events $t\colon V^0\to \{\bul,\squ\}$, 
  and a surjective map
	$\sigma\colon \Gen\to \Spe$, called \emph{gene-species map}, that assigns to each gene $g\in \Gen$
	the species $\sigma(g)$ that contains $g$. 

	Note, the main difference between the structure  
	of an event-labeled gene tree $(T;t,\sigma)$
    and a phylogenetic tree $T'$ as in Definition~\ref{def:network} is that 
		the root of $T$ has at least two children, while the root of $T'$ has exactly one child. 
	The events $\bul$ and $\squ$ are called \emph{speciation} and \emph{duplication}, respectively.
	Moreover, a vertex $v\in V^0$ with $t(v) = \bul$ is called a
        \emph{speciation vertex}, and otherwise, 
    	a \emph{duplication vertex}.
	In all figures that contain event-labeled gene trees, the events $\bul$ and $\squ$
	are represented as $\bullet$ and $\square$, respectively.

	In addition, for a subset $W \subseteq \Gen$, we put $\sigma(W)\coloneqq \{\sigma(w)\mid w\in W\}$.
		
	In what follows, we will always assume that for an event-labeled gene tree $(T;t,\sigma)$ on $\Gen$, 
	$|\sigma(\Gen)|>1$ holds, i.e., the genes in $\Gen$ are  from at least two  	distinct species.

\subsubsection*{MUL-trees}

Informally speaking, 
a \emph{MUL-tree (\emph{MU}ltiply-\emph{L}abelled-tree)} is a  phylogenetic tree $M$  
(cf.\ Definition~\ref{def:network}) where each leaf in $M$ is labeled by an element in $\Spe$, but
where different leaves may have the same label. 
More precisely, a MUL-tree \emph{(on $\Spe$)} is a pair $(M,\chi)$ \rev{where $M$ 
is a phylogenetic tree and} $\chi\colon \Spe \to  2^{L(M)}- \{\emptyset\}$ is a map
such that for all $x,y \in \Spe$ distinct, $\chi(x) \cap \chi(y) = \emptyset$ and for all 
$l \in L(M)$ there exists some $x \in \Spe$  with $l \in \chi(x)$.
Thus, in a MUL-tree $M$ the labeling of the leaf set $L(M)$ of $(M,\chi)$ is the multiset whose underlying set is $\Spe$. 
Note that our definition of a MUL-tree
is equivalent to the definition given in \cite[Section 2.2]{huber2016folding}, 
except that in the MUL-tree defined here we have an ``extra'' arc which is adjacent to the root (so we get a root with 
outdegree 1). As with networks, this extra arc is required to accommodate events in reconciliations which
occur before the first speciation.

If we allow $M$ to contain vertices with 
indegree 1 and outdegree 1 as well, \rev{then} we call $M$ a \emph{pseudo MUL-tree} (so, 
in this setting, any MUL-tree is also considered as a pseudo MUL-tree).
For a pseudo MUL-tree $M=(D,U)$, we also define 
\[D^1 = \{v\in D \mid v \text{ has in- and outdegree one in } M \}. \]
We say that two (pseudo) MUL-trees $(M_1,\chi_1)$ and $(M_2,\chi_2)$ on
$X$ are {\em isomorphic} if there is a digraph isomorphism
$\varphi: V(M_1) \to V(M_2)$ such that, for all $x\in X$
and $v \in V (M_1 )$, we have $v\in \chi_1(x)$ if
and only if $\varphi(v)\in \chi_2 (x)$.
In addition,we say that 
a pseudo-MUL-tree $(M',\chi)$ is a \emph{simple subdivision} of 
the MUL-tree $(M=(D,U),\chi)$  
if $M'$ is a subdivision of $M$ that is obtained from $M$ by replacing 
every arc $\rev{e=}(u,v)$ of $M$, where $v$ is a leaf of $M$ such that 
$v\in \chi(x)$ with	$x \in \Spe$  and $|\chi(x)|\geq 2$  
by the path $(u,v_e,v)$ in $M'$, where $v_e\notin D$. 
The latter, implies in particular  that $L(M)=L(M')$ and hence, 
  the map $\chi$ for $M'$ is well-defined.

\section{Reconciliation Maps to Networks}
\label{sec:rec-network}

In \cite{To2015}, a notion of a reconciliation map (which we shall call
a \emph{biTreeNet-reconciliation map}) from
a phylogenetic tree to a network was presented (see Definition~\ref{def:alpha-map}).
This map, however, 
assumes that the event labels on the tree are unknown, 
while a species network is given. Moreover, the \emph{biTreeNet-reconciliation map} ``axioms''
explicitly refer to binary gene trees and binary species networks. 
In practice, when gene trees are obtained from e.g.\ orthology-data, we cannot hope to obtain fully 
resolved (i.e., binary) gene trees. To overcome this problem
and therefore also be able to study reconciliations between non-binary gene trees 
and (possibly unknown) non-binary networks we next introduce the novel concept of a
\emph{TreeNet-reconciliation map}. That TreeNet-reconciliation maps indeed generalize
biTreeNet-reconciliation maps is shown in Proposition~\ref{prop:a-Implies-mu}. 
Furthermore, TreeNet-reconciliation maps provide a natural generalization of the
framework as used in
\cite{HHH+12,Hellmuth2017,Nojgaard2018} for the reconciliation of trees,
which is shown in Section \ref{sec:rec-tree}.

\begin{defi}[TreeNet-reconciliation map] \label{def:mu} 
  Suppose that $\Spe$ is a set of species, that $N=(W,F)$ is a  network on
  $\Spe$, and that $(T=(V,E);t,\sigma)$ is an event-labeled gene tree on $\Gen$. 
	Then we say that \emph{$N$ is a (species) network for
  $(T;t,\sigma)$} if there is a map $\mu\colon V\to W\cup F$ such that, for all $x\in V$:  
\begin{description}
\item[(R1)] \emph{Leaf Constraint.}  If $x\in \Gen$   then $\mu(x)=\sigma(x)$. 		\vspace{0.03in}
\item[(R2)] \emph{Event Constraint.}
	\begin{itemize}
		\item[(i)]  If $t(x)=\bul$ and $x$ has children $x_1,\dots,x_k$, $k\geq 2$, 
							then 

	 $\mu(x) \in Q^2_N(\mu(x_1),\dots,\mu(x_k))$.
		\item[(ii)] If $t(x) = \squ$, then $\mu(x)\in F$. 
	 	\end{itemize} \vspace{0.03in}
\item[(R3)] \emph{Ancestor Constraint.}		\\	
		Suppose $x,y\in V$ with $x\prec_{T} y$.
	\begin{itemize}
		\item[(i)] If $t(x) = t(y) = \squ$, then $\mu(x)\preceq_N \mu(y)$, 
		\item[(ii)] Otherwise, i.e., at least one of $t(x)$ and $t(y)$ is a speciation,  
						$\mu(x)\prec_N\mu(y)$.
	\end{itemize}
\end{description}
We call $\mu$ \rev{a} \emph{TreeNet-reconciliation map} from $(T;t,\sigma)$ to $N$. 
\end{defi}                

Property (R1) ensures that each leaf of $T$, i.e., each
gene in $\Gen$, is mapped to the species in which
it resides. Property (R2.i) implies the weaker property that $\mu(x)\in W^0$. 
Moreover, Property (R2.i) ensures that each speciation 
vertex $x$ of $T$ is mapped to a vertex in $N$ that separates
at least two of the images $\mu(x_i)$ and $\mu(x_j)$
of the children $x_1,\dots,x_k$\rev{, $k\geq 2$,}  of $x$. 
We emphasize that, even in the case 
$\mu(x_1) = \mu(x_2) = \cdots =\mu(x_k)$, 
 it is possible that 
$Q^2_N(\mu(x_1),\dots,\mu(x_k)) \neq \emptyset$, since $N$ may contain
distinct directed paths from $\mu(x)$ to $\mu(x_1)$ that  separate $\mu(x_1)$.
Property (R2.ii) ensures that each duplication vertex of $T$ is mapped to
an arc in $N$. Property (R3) implies that the ancestor relationships
in $T$ are preserved under $\mu$. Note, however, that  
two different duplication vertices might be mapped to the same arc of $N$
under $\mu$. 

We now show that  for the special case that both the gene tree and the species network are
binary, Definition \ref{def:mu} is a natural generalization of 
biTreeNet-reconciliation maps as defined in \cite{To2015}.
We begin by recalling the definition of this map (using our notation).
Note, in this definition the symbols
$\bul$ and $\squ$ still denote speciation and duplication events, respectively. 
For technical reasons, an additional symbol $\mathfrak{c}$ is used to annotate leaves.  

\begin{defi}[biTreeNet-reconciliation map \cite{To2015}] 
Let $\sigma\colon \Gen \to \Spe$ be a gene-species map.   
A \em{biTreeNet-reconciliation map} $\alpha=(\alpha_1,\alpha_{2})$ from a binary gene tree $T=(V,E)$ on $\Gen$ to a binary network $N=(W,F)$ on $\Spe$
is a pair of maps $\alpha_{1}\colon V\to W$ and $\alpha_{2}\colon V \to \{\bul,\squ,\mathfrak{c}\}$
that assigns to each vertex $u\in V$ a pair $(\alpha_{1}(u),\alpha_{2}(u))$ such that 
\begin{description}
	\item[(A1)]  $\alpha_{2}(u) = \mathfrak{c}$ if and only if $u\in \Gen$,  $\alpha_{1}(u)\in L(N)$, $\sigma(u)=\alpha_{1}(u)$, 
	\item[(A2)]  for every $u \in V^0$ with child vertices $u_1$ and $u_2$, 

							\hspace{0.3cm}	if $\alpha_{2}(u)=\bul$, then
							$\alpha_1(u) \in Q_N(\alpha_1(u_1),\alpha_1(u_2))$, and 
	\item[(A3)] \rev{for all} $u, v \in V$ such that $v\prec_T u$, 

						\hspace{0.3cm} if  $\alpha_{2}(u)=\squ$, then
							$\alpha_1(v)\preceq_N \alpha_1(u)$. Otherwise,  $\alpha_1(v)\prec_N \alpha_1(u)$.
\end{description}
\label{def:alpha-map}
\end{defi}

In Definition \ref{def:alpha-map}, $\alpha_{2}$ plays the role of our event-labeling
\rev{map}.
That is,  if $\alpha_{2}$ is given, then 
putting $t^{\alpha}(u)\coloneqq \alpha_2(u)$ for all $u\in V^0$ 
yields an event-labeled gene tree $(T;t^{\alpha}, \sigma)$ on $\Gen$. 

\begin{proposition}
  If there is a  biTreeNet-reconciliation map
	$\alpha$
  from a binary gene tree 
	$T$ on $\Gen$ to a binary network $N$ on $\Spe$, then there is 
	a TreeNet-reconciliation  map 
	from  $(T;t^{\alpha}, \sigma)$ to $N$.
	\label{prop:a-Implies-mu}
\end{proposition}
\begin{proof}
	In what follows, we show that there is a TreeNet-reconciliation map $\mu$
	from  $(T;t^{\alpha}, \sigma)$ to $N$, where $t^{\alpha}$
	is constructed as described above. 
	
	Let \rev{$\alpha=(\alpha_1,\alpha_2)$} be a biTreeNet-reconciliation
        map from $T=(V,E)$ to $N=(W,F)$. 
	We claim that $\mu\colon V\to W \cup F$ given by
	\begin{equation*}
	\mu(u) = \begin{cases}
	\alpha_{1}(u) &\mbox{ if } t^{\alpha}(u)\in \{\bul,\mathfrak{c}\}, \\
e^{\alpha_1(u)}
        \in F	&\mbox{ if } t^{\alpha}(u) = \squ \mbox{ and } \alpha_{1}(u) \mbox{ is not a hybrid vertex,} \\
	(\alpha_{1}(u),x) \in F	&\mbox{ otherwise, \rev{some $x\in W$}.} 
	\end{cases}
	\end{equation*}
	is a TreeNet-reconciliation map
	from  $(T;t^{\alpha}, \sigma)$ to $N$. Note, in the second condition $e^{\alpha_1(u)}$ is well-defined, 
	since $\alpha_{1}(u)$ must, in this case, be either  the root $\rho_N$
		or a tree vertex in $N$ and thus, there is only one incoming arc to $\alpha_{1}(u)$. 
	Moreover, in the last condition, the vertex $x$ is uniquely defined, since $\alpha_{1}(u)$ is then 
	a hybrid vertex and thus, there is only arc $(\alpha_{1}(u),x)$ in $N$.

	Clearly, Property~(A1) and the construction of $\mu$ immediately implies Property~(R1).
	
	We show now that \rev{Property~}(R3) is satisfied. 
	Let $u, v \in V$ such that $v\prec_T u$. 
	First assume that $t^{\alpha}(u)=\squ$. 
	There are now two mutually exclusive cases: either $\alpha_1(u)$ is 
	not a hybrid vertex and $\mu(u)$ is the arc $e^{\alpha_1(u)}$
	or $\alpha_1(u)$ is a
	hybrid vertex and $\mu(u)$ is the arc $(\alpha_{1}(u),x)$\rev{, some $x\in W$}. 
	If $t^{\alpha}(v)=\squ$, then, in both cases, Property (A3) and a straight-forward
        case analysis implies that  $\mu(v)\preceq_N \mu(u)$. 
	If $t^{\alpha}(v)=\bul$, then Property (A2) implies 
	that $\mu(v)$ must be a tree vertex. By construction, if $\alpha_1(v) = \alpha_1(u)$,
	then $\mu(u)$ is mapped to the unique arc $e^{\mu(v)}$
        and we obtain $\mu(v)\prec_N \mu(u)$. 
	In particular, $\alpha_1(v) \preceq_N \alpha_1(u)$ implies that 
	the arc $e^{\mu(v)}$
        is the ``lowest'' possible choice for $\mu(u)$ and thus, 
	we always have $\mu(v)\prec_N \mu(u)$, in case $t^{\alpha}(v)=\bul$. 
	
	Now assume that $t^{\alpha}(u)=\bul$.
	Then,  $\mu(u)  = \alpha_{1}(u)$. In particular, Property~(A1) implies $\mu(u)\in V^0$.
	Property~(A2) and the construction of $\mu$ implies 
	that $\mu(u)$ must be a tree vertex. 
	If $t^{\alpha}(v)=\bul$ or $v\in \Gen$, then $\alpha_1$ and $\mu$ coincide
	on $u$, resp., on $v$. Hence, Property~(A3) implies
	$\mu(v)\prec_N \mu(u)$. Thus, assume that $t^{\alpha}(v)=\squ$.
	The fact that	$\mu(u)$ is not a hybrid vertex together with 
	Property~(A3) and the construction of $\mu$ implies that 
	$\mu(v)= (x,y) \prec_N x \preceq_N \alpha_{1}(u) = \mu(v)$. 
	In summary, Property~(R3) is satisfied.  
	
	We continue by showing that Property~(R2) is satisfied. 
	By construction of $\mu$, Property (R2.ii) is satisfied. 
	For Property~(R2.i), let $u$ be a vertex in $T$ with $t^{\alpha}(u)=\bul$. Thus, 
	$\mu(u) = \alpha_{1}(u)$. 
	Since $T$ is binary, $u$ has exactly two children $u_1$ and $u_2$. 
	Property~(A2) implies  $\mu(u) = \alpha_{1}(u) \in Q_N(\alpha_1(u_1),\alpha_1(u_2))$. 	
	Thus, $\mu(u)$ must be a tree vertex. 
	In particular, there are directed paths $P_1$ and $P_2$ that start in $\mu(u)$
	to $\alpha_1(u_1)$ and $\alpha_1(u_2)$, respectively, where
	$P_1$ contains a child $z_1$ of $\mu(u)$ and $P_2$ a child $z_2$ of $\mu(u)$ that is distinct from $z_1$. 
	Property~(A3) implies $\alpha_1(u_1) \preceq_N z_1 \prec_N\alpha_1(u)$
	and $\alpha_1(u_2) \preceq_N z_2 \prec_N\alpha_1(u)$.
	The construction of $\mu$	implies that $\mu(u_1)$ is always 
	located on the path $P_2$ or the path $P_1x$ for some child $x$ of $\alpha_1(u_1)$ (if  $\alpha_1(u_1)$ is not a leaf) and 
	$\mu(u_2)$ is always 	located on the path $P_1$ or the path $P_2y$ for some child $y$ of $\alpha_1(u_2)$ (if  $\alpha_1(u_2)$ is not a leaf).
	Moreover, Property~(R3) implies 
	$\mu(u_1),\mu(u_2)\prec_N \mu(u)$. 
	Taken the latter two arguments together, we obtain $\mu(u) \in Q_N(\mu(u_1),\mu(u_2))$. 
	
	Thus, $\mu$ is a TreeNet-reconciliation  map \rev{from $(T;t^{\alpha}, \sigma)$ to $N$}.
\qed \end{proof}

	Note that Proposition~\ref{prop:a-Implies-mu} implies that the definition
of a TreeNet-reconciliation map applied to binary gene trees and binary
networks is a natural generalization of the concept of a
biTreeNet-reconciliation map.

\section{Reconciliation Maps to Trees}
\label{sec:rec-tree}

In this section, we consider the special case of reconciliation maps where the
network is a tree. As described in the introduction, the reconciliation of
gene trees (with or without
event-labels) with species trees has been the subject of numerous studies.
In \cite{Hellmuth2017} (tree) reconciliation maps between not necessarily binary
event-labeled gene trees and species trees were studied. In particular, a
\emph{tree reconciliation map} $\mu\colon V\to W\cup F$ from a given event-labeled gene tree
$(T=(V,E);t,\sigma)$ to a species tree $S = (W,F)$ is defined \rev{there}
which, in our
notation, is equivalent to replacing Property (R2.i) in Definition \ref{def:mu} by 
\begin{align*}
\textbf{(R2.i$^*$)} &  \text{ If } x\in V \text{ and } t(x)=\bul, \text{ then }  \mu(x) = \lca_S(\sigma(L_T(x))),
\end{align*}
and by adding the constraint 
\begin{align*}
\textbf{(R2.iii)} &  \text{ If } x\in V^0 \text{ and } t(x)=\bul, \text{ then } 
\mu(y_1) \text{ and } \mu(y_2) \text{ are incomparable in } \rev{S} \\
&		\text{ for any two distinct children } y_1 \text{  and  } y_2 \text{ of }x. 
\end{align*}

If a map $\mu$ from a given event-labeled gene tree $(T;t,\sigma)$ 
to a species tree $S$ satisfies
Properties (R1), (R2.i.$^*$), (R2.ii) and (R3) but not necessarily (R2.iii), then we shall call $\mu$ a 
\emph{relaxed tree reconciliation map} (thus, any tree reconciliation map is also relaxed). 
\rev{An example of a relaxed tree reconciliation map that is not a tree reconciliation map
is given in Fig.\ \ref{fig:rel-rec}.}
In what follows, we show for the case that the species network is in fact a
species tree that the concepts of a  
TreeNet-reconciliation map and a  relaxed tree reconciliation maps are equivalent
(Theorem~\ref{thm:tree2tree-reconc}). 

We first present some lemmas for later use.

\begin{lemma}
	\rev{Given any} phylogenetic tree $T=(V,E)$ and any two vertices $v,w\in V$ we have either 
	$Q_T(v,w)=\emptyset$ or	$Q_T(v,w)=\{\lca_T(v,w)\}$. 
	
	In particular, $Q_T(v,w)=\emptyset$ if and only if  $v$ and $w$ are comparable in $T$
\label{lem:QS}
\end{lemma}
\begin{proof}
	Clearly, if $v=w$, then $Q_T(v,w)=\emptyset$, since $T$ is a tree. Thus, let $v\neq w$.
	Note, any two vertices $v,w\in V$ are either comparable in $T$ or not.

	Let us first assume that $v$ and $w$ are comparable.  W.l.o.g.\ let $v\succ_T w$. 
  Thus, there is a unique path from $\lca_T(v,w)=v$ to $w$ in $T$. 	
	Therefore, $v$ and $w$ cannot be separated by any vertex in $V$. Hence, 
	$Q_T(v,w)=\emptyset$. 

	If $v$ and $w$ are incomparable, then $\lca_T(v,w) \not\in\{ v,w\}$. Hence, 
	there is a unique directed path from $\lca_T(v,w)$ to $v$ and 
	a unique path from $\lca_T(v,w)$ to $w$ that have only the vertex 
	$\lca_T(v,w)$ in common. Thus, $\lca_T(v,w)\in Q_T(v,w)$. 
	Since $\lca_T(v,w)$ is the only vertex in $T$ that separates
	$v$ and $w$, we have $Q_T(v,w)=\{\lca_T(v,w)\}$.
\qed \end{proof}

\begin{lemma}
	Let $(T=(V,E);t,\sigma)$  be an event-labeled gene tree, let $S=(W,F)$ be a species tree  on $\Spe$, 
	and let $x\in V$ with $t(x)=\bul$ and children $x_1,\dots,x_k$, $k\geq 2$, in $T$. 
	Suppose $\mu$ is a TreeNet-reconciliation map from $(T;t,\sigma)$  to 
	$S$. Then, $\mu(x) \succeq_S q$	for all $q\in Q^2_S(\mu(x_1),\dots,\mu(x_k))$, i.e., 
	$\mu(x)$ is the $\succeq_S$-maximal element in $Q^2_S(\mu(x_1),\dots,\mu(x_k))$
	\label{lem:Q2S}
\end{lemma}
\begin{proof}
	For simplicity, for every $x_i\in V$
	let $v_i$ be either the vertex $\mu(x_i)$, if $\mu(x_i)\in W$
	or the vertex $h_S(e)$ if $\mu(x_i) = e\in F$, $1\leq i \leq k$.
  Moreover, put $Q^2_S \coloneqq Q^2_S(\mu(x_1),\dots,\mu(x_k))$.

 	We show first that there is always a vertex $q^* \in Q^2_S$
	with $q^*\succeq_S q$ for all $q\in Q^2_S$. 
	Property (R2.i) implies that $Q^2_S\neq \emptyset$. 
	If $|Q^2_S|=1$, then the lemma trivially holds. 
	Now, assume that there are distinct $q,q'\in Q^2_S$ that are incomparable.
	By Lemma \ref{lem:QS}, $q=\lca_S(v_i,v_j)$ and $q'=\lca_S(v_r,v_s)$
	where $v_i$ and $v_j$ as well as $v_r$ and $v_s$ are incomparable in $S$
	and $i,j,r,s\in \{1,\dots,k\}$. Since $S$ is a tree, we have 
	$\lca_S(q,q') = \lca_S(v_i,v_r)$. Thus, $\lca_S(q,q')$ separates
		$v_i$ and $v_r$ and therefore, $q'' \coloneqq \lca_S(q,q') \in Q^2_S$. 
	In other words, for any two  vertices  $q,q'\in Q^2_S$ 
	there is always a vertex $q''\in Q^2_S$ with  $q''\succeq_S q,q'$.
	The latter arguments imply that there is a unique $\succeq_S$-maximal element $q^*$ in $Q^2_S$
	that is an ancestor of all elements in $Q^2_S$, as required.

	Property (R2.i) implies that $\mu(x)\in Q^2_S$. 
	By Lemma \ref{lem:QS} and since $S$ is a tree, 
	there are two distinct children $x_r$ and $x_s$ of $x$ such that 
	$\mu(x)\in Q_S(\mu(x_r),\mu(x_s))$ and $\mu(x) = \lca_S(v_r,v_s)$. 
	It remains to show that $q^* = \mu(x)$.
  Assume for contradiction that $q^*\neq \mu(x)$ and thus, $q^*\succ_S \mu(x)$. 
  Again by Lemma \ref{lem:QS}, $q^* = \lca_S(v_i,v_j)$ for some $1\leq i < j \leq k$. 
	Since $q^*$ separates $v_i$ and $v_j$ and  $q^*\succ_S \mu(x)$  
	at least one of $v_i$ and $v_j$, say $v_i$, must be incomparable
	with $\mu(x)$. Since $x_i$ is a child of $x$ in $T$ we can apply 
	Property (R3) to conclude that $\mu(x)\succ_S \mu(x_i)\succeq_S v_i$. 
	Hence, $v_i$ and $\mu(x)$ must be comparable in $S$, a contradiction. 
\qed \end{proof}

\begin{lemma}
	Let  $(T=(V,E);t,\sigma)$ be an event-labeled gene tree
\rev{such that there 
  exists a relaxed tree reconciliation map $\mu$
	from $(T;t,\sigma)$  to some species tree $S$.}
	Let $x\in V$ with $t(x) = \bul$. Then, there exists two distinct children
	$x_i$ and $x_j$ of $x$ in $T$ that satisfy the following two properties, 
	\begin{itemize}
		\item[(i)] $\mu(x)=\lca_S(\mu(x_i),\mu(x_j))$, and
		\item[(ii)] $\mu(x_i)$ and $\mu(x_j)$  are incomparable in $S$. 
	\end{itemize}
	Furthermore, for any child 	$x_i$ of $x$ in $T$ there is a child $x_j$ of $x$ in $T$ such that 
	\begin{itemize}
		\item[(iii)] 								$\sigma(L_T(x_i))\cap \sigma(L_T(x_j)) = \emptyset$.
	\end{itemize}	
\label{lem:tree2tree-reconc}
\end{lemma}
\begin{proof}
  Let $x\in V$ with $t(x) = \bul$ and children $x_1,\dots,x_k$, $k\geq 2$, in $T$
  \rev{and put $S=(W,F)$}. 

	We combine the proof of Statements~(i) and (ii). 
	Property (R2.i$^*$) implies that $\mu(x) = \lca_S(\sigma(L_T(x)))$. 
	Clearly,  $\mu(x)$ can be expressed as $\mu(x) = \lca_S(z_1,z_2)$
	for some leaves $z_1,z_2\in \sigma(L_T(x))$. 
	Since $\sigma(L_T(x)) = \bigcup_{i=1}^k\sigma(L_T(x_i))$, 
	there exist children of $x$ in $T$, say w.l.o.g.\ $x_1$ and $x_2$, 
	with $z_1\in \sigma(L_T(x_1))$ and $z_2\in \sigma(L_T(x_2))$. 
	Note that $x_1$ and $x_2$ must be distinct since  $x_1=x_2$ and Property~(R3) would imply that 
	$\mu(x)\succ_S\mu(x_1)\succeq_S \lca_S(z_1,z_2)$ and 
	thus, $\mu(x)\neq \lca_S(z_1,z_2)$; a contradiction. 
	The latter together with 
	$\mu(x) = \lca_S(z_1,z_2)$;  $\mu(x) \succ_S \mu(x_1)\succeq_S z_1$;
	and  $\mu(x) \succ_S \mu(x_2)\succeq_S z_2$
	implies that $\mu(x_1)$ and $\mu(x_2)$ 	must be incomparable 
	in $S$ and $\lca_S(\mu(x_1), \mu(x_2)) = \mu(x)$. This establishes 
	Statements~(i) and (ii).

	It remains to show that Statement~(iii) holds. Assume for contradiction
	that there is a child of $x$, say $x_1$, such that 
	$\sigma(L_T(x_1))\cap \sigma(L_T(x_j)) \neq \emptyset$, $2\leq j\leq k$.
	By Property (R3), we have $\mu(x)\succ_S\mu(x_1)$.
	
	Let $e$ be the first arc on the (unique) directed path from $\mu(x)$ to $\mu(x_1)$
  in $S$.   Since $\mu(x)\succ_S\mu(x_1)$ and 
  Property (R2) \rev{holds}, we have 
	$e\succeq_S \mu(x_1)$. By Property (R3), we have $ e\succeq_S\mu(x_1)\succeq_S z$
	for all $z\in   \sigma(L_T(x_1))$, and therefore, 
	$y:=h_S(e)\succeq_S z$ for all $z\in   \sigma(L_T(x_1))$.
	In other words, any path from $\mu(x)$ to $z$ must contain
	the arc $e$ for all $z\in   \sigma(L_T(x_1))$

	By assumption, there is a leaf $z \in \sigma(L_T(x_1))\cap \sigma(L_T(x_2))$.
	By Property (R3) and the latter arguments, $\mu(x)\succ_S e\succeq_S\mu(x_2)\succeq_S z$ 
	Again, Property~(R3) implies that  $\mu(x_2)\succeq_S z'$
	for all $z'\in   \sigma(L_T(x_2))$, and therefore, 
	$y\succeq_S z'$ for all $z'\in   \sigma(L_T(x_2))$.
	Hence, any path from $\mu(x)$ to $z'$ must contain
	the arc $e$ for all $z'\in   \sigma(L_T(x_2))$.
	Repeating the latter arguments shows that  
	any path from $\mu(x)$ to $z'$ must contain
	the arc $e$ for all $z'\in   \sigma(L_T(x_i))$ and  $i\in\{1,\dots,k\}$. 
	Thus, $y\succeq_S z'$ for all $z'\in \bigcup_{i=1}^k \sigma(L_T(x_i)) = \sigma(L_T(x))$. 
	Hence, $\mu(x) \succ_S y \succeq_S \lca_S(\sigma(L_T(x)))$;
	a contradiction to Property~(R2.i$^*$). 
\qed \end{proof}

We are now in the position to prove the \rev{aforementioned} equivalence between 
TreeNet-reconciliation and relaxed tree reconciliation maps.

\begin{theorem}
  Let  $(T=(V,E);t,\sigma)$ be an event-labeled gene tree,
  \rev{$S=(W,F)$ be a species tree, and $\mu:V\to W\cup F$ be a map.}
  Then, $\mu$ is a relaxed tree reconciliation map from $(T;t,\sigma)$  to
  $S$
	if and only if $\mu$ is a TreeNet-reconciliation map
	from $(T;t,\sigma)$  to
        $S$.
	\label{thm:tree2tree-reconc}
\end{theorem}
\begin{proof}
  Note first that	all Properties (R1), (R2.ii) and (R3) in Definition~\ref{def:mu} 
	remain the same for both, TreeNet-reconciliation maps and 
	relaxed tree reconciliation maps. Hence, 
	it suffices to show that if $\mu$ is a TreeNet-reconciliation map
  then Property~(R2.i$^*$) holds and that if $\mu$ is a relaxed tree reconciliation
  then Property~(R2.i) \rev{holds}.

  Assume first that $\mu$  is a relaxed tree reconciliation map
  from $(T;t,\sigma)$ to $S$. Let $x\in V$ be a vertex with $t(x) = \bul$
  and children $x_1,\dots,x_k$\rev{, $k\geq 2$,} in $T$. Then
  Lemma~\ref{lem:tree2tree-reconc}(i) implies that 
	$\mu(x) =\lca_S(\mu(x_i), \mu(x_j))$, for some $i,j\in \{1,\dots,k\}$.
	By Lemma~\ref{lem:QS}, $Q_S(\mu(x_i),\mu(x_j)) = \{\mu(x)\}$ follows.
	Hence, $\mu(x)\in Q^2_S(\mu(x_1),\dots,\mu(x_k))$. 
	Thus, Property~(R2.i) holds.

	Now assume that $\mu$ is TreeNet-reconciliation map
	from $(T;t,\sigma)$  to $S$.
	In what follows we will make use of the following observation:
	If $y_1$ and $y_2$ are incomparable in $S$ and $y=\lca_S(y_1,y_2)$
        \rev{and $u,v,w\in W$ are such that}
        $y\succeq_S u$, $y_1\succeq_S v$ and $y_2\succeq_S w$, 
	then 
	\begin{equation}
		y = \lca_S(v,w) = \lca_S(u,v,w) = \lca_S(y_1,y_2,w). \label{eq:master}
	\end{equation}
	Property~(R2.i) and Lemma \ref{lem:QS} imply that
	$\mu(x) = \lca_S(\mu(x_i), \mu(x_j)) \in Q^2_S(\mu(x_1),\dots,\mu(x_k))$
	for some distinct $i,j\in\{1,\dots,k\}$. 
	W.l.o.g., let $\mu(x) = \lca_S(\mu(x_1), \mu(x_2)) \in Q^2_S(\mu(x_1),\dots,\mu(x_k))$. 
  Property~(R3) 	and 	Lemma \ref{lem:Q2S} imply that 
	$\mu(x)\succeq_S \lca_S(\mu(x_r), \mu(x_s))$	for all $r,s\in\{1,\dots,k\}$. 
	Repeated application of Equ.\ \eqref{eq:master} implies that 
	\begin{align*}
		\mu(x) = \lca_S(\mu(x_1), \mu(x_2), \lca_S(\mu(x_r), \mu(x_s))) 
					=\lca_S(\mu(x_1), \mu(x_2), \mu(x_r), \mu(x_s))
	\end{align*}
	for all $r,s\in\{1,\dots,k\}$ and, again by repeated application of Equ.\ \eqref{eq:master},  we obtain 
	\[
		\mu(x) = \lca_S(\mu(x_1),\dots, \mu(x_k)). 	
	\]

	Now, put	$v_i\coloneqq \lca_S(\sigma(L_T(x_i)))$, $1\leq i\leq k$. 
	Property~(R3) implies that  $\mu(x_r)\succeq_S z$ for all $z\in \sigma(L_T(x_r))$ and $1\leq r\leq k$. 
	Thus, $\mu(x_1) \succeq_S  v_1$ and 
	$\mu(x_2) \succeq_S v_2$. 	
  Since $\mu(x_1)$ and $\mu(x_2)$ are separated by $\mu(x)$, they 
	must be incomparable. Equ.\ \eqref{eq:master} implies that 
	\[
		\mu(x) = \lca_S(v_1,v_2). 
	\]
  Again, by   Property (R3), we have $\mu(x)\succ_S \mu(x_r)$ and 
	$\mu(x_r) \succeq_S \lca_S(\sigma(L_T(x_r)))$, $1\leq r\leq k$. 
	Hence, Equ.\ \eqref{eq:master} implies that 
	 \[\mu(x) = \lca_S(v_1,v_2,v_r) \text{ for all }  r\in\{1,\dots,k\}.\]
	 Repeated application of Equ.\ \eqref{eq:master} implies
	\begin{equation}
		\mu(x) = \lca_S(v_1,v_2,v_3,v_4,\dots, v_k). \label{eq:final}
	\end{equation}
	To make the final step, we first observe that 
	for non-empty vertex sets $A, B$ of a tree, $\lca(A\cup B)=\lca(\lca(A),\lca(B))$ holds.
  The latter and $\sigma(L_T(x)) = \bigcup_{i=1}^k\sigma(L_T(x_i))$ allows us to rewrite
	Equ.\ \eqref{eq:final} as	$\mu(x)  = \lca_S(\sigma(L_T(x)))$.
	Therefore, the map $\mu$ satisfies Property~(R2.i$^*$). 
\qed \end{proof}

 As an immediate consequence of Theorem~\ref{thm:tree2tree-reconc} we obtain

 \begin{corollary}
   \rev{ Suppose $(T(V,E);t,\sigma)$ is an event-labeled gene tree,  $S=(W,F)$
     is a species tree and $\mu:V\to W\cup F$ is a map. Then}
   $\mu$ is a  tree reconciliation map  from
   $(T;t,\sigma)$  to a
   $S$
	if and only if $\mu$ is a TreeNet-reconciliation map from $(T;t,\sigma)$  to $S$
	that additionally satisfies Property~(R2.iii). 
	\label{cor:mu1}
\end{corollary}

As a consequence of Lemma \ref{lem:tree2tree-reconc}(ii), 
if $\mu$ is a relaxed tree reconciliation map between a binary 
gene tree and a (not necessarily binary) species tree, 
then Property (R2.iii) is always satisfied.
Thus any  relaxed tree reconciliation map between a binary 
gene tree and a species tree is also a tree reconciliation map. 
Therefore, Theorem \ref{thm:tree2tree-reconc} implies

\begin{corollary}
  Suppose $(T;t,\sigma)$ is a \emph{binary} \rev{event-labeled}
  gene tree, $S=\rev{(W,F)}$ \rev{is}
	a species tree \rev{ and $\mu:V\to W\cup F$ is a map}. 
	Then, $\mu$ is a  tree reconciliation map from  $(T;t,\sigma)$  to  $S$
	if and only if $\mu$ is a TreeNet-reconciliation map from $(T;t,\sigma)$  to $S$.
\label{cor:bin}
\end{corollary}

\section{Informative Triples}
\label{sec:triples}

It is of interest to understand when a species tree or species network 
exists for a given gene tree  $(T;t,\sigma)$, and
if it does exist, which constraints (if any) such an network imposes on 
$(T;t,\sigma)$ for different kinds of networks. In this section,
we shall consider constraints that are given in terms of triples. 
We start with briefly recalling the terminology surrounding them.

A \emph{triple $ab|c$}  is a binary reduced tree $T$ on three leaves $a,b$ and $c$ such that
the path from $a$ to $b$ does not intersect the path from $c$ to the root $\rho_T$.
A network $N$ on $X$ \emph{displays} a triple $ab|c$, if $a,b,c\in X$ and $ab|c$ 
can be obtained from
$N$ by  deleting  arcs  and  vertices,  and  suppressing vertices with in- and outdegree one.
Note, that no distinction is made between
$ab|c$ and $ba|c$. 
We write $R(N)$ for the set of all triples that are displayed by the network $N$. 
A set $R$ of triples is  \emph{compatible} if there is a phylogenetic tree $T$ on $X= \bigcup_{t\in R} L(t)$ such that
$R\subseteq R(T)$.

\begin{figure}[tbp]
  \begin{center}
    \includegraphics[width=.9\textwidth]{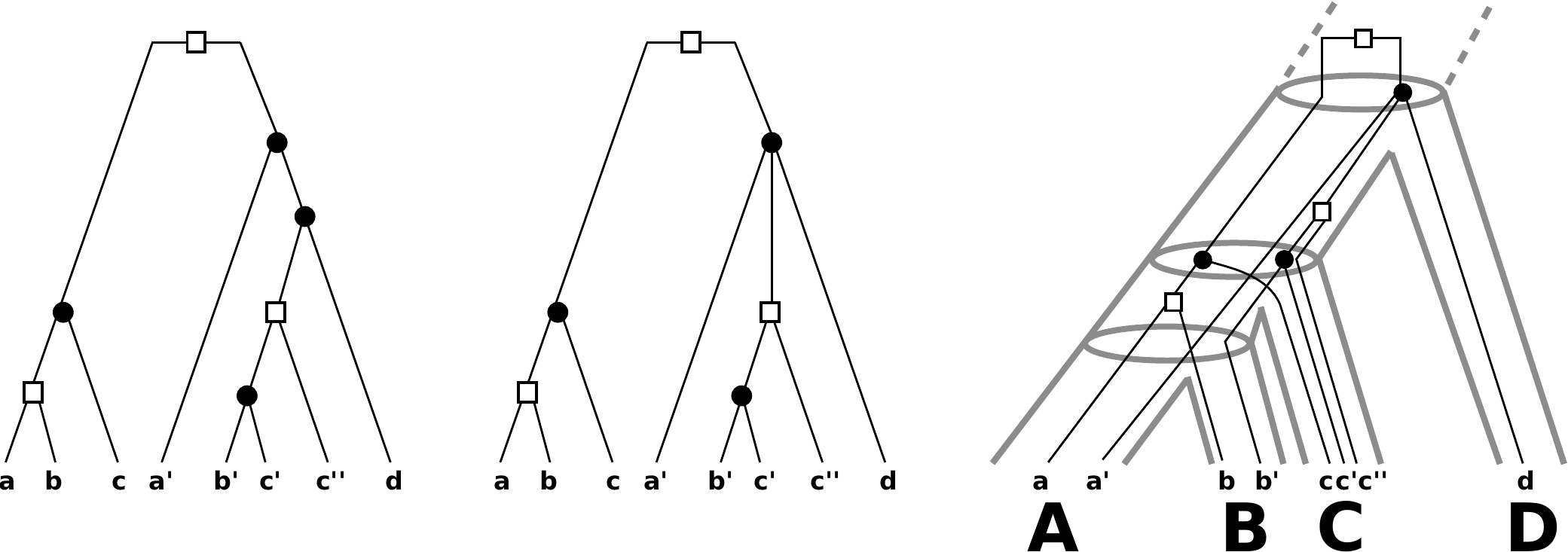}
  \end{center}
	\caption{		(\emph{Adapted from \cite[Fig.\ 3]{Hellmuth2017}})
	Two event-labeled gene trees  $(T';t',\sigma)$ (left) and  $(T;t,\sigma)$ (middle) on a 
	set $\Gen = \{a,a',b,b',c,c',c'',d\}$. 	\rev{Speciation and duplication events 
	are represented as $\bullet$ and $\square$, respectively. } 
The relaxed tree reconciliation map from $(T;t,\sigma)$
	to the species tree $S$  on $\Spe=\{A,B,C,D\}$  is 
	implicitly shown	by drawing the gene tree inside the tube-like species tree $S$ (right).
	 Here, 	$\sigma$  maps each gene in $\Gen$ to the species 
	(capitals below the genes) $A,B,C,D\in \Spe$. 
	Although there is a relaxed tree reconciliation map from the non-binary gene tree $(T;t,\sigma)$ to $S$, 
	there exists no tree reconciliation map from $(T;t,\sigma)$ to any species tree, 
  since $AB|C, BC|A\in \mc{S}(T;t,\sigma)$ implies that $\mc{S}(T;t,\sigma)$ is incompatible (cf.\ Thm.\ \ref{thm:speciestriples2}). 
	Moreover, there exists neither a (relaxed) tree reconciliation map nor a TreeNet-reconciliation map
  from the binary gene tree $(T';t',\sigma)$ to any species tree, since the set $\mc{S}(T';t',\sigma)$ contains the 
 	triples $AB|C$ and $BC|A$ and is, therefore, incompatible (cf.\ Thm.\ \ref{thm:speciestriples}).
        } 
	\label{fig:rel-rec}
\end{figure}

Interestingly, the existence of tree reconciliation maps from event-labeled gene trees $(T;t,\sigma)$
to some species tree can be characterized in terms of underlying ``informative species triples'' 
displayed by the gene tree,  \cite{Hellmuth2017,HHH+12},  that is, 
in terms of the set
	\begin{align*}
	\mathcal{S}(T;t,\sigma) \coloneqq \{  \sigma(a) \sigma(b) | \sigma (c) \colon &
								ab|c \in R(T), t(\lca_T(a,b,c))=\bul, \text{ and } \\ & \sigma(a), \sigma(b), \sigma (c) \text{ are pairwise distinct} \}  
	\end{align*}

In particular, in \cite[Theorem 8]{HHH+12} the existence of relaxed
tree reconciliation maps for
binary \rev{event-labeled} gene trees \rev{$(T;t,\sigma)$} is characterized in terms of the 
compatibility of $\mathcal{S}(T;t,\sigma)$  
under the following assumption:
For any vertex $x$ \rev{of $T$} with $t(x) = \bul$
and children
$x_1$ and $x_2$ we have  $\sigma(L_T(x_1))\cap \sigma(L_T(x_2)) = \emptyset$. 
Lemma~\ref{lem:tree2tree-reconc}(iii) implies that this condition
is always satisfied for relaxed tree reconciliation maps and, therefore, also
by  tree reconciliation maps between such trees. The latter argument combined with Corollary~\ref{cor:bin} 
and Theorem~\ref{thm:tree2tree-reconc} therefore immediately implies the following 	
mild generalization of the aforementioned characterization from \cite{HHH+12}:

\begin{theorem}
  Let  $(T=(V,E);t,\sigma)$ be an event-labeled binary gene tree
  \rev{and let $S$ be a species tree}. 
  Then, there exists a relaxed tree reconciliation map $\mu$
  \rev{from} $(T;t,\sigma)$
  \rev{to} $S$
	(or equivalently a TreeNet-reconciliation map $\mu$ from $(T;t,\sigma)$ to $S$)
	if and only if $\mathcal{S}(T;t,\sigma)$ is compatible and,
  for any
	$x\in V$ with $t(x) = \bul$ with children
	$x_1$ and $x_2$, we have $\sigma(L_T(x_1))\cap \sigma(L_T(x_2)) = \emptyset$. 
	\label{thm:speciestriples}
\end{theorem}

It is worth mentioning that a relaxed tree reconciliation map $\mu$ exists for a
gene tree $(T;t,\sigma)$ that satisfies the conditions in Theorem \ref{thm:speciestriples}
to \emph{any} species tree $S$ that displays all triples in $\mathcal{S}(T;t,\sigma)$ \cite{HHH+12}. 
Moreover, this map can be computed in polynomial time \cite{HHH+12}. 

Note that the main result in \cite[Theorem 8]{HHH+12}
was generalized in \cite{Hellmuth2017} for non-binary \rev{event-labeled} gene trees
as follows.

\begin{theorem}[{\cite[Thm.\ 5.7]{Hellmuth2017}}]
	Let  $(T=(V,E);t,\sigma)$ be an event-labeled gene tree such that 
	for every vertex $x$ in $T$ with $t(x)=\bul$ and children  $x_1\dots,x_k$, $k\geq 2$, we
        have
	$\sigma(L_T(x_i)) \cap \sigma(L_T(x_j)) = \emptyset$, $1\leq i<j\leq k$.
	Then, there exists a tree reconciliation map $\mu$ \rev{from}
        $(T;t,\sigma)$ 
        \rev{to} some species tree $S$
	if and only if $\mathcal{S}(T;t,\sigma)$ is compatible. 
	\label{thm:speciestriples2}
\end{theorem}

\begin{figure}[tbp]
  \begin{center}
    \includegraphics[width=.9\textwidth]{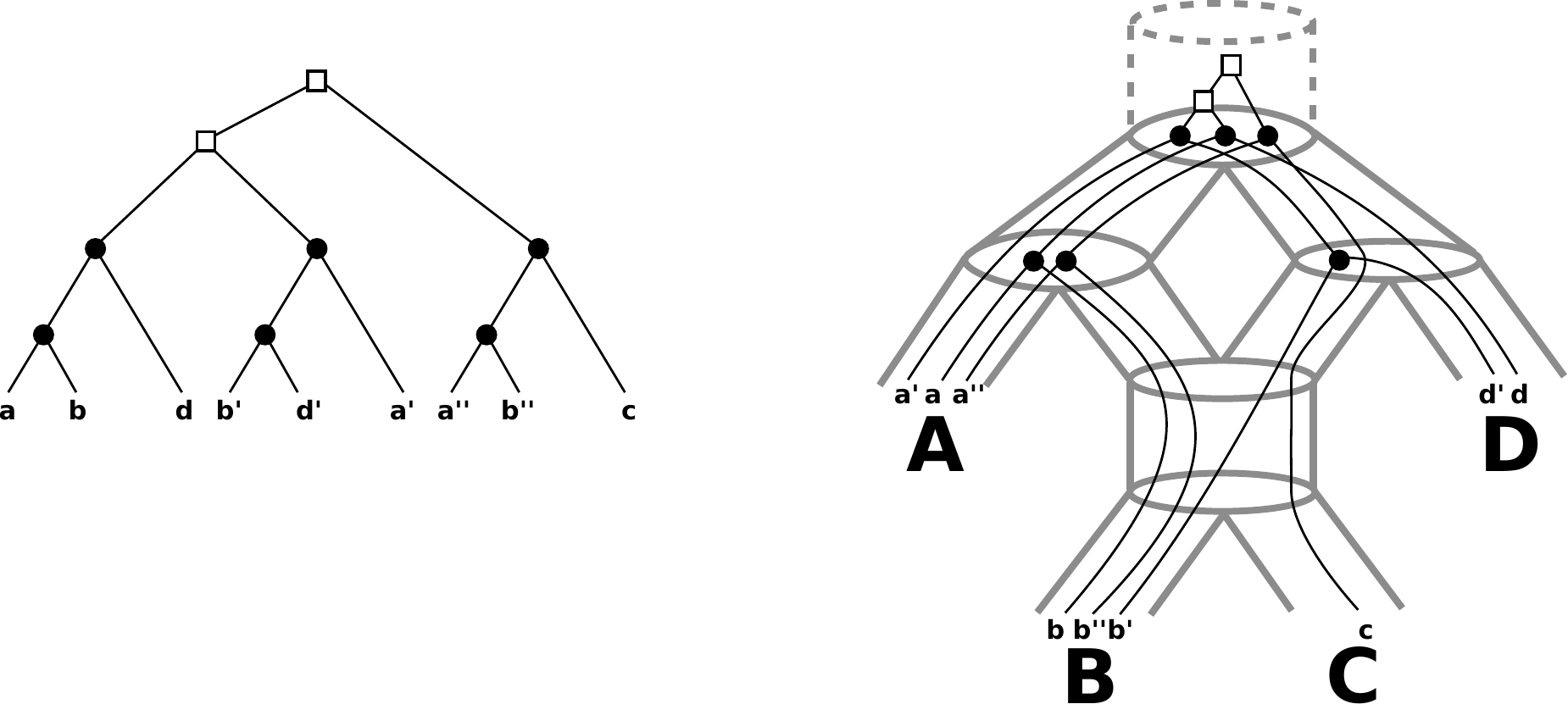}
  \end{center}
	\caption{
	Left, an event-labeled binary gene tree $(T;t,\sigma)$ on a 
	set $\Gen = \{a,a',a'',b,b',b'',c,d,d'\}$ of genes is shown. 
		\rev{Speciation and duplication events 
	are represented as $\bullet$ and $\square$, respectively. } 
	For $(T;t,\sigma)$ there is neither a 
	  (relaxed) tree reconciliation  nor a TreeNet-reconciliation map
  to any species tree. 
	Right, a \rev{tube-like} network $N$ on $\Spe=\{A,B,C,D\}$ is shown. 
	There is a TreeNet-reconciliation map from $(T;t,\sigma)$ 
	to $N$, which is implicitly shown 
	by drawing the gene tree
        \rev{inside $N$}. See text for further details.
	   } 
	\label{fig:rel-rec2}
\end{figure}

One might hope to obtain similar characterizations for 
relaxed tree reconciliation maps or TreeNet-reconciliation maps
for non-binary \rev{event-labeled} gene trees.
However, this does not appear to be straight-forward. 

In particular, in Fig.\ \ref{fig:rel-rec} an example is presented (based
on \cite[Fig.\ 3]{Hellmuth2017}) which shows that the compatibility of $\mathcal{S}(T;t,\sigma)$ 
is not necessary for the existence of a 
relaxed tree reconciliation map for a non-binary event-labeled gene tree $(T;t,\sigma)$, 
even though the conditions in Theorem \ref{thm:speciestriples2} for $(T;t,\sigma)$ are met.
Thus, one might be inclined to try and obtain at least a characterization of
TreeNet-reconciliation maps 
in terms of ``informative'' species triples from a \emph{binary} gene tree.
However, this does not seem to be as straight-forward as it might sound in view
of 
the examples given in Figs.\ \ref{fig:rel-rec2} and \ref{fig:counter-triples}.

\begin{figure}[tbp]
  \begin{center}
    \includegraphics[width=.9\textwidth]{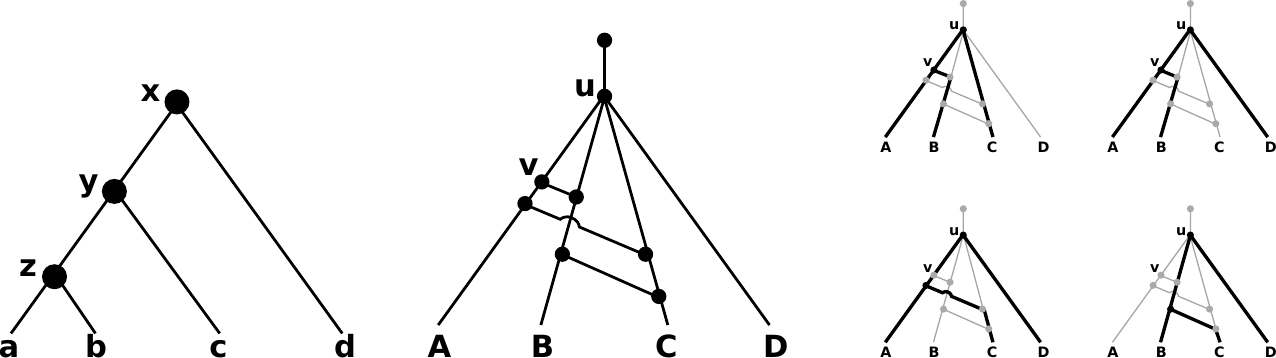}
  \end{center}
	\caption{
		Left, an event-labeled binary gene tree $(T;t,\sigma)$ on a 
	 set $\Gen = \{a,b,c,d\}$ of genes where all inner vertices are 
	 speciation vertices. 	The species network (middle) displays all triples in
	$\mc{S}(T;t,\sigma)  =\{AB|C, AB|D, AC|D, BC|D\}$ as highlighted in the right part.
	However, there is no reconciliation map $\mu$ from $(T;t,\sigma)$ to $N$. 
	See text for further details.
    } 
	\label{fig:counter-triples}
\end{figure}

More precisely,  
consider the event-labeled binary gene tree $(T;t,\sigma)$ in Fig.\ \ref{fig:rel-rec2} (left).
For this tree we have $\sigma(a) = \sigma(a') = \sigma(a'') = A$, 
$\sigma(b) = \sigma(b') =\sigma(b'') = B$,
$\sigma(c) = C$ and $\sigma(d) = \sigma(d') = D$.
Therefore,  $AB|D, BD|A\in \mc{S}(T;t,\sigma)$ and hence, the set 
	$\mc{S}(T;t,\sigma)$ is incompatible. 
	Thus, no (relaxed) tree reconciliation  or TreeNet-reconciliation map
  from $(T;t,\sigma)$
	to a species tree can exist in view of Theorem \ref{thm:speciestriples}. 
 Fig.\ \ref{fig:rel-rec2} (right) shows a network $N$ on $\Spe = \{A,B,C,D\}$. 
	There is a TreeNet-reconciliation map  from $(T;t,\sigma)$ 
	to $N$, which is implicitly shown in Fig.\ \ref{fig:rel-rec2} (right)
        by drawing $(T;t,\sigma)$ inside $N$. 
	Here,  $N$ does not display the species triple
	$AB|C\in  \mc{S}(T;t,\sigma)$. Thus, it is not necessary that 
 each triple of $\mathcal{S}(T;t,\sigma)$ is displayed by $N$
	for the existence of a TreeNet-reconciliation map. 

Now consider the 
 event-labeled binary gene tree $(T;t,\sigma)$ as in Fig.\ \ref{fig:counter-triples} (left). 
	 Here, $\sigma(a)=A$, $\sigma(b)=B$, $\sigma(c)=C$, and $\sigma(d)=D$.
	The network $N=(W,F)$ in Fig.\ \ref{fig:counter-triples} (middle) displays all triples in
	$\mc{S}(T;t,\sigma)  =\{AB|C, AB|D, AC|D, BC|D\}$.
	However, there is no TreeNet-reconciliation map $\mu$ from $(T;t,\sigma)$ to $N$. 
	To see this, assume for contradiction that there is a TreeNet-reconciliation map $\mu$. 
	Observe first that each gene is mapped to the corresponding 
	species (cf. Property~(R1)). Due to Property~(R3), we have $\mu(x)\succ_N \mu(d) = D$. 
	The latter together with  Property~(R2.i) implies that only $\mu(x)=u$ is possible. 
	In addition Property~(R3) implies that $\mu(x)\succ_N\mu(y)\succ_N A,B,C$ 
	and hence, only $\mu(y)=v$ is possible. However, Properties~(R2.i) and (R3) imply that $\mu(z)\in W^0$ and 
	$\mu(y) \succ_N\mu(z)\succ_N A,B$, which is not possible. 
	Thus, $N$ is not a network for $(T;t,\sigma)$. 
	Note,  $(T;t,\sigma)$  satisfies the assumptions in Theorem \ref{thm:speciestriples}
	and $\mc{S}(T;t,\sigma)$ is compatible. Thus, Theorem \ref{thm:speciestriples} implies that 
	for any species tree $S$ that displays
	all triples in $\mc{S}(T;t,\sigma)$ there is a TreeNet-reconciliation map
	from $(T;t,\sigma)$ to $S$.

\section{Unfoldings and Foldings}  
\label{sec:fold}

In the next section we shall show that we can always 
reconcile a given event-labeled gene tree with \emph{some} species network.
To show this we shall use the concept of \emph{foldings} of MUL-trees
which we shall consider in this section.

We begin by recalling an unfolding operation $U$ of networks that was 
first proposed in \cite{HM:06}.
For any network $N$ on $\Spe$, this operation constructs a pseudo MUL-tree
$(U^*(N),\chi^*)$  as follows:
\begin{itemize}
	\item The vertex set $V^*$ of $U^*(N)$ is the set of all directed paths $P$ in $N$
	that start at the root of $N$ and end in a vertex of $N$,
	\item there is an arc from a vertex $P$ in $U^*(N)$ to a vertex $P'$ in $U^*(N)$ if and only if $P' = P a$
	holds for some arc $a$ in $N$, and 
	\item and $\chi^*\colon \Spe\to  2^{L(U^*(N))}- \{\emptyset\} $
	is the map that assigns to all $x\in \Spe$ the set of directed paths in
	$V^*$ that end in $x$.
\end{itemize}

The MUL-tree obtained by suppressing all indegree one and outdegree one vertices
in $U^*(N )$, if there are any, is denoted by $U(N )$. In \cite{HM:06}, 
it is shown that $(U(N ),\chi^*)$ is indeed a MUL-tree and that it is unique.

We now want to perform the converse operation and ``fold up'' 
a MUL-tree into a network. We formalize this by 
adapting the concept of a folding map as defined in
\cite[p.1771]{huber2016folding}.

\begin{defi}[Folding map of a MUL-tree onto a network] 
	Given a pseudo MUL-tree $(M=(D,U),\chi)$ labeled by $\Spe$, 
	and a  network $N=(W,F)$ on $\Spe$, a \emph{folding 
	map} $f\colon (M,\chi) \to N $ is a pair of surjective
        functions $\fV\colon D \to W$ and
        $\fE\colon U \to F$ 
	such that 
	\begin{itemize}
		\item[(F1)] 
		for all $a \in U$ we have $t_N(\fE(a)) = \fV(t_M(a))$ and $h_N(\fE(a)) = \fV(h_M(a))$,
		\item[(F2)] if $v \in L(M)$ with $v \in \chi(x)$ for some $x\in \Spe$ then $\fV(v)=x$, and
		\item[(F3)] for each arc $a \in F$ and $v \in D$ such that $\fV(v)=t_N(a)$
		there is a unique arc $\widetilde{a^v}  \in U$ (the {\em lifting of the arc $a$ at $v$}) 
		such that $\fE(\widetilde{a^v})=a$ and $t_M(\widetilde{a^v}) = v$. 
	\end{itemize}

	We say that a \emph{MUL-tree $M$ can be folded into a network
        $N$}, whenever there is a
		subdivision $M'$ of $M$ such that there is a folding map $f$ from $M'$ to $N$. 
	\label{def:MUL2N}
\end{defi}

	Informally speaking, 
	Property (F1) ensures that the maps $\fV$ and $\fE$ are linked, that is, 
	each arc $a$ in $M$ is mapped to an arc $a'$ in $N$ such that
	the tail and head of $a$ is mapped to the tail and head of $a'$, respectively. 
	Property (F2) ensures that each leaf (gene) is mapped to the species in which it resides. 
	Finally, Property (F3) ensures that two distinct arcs of $M$ with the same
        tail are never mapped
	to the same arc in $N$. 

        The following two results are restated from \cite[Theorem~2 and Corollary 2]{huber2016folding},
        and provide a
useful connection between foldings of MUL-trees into a multi-arc free
network $N$ and $(U^*(N),\chi^*)$:

\begin{theorem}\label{universal-folding}
  Let $N$ be a multi-arc free network. Then the
  mapping $f^*: (U^*(N),\chi^*) \to N$ which takes 
  each vertex in $U^*(N)$ to its last vertex and
  each arc in $U^*(N)$ to its
	corresponding arc in $N$ is a folding map.
\end{theorem}

\begin{proposition}\label{prop:iso-mul-fold}
  Suppose that $(M,\chi)$ is a pseudo MUL-tree and $N$ is a multi-arc free
  network, both on $\Spe$, and that $f\colon (M,\chi)\to N$ is a folding map. 
Then  $(M,\chi)$ is isomorphic with $(U^*(N),\chi^*)$.
\end{proposition}

We conclude this section with two useful results concerning folding maps.
The first generalizes a result from \cite[p.\ 618]{HM:06}. As part of this,
we first present a construction that associates a network $F(M)$
to a MUL-tree $(M,\chi)$.

\begin{defi}
	Let $(M,\chi)$ with $M=(D,U)$ be a MUL-tree on $\Spe$ and
  $(M',\chi)$ with $M'=(D',U')$ be its simple subdivision.
	Now we construct a digraph $N=(W,F)$ as follows:

	For each $x \in \Spe$ find all the arcs in $M'$ \rev{whose}
        head
        \rev{is} contained in $\chi(x)$
	and identify them to obtain a digraph $N=(W,F)$. To be more precise, 
	for $x\in \Spe$ define $U_x\subseteq U$  as the set of all arcs in $M$ that are incident with a leaf $l$
	with $l\in \chi(x)$. If $|\chi(x)|\geq 2$, then	every arc $e=(u,v)\in U_x$ of $M$ is replaced	
	by the path $(u,v_e,v)$ to obtain $M'$, where $v_e\notin D$.	
	For each $x\in \Spe$, let $U_x^*$ be the set of all such new arcs $(v_e,v)$, where $(u,v)\in U_x$.  
	Hence, $|U_x^*|\geq 2$.
	Now, we obtain the digraph $N=(W,F)$ (with leaf set $\Spe$ and possibly with multi-arcs) 
	by identifying 	
	for each $x \in \Spe$ all the	arcs in $U^*_x$
	that yield the unique arc $(\parent(x),x)$ in $N$. 
  By construction, all vertices in $D\setminus L(M)$ are still contained in $N$. 
	\label{def:MULfoldN}
\end{defi}

See Fig.\ \ref{fig:MULsimple-foldN-2} and \ref{fig:MULsimple-foldN} 
for illustrative examples of Definition \ref{def:MULfoldN}. 
In fact, the digraph $N$ as in Definition \ref{def:MULfoldN} is a network  
as shown in the following 
\begin{lemma}
	Let $(M,\chi)$ be a MUL-tree on $\Spe$, let
  $(M',\chi)$ \rev{denote} its simple subdivision and let $N$
	be the digraph as in Definition \ref{def:MULfoldN}.
	Then, $N$ is a network.
\label{lem:NisNetwork}
\end{lemma} 
\begin{proof}
	Let $M=(D,U)$, $M'=(D',U')$  and $N=(W,F)$
	be the digraph as in Definition~\ref{def:MULfoldN}.

	To see that $N$ is a network, note first that, by construction, $N$
        satisfies Property~(N2).	
	To see Properties~(N1) and (N3), suppose that $x\in W$. 
  If $x\in D\setminus L(M)$, then, again
	by construction, the indegree and outdegree of $x$ in $N$ is the same as
	the indegree and outdegree of $x$ in $M$. Hence,  Property~(N1)
	holds in case $x$ is the root of $M$ and Property~(N3)
	holds if $x$ is an inner vertex of $M$ distinct from the root.
	Now, assume that $x$ is a leaf of $N$. By construction, 
	there is a unique arc $(y,x)$ in $N$. 
  If $y\in D$, then, again
	by construction, the indegree and outdegree of $x$ in $N$ is the same as
	the indegree and outdegree of $x$ in $M$.
	So assume that $y\notin D$. Then $|U_x^*|\geq 2$. It follows that the indegree of
	$y$ in $N$ is at least 2. Since, again by construction, the outdegree
	of $y$ must be 1, it follows that $y$ is a hybrid vertex of $N$.
	Thus, $N$ satisfies Property~(N3). Therefore, $N$ is indeed a network.
\qed
\end{proof}

\begin{lemma}
	Every MUL-tree $(M,\chi)$ on $\Spe$ can be folded into a network on $\Spe$. 
	The network as well as the underlying folding map can be constructed in polynomial time. 
	\label{lem:MULfoldN}
\end{lemma}
\begin{proof}
	Let $(M,\chi)$ with $M=(D,U)$ be a MUL-tree on $\Spe$ and let
  $(M',\chi)$ with $M'=(D',U')$ be its simple subdivision and let $N=(W,F)$
	be the digraph as in Definition \ref{def:MULfoldN}. By Lemma \ref{lem:NisNetwork}, 
	$N$ is a network.

	We next show that there exists a folding map $f:(M',\chi)\to N$ of $(M',\chi)$ into $N$. To see
   this consider first the map
	$\fV\colon D' \to W$ 
	given by  
	\[\fV(v) = \begin{cases}
			v & \text{if } v \in D \\
				t_N(e^x) & \, \text{else.} 
	\end{cases}\]
	By construction, $\fV$ is surjective.
		To obtain the map $\fE\colon U' \to F$, let us first examine in more detail the MUL-tree $M$,  its simple subdivision $M'$
		and the  network $N$. Let $x\in \Spe$. 
		If $|\chi(x)|\geq 2$, then each arc $e=(a,b)\in U_x$ in $M$ is replaced by 
		the path $(a,b_e,b)$ to obtain $M'$ and $\fV(b_e) = t_N(e^x)$ and $\fV(b) = x$. 
		Note that $N$ may contain multi-arcs. 
		More precisely, the construction implies that	there are multi-arcs $f_1,\dots,f_{\ell}$, $\ell\geq 2$ between 
		two vertices $u$ and $v$ in $N$ with $v\prec_N u$  if and only if $\fV^{-1}(u) = a \in D$ and 
		$a$ is adjacent in $M$ to exactly $\ell$ leaves $b_1,\dots b_{\ell}$ with  $b_i\in \chi(x)$ for some $x\in \Spe$ and $1\leq i\leq \ell$.
		In other words, 	there are exactly $\ell\geq 2$ arcs $e_1,\dots,e_{\ell}$ in $U_x$
		with $t_M(e_i) = a$. 	  
		Now, we put $\fE(e_i) = f_i$ for all such arcs $e_i \in U_x$, $1\leq i \leq \ell$ and $x\in  \Spe$. 
		\rev{For} all other arcs $e=(u,v)\in U'$ that do not result in multi-arcs in $N$, we
		put $\fE(e) = (\fV(u),\fV(v))$.
	By construction $\fE$ is clearly surjective. Leaving the details
	of the proof to the interested reader, it is also not difficult
	to see that $f$ satisfies Properties~(F1) - (F3). Thus, $f$ is 
		a folding map of $(M',\chi)$ into $N$. 
				By Definition~\ref{def:MUL2N} and since $M'$ is a subdivision of $M$, it follows that 
		the MUL-tree $(M,\chi)$ can be folded into $N$.

	Finally, all construction steps to obtain $N$ and $f$  can obviously be carried out in polynomial time. 
\qed \end{proof}

\begin{figure}[tbp]
	\begin{center}
		\includegraphics[width=1.\textwidth]{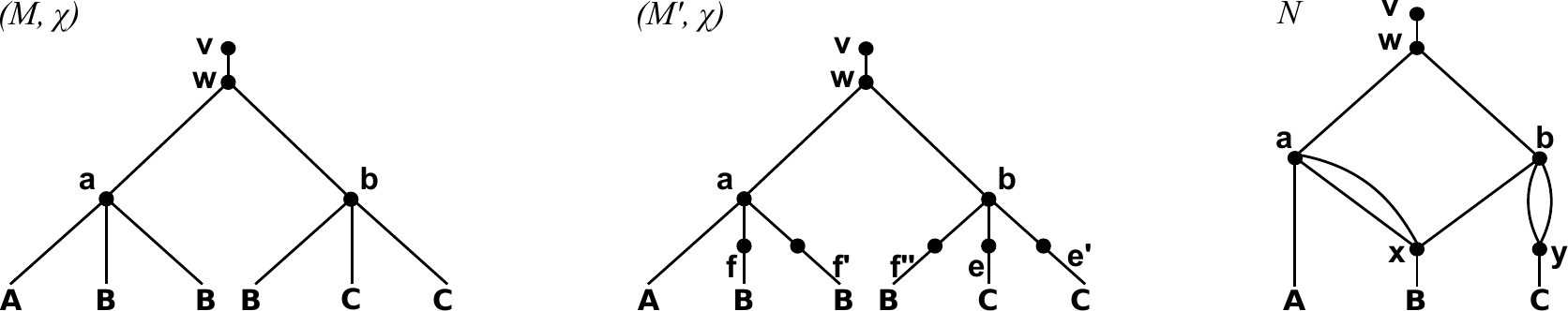}
	\end{center}
	\caption{ 
			Left, a MUL-tree $(M,\chi)$ and its simple subdivision 
			$(M',\chi)$ (middle panel). In the right panel,    
				a network $N$ that is obtained from 
			$(M',\chi)$ by identifying the two arcs 
			$e$ and $e'$ to obtain the arc $(y,C)$ in $N$, 	and 
			the arcs $f, f'$ and $f''$ to obtain the arc $(x,B)$ in $N$. 
			In addition, the identification of $f$ and $f'$ (resp.\ $e$ and $e'$) yields the two
			multi-arcs between $a$ and $x$ (resp.\ $b$ and $y$) in $N$, 
			see Definition \ref{def:MULfoldN}
		 for details.  Note that $N$ is the ``fold up'' of $(M',\chi)$ and, therefore, 
			also of $(M,\chi)$. 
	}
	\label{fig:MULsimple-foldN-2}
\end{figure}

Our second result shows that \rev{folding maps are} ancestor preserving. 

\begin{lemma}
	Let $f=(\fV,\fE)$ be folding map  
from a pseudo MUL-tree $(M = (D,U),\chi)$ on $\Spe$ into a
	network $N=(W,F)$ on $\Spe$. Then, 
	for any $a,b\in D\cup U$ with $a\preceq_M b$, we have
	$g(a)\preceq_N g(b)$ where $g(x)$ with $x\in\{a,b\}$ is $\fV(x)$ if $x\in D$ and $\fE(x)$ otherwise. 
	In particular, if $a$ and $b$ are distinct and not both contained in $U$, then $g(a)\prec_N g(b)$.
	\label{lem:MUL-ancestor}
\end{lemma}
\begin{proof}
	First assume that $a,b\in U$ and $a\preceq_M b$. 
	W.l.o.g. we may assume that $a\prec_M b$ as for $a=b$ the lemma trivially holds. 
	Then, there is a (possibly single-vertex) directed path $P$ from $h_N(b)$ to $t_N(a)$. 
	Let us denote the vertices crossed by $P$ by $h_N(b) = v_1,\dots,v_k = t_N(a)$ (in order of their appearance in $P$)
	and, in case $k\geq 2$, let $e_i = (v_i,v_{i+1})\in U$, $1\leq i < k$. Moreover, let $v_0,v_{k+1}\in D$ such that
	$b=(v_0,v_1)$ and $a=(v_k,v_{k+1})$.
	By Property (F1), we have $\fV(v_{i+1})\prec_N \fV(v_{i})$ for all $0\leq i\leq k$. 
	Therefore, $\fE(a) \prec_M \fE(e_{k-1})\prec_N \dots \prec_N \fE(e_{1}) \prec_N \fE(b)$. 
	
	Now assume that $a\in U$, $b\in D$ and $a\preceq_M b$. Then
	there is an arc $e$ with tail $t_M(e)=b$. Hence,  $a\preceq_M e$.
	By the previous argument we have $\fE(a)\preceq_N \fE(e)$ and 
	Property~(F1) implies $\fE(e)\prec_N \fV(b)$.
	
	Now assume that $a\in D$, $b\in U$ and $a\preceq_M b$. Then there is an arc $e$ with tail $h_M(e)=a$.
	Similar arguments as in the previous case show again that $\fV(a) \prec_N \fE(b)$. 
	
	Finally, assume that $a,b\in D$. If $a=b$, then clearly $\fV(a)=\fV(b)$.  
	So assume $a\neq b$ and, thus, $a\prec_M b$. Then, there is 
	an arc $e$ with $t_M(e)=b$ and an arc $e'$ with $h_M(e') = a$. 
	Similar arguments as in the previous cases combined with 
	Property~(F1) imply that 
	$\fV(a)\prec_N \fE(e') \preceq_N \fE(e) \prec_N \fV(b)$. 
\qed \end{proof}

\section{Existence of Reconciliation Maps to Networks}
\label{sec:existence}

In Fig.\ \ref{fig:rel-rec2}, we presented an example which 
shows that if there is a TreeNet-reconciliation map from an event-labeled  gene tree
$(T;t,\sigma)$ to a network $N$, then $N$ does not need to display all
informative triples in $\mc{S}(T;t,\sigma)$. In fact, the network may display
species triples that are not supported by $(T;t,\sigma)$. For example, 
the network in Fig.\ \ref{fig:rel-rec2} displays 
the triple $BC|A$ although $AB|C \in \mc{S}(T;t,\sigma)$.
In other words,
a network $N$, for which a TreeNet-reconciliation map from $(T;t,\sigma)$ to $N$ exists, 
does not need to preserve much (or possibly even any) of the structure of $T$. Thus, the question
arises as to whether there always exists a TreeNet-reconciliation map from $(T;t,\sigma)$ to
some network? In Theorem~\ref{thm:composed}, the main result of
this section, we show that this is indeed always the case.

To establish Theorem~\ref{thm:composed},
we first define reconciliation maps between event-labeled gene trees and pseudo-MUL-trees,
a topic that has recently been studied in a somewhat different form in \cite{gregg2017gene}.
Then, we show how to associate to any event-labeled gene tree $(T;t,\sigma)$
a MUL-tree $(M(T;t,\sigma),\chi)$ such that $(T;t,\sigma)$ can be reconciled
with this MUL-tree via some map $\kappa_{(T;t,\sigma)}$. Using the notion
of a 	``combined  reconciliation map'', we then exploit $\kappa_{(T;t,\sigma)}$ to
define a reconciliation map between $(T;t,\sigma)$ and the fold up of $(M(T;t,\sigma),\chi)$.

\begin{defi}[Reconciliation map to a pseudo MUL-tree] \label{def:p-mu} 
	Suppose that $\Spe$ is a set of species, \rev{that} $M=(D,U)$ is a pseudo MUL-tree on $\Spe$, 
	and \rev{that} $(T;t,\sigma)$	is an event-labeled gene tree on $\Gen$. 

	Then,  we say that \emph{$(M,\chi)$ is a pseudo MUL-tree for
	$(T\rev{=(V,E)};t,\sigma)$} if there is a map $\kappa\colon V\to (D\backslash D^1)\cup U$ such that, for all $x\in V$:  
	\begin{description}
		\item[(M1)] \emph{Leaf Constraint.} If $x\in \Gen$ then $\kappa(x)\in L(M)$ and $\kappa(x)\in \chi(\sigma(x))$
		\item[(M2)] \emph{Event Constraint.}
		\begin{itemize}
			\item[(i)]  If $t(x)=\bul$ and $x$ has children $x_1,\dots,x_k$, $k\geq 2$, then 	
									$\kappa(x) \in D\backslash L(M)$ and there exists a directed path $P_i$ from $\kappa(x)$ to
										 $\kappa(x_i)$ and a directed path $P_j$ from $\kappa(x)$ to $\kappa(x_j)$
											for two distinct $i,j\in \{1,\dots,k\}$ such that, in $M$,
											the	first arc on $P_i$ is incomparable with
											the first arc on  $P_j$. 
			\item[(ii)] If $t(x) = \squ$, then $\kappa(x)\in U$. 
		\end{itemize} \vspace{0.03in}
		\item[(M3)] \emph{Ancestor Constraint.}		\\	
		Let $x,y\in V$ with $x\prec_{T} y$, then  
		we distinguish between the two cases:
		\begin{itemize}
			\item[(i)] If $t(x) = t(y) = \squ$, then $\kappa(x)\preceq_M \kappa(y)$, 
			\item[(ii)] otherwise, i.e., at least one of $t(x)$ and $t(y)$ is a speciation $\bul$, 
			$\kappa(x)\prec_M\kappa(y)$.
		\end{itemize}
	\end{description}
	We call $\kappa$ the \emph{MUL-reconciliation map} from $(T;t,\sigma)$ to $M$. 
\end{defi}

Note that Properties~(M1), (M2.ii) and (M3) are canonical extensions 
of the Properties~(R1), (R2.ii) and (R3) of TreeNet-reconciliation maps.
Moreover,  
Property (M2.i) and the fact that $M$ has no hybrid vertices
implies the \rev{following} weaker property for all $x\in D$: 
if $t(x)=\bul$ and $x$ has at least two children, then there exist two children $x_1$ and $x_2$ 
such that $\kappa(x_1) $ and $\kappa(x_2)$ 
are incomparable in $M$. However, the converse implication does not 
always hold. Moreover, Property~(M2.i) cannot be weakened to 
establish results for TreeNet-reconciliation maps
based on MUL-reconciliation maps and particular foldings.

  We next provide a construction that allows us to associate a MUL-tree
  to an event-labeled gene tree.
  Suppose $(T;t,\sigma)$ is an event-labeled gene tree.
  Let $(M(T;t,\sigma),\chi)$
  denote the MUL-tree obtained from $(T;t,\sigma)$ as follows:
  First replace every leaf of $T$ by its label under $\sigma$. Next,
  add an incoming arc to the root \rev{of $T$} to obtain a tree with root having outdegree 1.
  The resulting MUL-tree is $M(T;t,\sigma)$. To obtain $\chi$, we
put $\chi(x)=\{l\in L(T)\mid \sigma(l)=x\}$, for all $x\in \Spe$.

  \begin{defi}
    Suppose that  $(T=(\rev{V},E);t,\sigma)$ is an event-labeled gene tree. Then  we
    call the MUL-tree $(M(T;t,\sigma),\chi)$
    the \emph{MUL-tree for $(T;t,\sigma)$}. Furthermore, \rev{putting $M(T;t,\sigma)=(D,U)$}
    we refer to the map $\kappa_{(T;t,\sigma)}:V\to \rev{D\cup U}$
    given by putting
$\kappa_{(T;t,\sigma)}(x) = \sigma(x)$ if $x\in V$ and, for all $x\in V^0$, by
putting $\kappa_{(T;t,\sigma)}(x) = x$ if  $t(x) = \bul$ and $\kappa_{(T;t,\sigma)}(x) = e^x$ otherwise
as the {\em (trivial) map} from 
$(T;t,\sigma)$ to $(M(T;t,\sigma),\chi)$ 
\label{def:assoMUL}
\end{defi}

  \begin{lemma}\label{trivial}
The map  $\kappa_{(T;t,\sigma)}$ is a MUL-reconciliation from $(T;t,\sigma)$ to $(M(T;t,\sigma),\chi)$.
\end{lemma}
\begin{proof}
	For simplicity let $M\coloneqq M(T;t,\sigma)$. 
	By definition, $\kappa \coloneqq \kappa_{(T;t,\sigma)}$ satisfies (M1), (M2.ii) and (M3).
	To see that $\kappa$ is a MUL-reconciliation map, it thus remains to show that 
	Property (M2.i) is satisfied. 
	Let  $x$ be a vertex in $T$ with $t(x)=\bul$ and with children $x_1,\dots,x_k$, $k\geq 2$. 
	By definition, $\kappa(x) \in D\backslash L(M)$. Now let $x_i$ and $x_j$ be two distinct 
	children of $x$. By construction of $M$, all vertices of $T$ are contained in $M$. 
	However, to make the reading easier, we denote by $v'$ the vertices in $M$
	that correspond to vertex $v$ in $T$.  
	Each of the two vertices $x_i$ and $x_j$ may be a leaf or an inner vertex 
	equipped with a particular event $\bul$ or $\squ$.
 	A straight-forward case analysis and the fact that $\kappa(x_l)$
	is either $x'_l$ or the arc $e^{x'_l}$ where $l\in \{i,j\}$ together 
	with Property (M3) shows that in either case we have $x'_i \preceq_M \kappa(x_i)\prec_M\kappa(x)$
	and $x'_j \preceq_M \kappa(x_j)\prec_M\kappa(x)$.  
  Note that, by construction, the path from $\kappa(x)$ to $\kappa(x_i)$ is the arc $e^{x'_i}$
	and the path from $\kappa(x)$ to $\kappa(x_j)$ is the arc $e^{x'_j}$. 
	Since, $M$ is a tree and $x_i$ and $x_j$ are incomparable in $T$
	the vertices $x'_i$ and $x'_j$ as well as the arcs $e^{x'_i}$ and $e^{x'_j}$ 
	must be incomparable in $M$. 
	Therefore, (M2.i) is also satisfied.
\qed \end{proof}

\rev{Calling the map $\kappa_{(T;t,\sigma)}$ the \emph{trivial MUL-reconciliation map},}
we now prove a technical result which will allow us to link trivial MUL-reconciliation maps
with foldings.

\begin{lemma}\label{lem:reconcMUL-pseudoMUL} 
  \rev{  Suppose that $(T;t,\sigma)$ is an event-labeled gene tree and that $(M,\chi)$
    is a MUL-tree.}
  If there is a  MUL-reconciliation map from $(T;t,\sigma)$ to
  $(M,\chi)$, 
  then there is a  MUL-reconciliation map from $(T;t,\sigma)$ to any subdivision
  $(M',\chi)$ of $(M,\chi)$. 
\end{lemma}
\begin{proof}
Let $\kappa$ be a MUL-reconciliation map from an event-labeled gene tree 
$(T\rev{=(V,E)};t,\sigma)$ to some MUL-tree $(M\rev{=(D,U)},\chi)$ and 
$(M',\chi)$ be a subdivision of $(M,\chi)$. 
For each arc $e=(u,v)$ of $M$ that is subdivided in the construction of $M'$
by a directed path $P_{uv}$ from $u$ to $v$ in $M'$
we denote by $e^*$ the last arc in $P_{uv}$. If $e$ is not subdivided, we put $e^*=e$. 
For all  $v\in \rev{V}$
put 
	\[\kappa'(v) = \begin{cases}
	  \kappa(v) & \text{if } \kappa(v)\in \rev{D}
          \\
		 {\kappa(v)}^* & \text{else.}
	\end{cases}\]
It is now straight-forward to see that $\kappa'$ is a 
MUL-reconciliation map from $(T;t,\sigma)$ to $(M',\chi)$. 
\qed \end{proof}

\begin{figure}[tbp]
  \begin{center}
    \includegraphics[width=.9\textwidth]{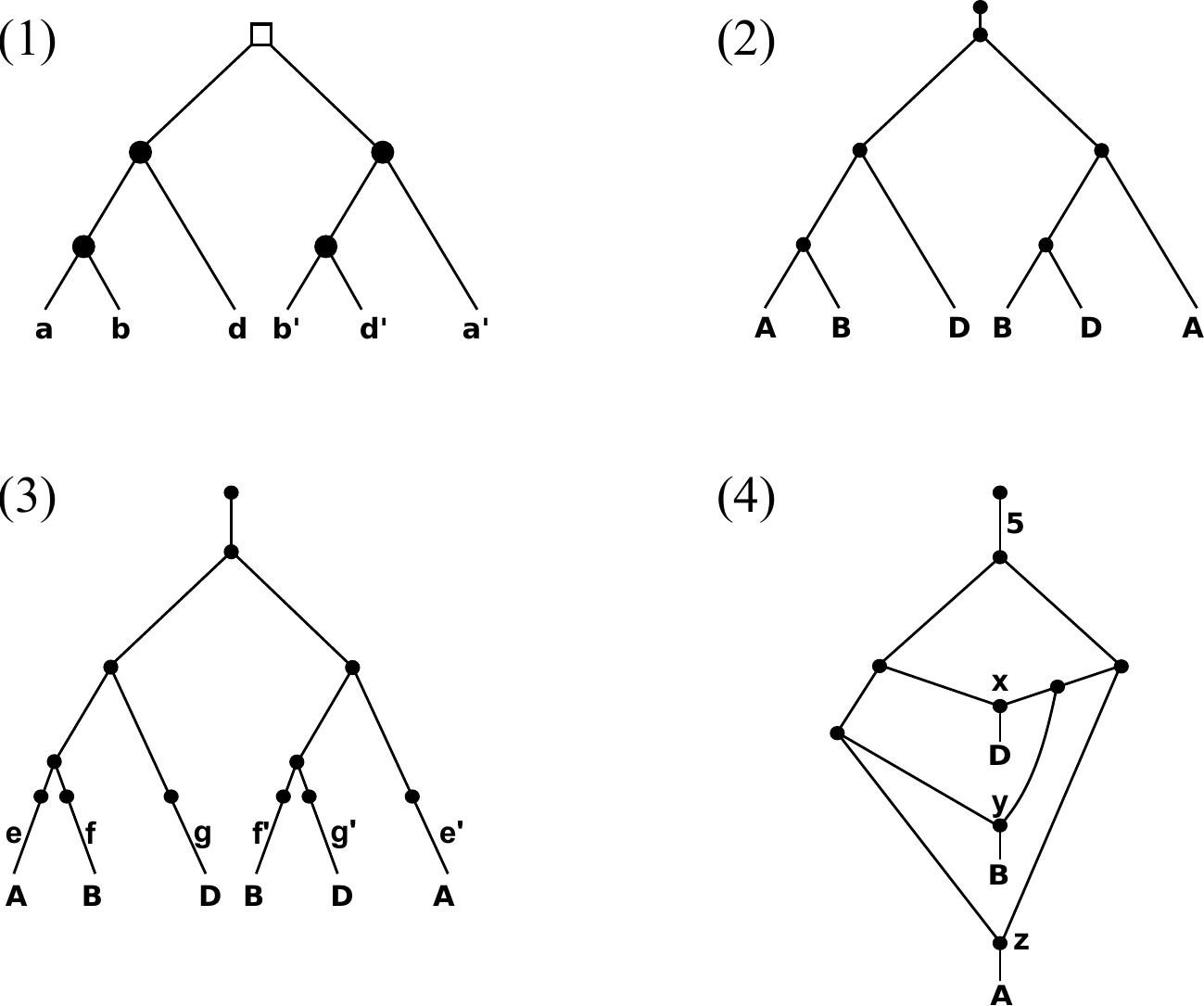}
  \end{center}
	\caption{
	 In Panel (1), we depict an event-labeled gene tree $(T;t,\sigma)$.
		\rev{Speciation and duplication events 
	are represented as $\bullet$ and $\square$, respectively. }  
	The corresponding MUL-tree $(M(T;t,\sigma),\chi)$ 
  and the simple subdivision  $(M',\chi)$ of $(M(T;t,\sigma),\chi)$
	are shown in Panels (2) and (3), respectively.
   As shown in the proof of Lemma \ref{lem:MULfoldN}, 
	$(M',\chi)$ can be folded onto $N$ (Panel (4)). Here, $N$ is obtained from $(M',\chi)$ by identifying 
	the arcs $e$ and $e'$ that yield the arc $(z,A)$ in $N$, 
	the arcs $f$ and $f'$ that yield the arc $(y,B)$ in $N$, and 
	the arcs $g$ and $g'$ that yield the arc $(x,D)$ in $N$.  
	By Definition~\ref{def:MUL2N} and since $M'$ is a subdivision of $M(T;t,\sigma)$, 
	the MUL-tree $(M(T;t,\sigma),\chi)$ can be folded onto $N$. }
		\label{fig:MULsimple-foldN}
\end{figure}

We are now in the position to establish the main result of this section.

\begin{theorem}\label{thm:composed}
Given a MUL-reconciliation map $\kappa$ from an
event-labeled gene tree $(T\rev{=(V,E)};t,\sigma)$ to a pseudo MUL-tree $(M,\chi)$, 
and a folding map $f$ from $(M\rev{=(D,U)},\chi)$ to a  network $N\rev{=(W,F)}$.
Then, the map $\mu_{\kappa,f} \colon V \to W \cup F$ defined by putting  
for every $v\in V$
\[\mu_{\kappa,f}(v) = \begin{cases}
	\fV(\kappa(v)) & \text{if } \kappa(v) \in D, \\
	\fE(\kappa(v)) & \, \text{\rev{otherwise},} 
\end{cases}\]
is a \rev{TreeNet-reconciliation map from} $(T;t,\sigma)$  and $N$.
\end{theorem} 
\begin{proof} 
We need to show that $\mu\coloneqq \mu_{\kappa,f}$ satisfies Properties (R1) -- (R3).

To see Property (R1), let $x\in \Gen$.  
Since Property (M1) is satisfied for $\kappa$ we
have $\kappa(x) \in L(M)$ and $\kappa(x) \in \chi(\sigma(x))$.
By Property (F2) and \rev{the} construction of $\mu$, we have  
$\mu(x) = \fV(\kappa(x)) = \sigma(x)$. Thus, Property (R1) holds.

To see Property (R2.i), let $x \in V$ be a vertex with  $t(x)= \bul$ and children $x_1,\dots,x_k$, $k\geq2$. 
We need to show that $\mu(x) \in Q^2_N(\mu(x_1),\dots,\mu(x_k))$.
By (M2.i),  there exists a directed path $P_i$ from $\kappa(x)$ to
$\kappa(x_i)$ in $M$ and a directed path $P_j$ from $\kappa(x)$ to
$\kappa(x_j)$ in $M$ for two distinct $i,j\in \{1,\dots,k\}$ such that 
in $M$ the	first arc on $P_i$ is incomparable with
the first arc on  $P_j$.
Let $a_i$ and $a_j$ be the first arc on $P_i$ and $P_j$, respectively. 
Hence, $t_M(a_i) = t_M(a_j) = \kappa(x)$ and, since $f$ is a folding map,  
Property (F1) implies that 
$t_N(\fE(a_i)) = t_N(\fE(a_j)) = \fV(\kappa(x))$.
The latter together with (F1) implies that  $a_i$ and $a_j$ are mapped in $N$
to either the same arc or to two distinct arcs that share the same tail. 

Assume first that $\fE(a_i) = \fE(a_j) = e$ and put $v=\kappa(x)$. 
Then, $\fV(v) = t_N(e)$ and 
$\widetilde{a}\in \{a_i,a_j\}$ is an arc in $M$ that satisfies 
$\fE(\widetilde{a}) = e$ and $t_M(\widetilde{a}) = v$; contradicting Property (F3). 
Therefore,  $a_i$ and $a_j$ must be mapped in $N$
to two distinct arcs $e_i$ and $e_j$, respectively, that share the same tail.
Since, $e_i\neq e_j$ and $t_N(e_i) = t_N(e_j)= \fV(\kappa(x))$, the 
arcs $e_i$ and $e_j$ are incomparable in $N$.
Combined  with Property~(M3) and Lemma~\ref{lem:MUL-ancestor} it follows that 
there is a path $P'_i$ from $\mu(x) $ to $\mu(x_i)$ in $N$ that contains $e_i$ and 
a path $P'_j$ from $\mu(x) $ to $\mu(x_j)$ in $N$ that contains $e_j$. 
Thus, $\mu(x_i)$ and $\mu(x_j)$ are separated by $\mu(x)$. 
Therefore, $\mu(x)\in Q^2_N(\mu(x_1), \dots, \mu(x_k))$.

Clearly, Property~(R2.ii) follows from the fact that $\kappa$ satisfies (M2.ii), 
that $\fE$ maps an arc of $M$  to an arc of $N$  and  the construction of $\mu$.

It remains to show that Property~(R3) is satisfied. 
Suppose that $x,y\in V$ with $x \prec_T  y$. Clearly,  $x\neq y$. 
If  $t(x)=t(y)=\squ$, then $\kappa(x) \preceq_M \kappa(y)$ in view of Property (M3.i).  
Lemma~\ref{lem:MUL-ancestor} implies that $\mu(x) \preceq_N \mu(y)$. 
Thus, \rev{Property~}(R3.i) is satisfied. 
Now assume that at least one of $t(x)$ or $t(y)$ is a speciation \rev{vertex}. 
Property~(M3.ii) implies that $\kappa(x) \prec_M \kappa(y)$. 
Note that not both of $\kappa(x) $ and $ \kappa(y)$ 
can be contained in $U$. Again, 
Lemma \ref{lem:MUL-ancestor} implies that 
$\mu(x) \prec_N \mu(y)$. Therefore, (R3.ii) is also satisfied by $\mu$. 

In summary, it follows that $\mu$ is a \rev{TreeNet-}reconciliation
map from $(T;t,\sigma)$ to $N$. 
\qed \end{proof}

\begin{corollary}
	\rev{Given any event-labeled gene tree $(T;t,\sigma)$,  there exists} a 
	species network $N$  for $(T;t,\sigma)$.
	\rev{Moreover}, the network $N$ as well as the \rev{TreeNet-}reconciliation 
	map from $(T;t,\sigma)$	to $N$ can be constructed in polynomial time. 
	\label{cor:Nexists}
\end{corollary}
\begin{proof} 
Let $\kappa_{(T;t,\sigma)}$ be the trivial MUL-reconciliation from  $(T;t,\sigma)$
to $(M,\chi)$ with $M\coloneqq M(T;t,\sigma)$ (which is a reconciliation map by Lemma~\ref{trivial}).
Moreover, let $(M',\chi)$ be the simple subdivision of $(M,\chi)$.
\rev{Let} $N$ be the network as constructed
in the proof of Lemma \ref{lem:MULfoldN}
and \rev{let} $f$ \rev{denote the underlying}
folding map from $M'$ to $N$. 
Lemma \ref{lem:reconcMUL-pseudoMUL} implies that there
is a reconciliation map $\kappa'$ from   $(T;t,\sigma)$ to $(M',\chi)$. 
Then, by Theorem \ref{thm:composed}, the composed reconciliation map $\mu_{\kappa',f}$ is a 
reconciliation map from  $(T;t,\sigma)$ to $N$. 
Hence, $N$ is a	species network  for $(T;t,\sigma)$.

Finally, Lemma \ref{lem:MULfoldN} implies that the network $N$ and the folding map 
$f$ can be constructed in polynomial time. Furthermore, the construction of $\kappa$
and thus, of $\kappa'$ as well as of $\mu_{\kappa',f}$ can be done in polynomial time. 
This proves the second part of the corollary.
\qed \end{proof}

\section{Existence of Reconciliations for Multi-Arc Free Networks}
\label{sec:existence-multiarcFree}

In the last section, we showed that  every event-labeled gene tree 
can be reconciled with some network. An important assumption in
this result is that the network is permitted to contain multi-arcs. Although not unreasonable, 
in practice (and in much of the literature on networks) it can be desirable to 
restrict attention to networks which do not have multi-arcs.
\rev{\begin{defi}
	An event-labeled gene tree $(T;t,\sigma)$ is \emph{well-behaved}
	if given any $v\in V$ with $t(v)=\bul$, for every child $v'$ of $v$ in $T$
	there is another child $v''$ of $v$ in $T$  with $\sigma(L_T(v'))\neq \sigma(L_T(v''))$.
\label{def:well-behaved}
\end{defi}
Equivalently, an event-labeled gene tree $(T;t,\sigma)$ is well-behaved
if for all speciation vertices $v$ in $T$, 
$\sigma(L_T(v_1))= \sigma(L_T(v_2))= \dots = \sigma(L_T(v_k))$ does not hold for
$v_1,v_2\dots v_k$, $k\geq 2$, the children of $v$.
The latter is a reasonable and a quite weak restriction as in many applications 
the stronger condition $\sigma(L_T(v'))\cap \sigma(L_T(v'')) = \emptyset$
for any two distinct children $v'$ and $v''$ of a speciation vertex is (at least implicitly) 
required; see e.g. \cite{Lafond2014,Hellmuth2017,LDEM:16,Tofigh2011,BLZ:00}. }
In this section, we show that for every \rev{well-behaved} event-labeled gene tree $(T;t,\sigma)$
there is a TreeNet-reconciliation map to some network $N$ without multi-arcs and that, in addition, 
$N$ displays all triples in $\mc{S}(T;t,\sigma)$.

To this end, we will first take the folding of the simple
subdivision $(M',\chi)$ of
$(M(T;t,\sigma),\chi)$ to obtain a network $N$, as specified in 
Definition \ref{def:MULfoldN} and the proof of Lemma \ref{lem:MULfoldN}. However, this network may
contain multi-arcs
(see Fig.\ \ref{fig:MULsimple-foldN-2}). 
To adjust for this, one may be tempted to simply remove multi-arcs
and subsequently 
suppress 
degree two vertices. 
This, however, can be problematic as it may result  
in a network $N'$ for which no folding map from $M'$ to $N'$ exists. 
For example, consider the network $N$ depicted 
in Fig.\ \ref{fig:MULsimple-foldN-2} (right). 
The removal of one of the arcs
between $y$ and $b$, yields
	 a graph which is not a network (since then vertex $y$ has
	 in- and outdegree one). Additional
         suppression of $y$ results in a network $N'$, for
         which no folding map from 
	 the pseudo MUL-tree $M'$ pictured in
         Fig.\ \ref{fig:MULsimple-foldN-2} (center) exists, 
         since Property (F1) is violated for the resulting arc in $N$
         and vertex $b$ of $M'$.
Since there is no folding map from $(M',\chi)$ to $N'$, 
we can cannot apply  Theorem \ref{thm:composed} to conclude that
there is a TreeNet-reconciliation map from 
$(T;t,\sigma)$ to $N'$ (if there is \rev{one}). 

In order to obtain a multi-arc free network from $N$ for which there is a
TreeNet-reconciliation from $(T;t,\sigma)$, we next provide an alternative 
approach. For this, we need to define two sets which will
turn out to be helpful in the construction of
multi-arc free networks from networks that may contain  multi-arcs.

\begin{defi}
	Let $N=(W,F)$ be a network on $\Spe$. 
	For each $x\in \Spe$, we  denote by $\W^N_x$ the
        inclusion-maximal subset of vertices of $W$
	that comprises all vertices $v\in W$ that satisfy
        $x\prec_{N} v$ and	 $L_N(v) = \{x\}$. 

	Moreover,  we define $\V^N_x\subseteq W$ as the set
        of all vertices $z\in W$ with $z\notin \W^{N}_x$
	and there is an arc $(z,w)$ in $N$ with $w\in \W^{N}_x$. 	
\end{defi}

To illustrate theses two sets consider the bottom left network $N$  depicted in
Fig.~\ref{fig:multiarc-free}. Then $\W^N_x$ consists
of all vertices $v$ in $N$ that are highlighted by
colored ``{\large$\mathbf{\star}$}'' and that satisfy
$v\succ_N x$ for the particular leaf $x\in \Spe$. 
Thus, the set $W_1 = W\setminus \cup_{x\in\Spe'} (\W^N_x \cup \{x\})$
where $\Spe'$ denotes the set of all $x\in \Spe$ with $\W^N_x \neq \emptyset$
						is the set $W_1=\{ 1,2,3,4\}$.

\begin{lemma}\label{lem:new}
		Let $N=(W,F)$ be a network on $\Spe$. 	
	Then, the following statements are satisfied.
      \begin{description}
          \item[(i)] $x\not\in \W^N_x$, for all $x\in \Spe$.
          \item[(ii)]  For all $x\in \Spe$, if $v\in \W^N_x$,
            then every vertex $u\in W$ with
        		$u\prec_N v$ and $u\neq x$
			      must be contained in $\W^N_x$.
                            \item[(iii)] There are no arcs between $\W^N_x$ and $\W^N_y$
                              and $\W^N_x\cap \W^N_y = \emptyset$, for all $x,y\in \Spe$ distinct.
          \item[(iv)] If $\W^N_x \neq \emptyset$, then $\V^N_x\neq \emptyset$
        \end{description}
\end{lemma}
\begin{proof}
	Property (i) is trivially satisfied, since $x\not\prec_{N} x$.

  For Property (ii), assume for contradiction that there exists some $v\in
  \W^N_x$ and some vertex $u\in W$ with $u\prec_N v$ and $u\neq x$ that is not
  contained in $\W^N_x$. Then $L_{N}(u) \neq \{x\}$ and since $v\succ_N u$ also
  $L_{N}(v) \neq \{x\}$; a contradiction. 
   
	We continue with Property (iii). In view of Lemma~\ref{lem:new}(ii), there are
	no arcs $(u,v)\in F$ with $u\in \W^N_x$ and $v\notin \W^N_x$ and thus, in
	particular, no arcs $(u,v)\in F$ with $u\in \W^N_x$ and $v\in \W^N_y$ for all
	distinct $x,y\in \Spe$. By construction, $\W^N_x\cap \W^N_y = \emptyset$,
  for all distinct $x,y\in \Spe$.

	For Property (iv), assume for contradiction that $\W^N_x \neq \emptyset$ but
	$\V^N_x = \emptyset$. Hence, there is no arc $(z,w)\in F$ with $z\notin
	\W^{N}_x$ and $w\in \W^{N}_x$. Thus, $\rho_N\in \W^N_x$; a contradiction to
	$\rho_N\succ_N z$ for all $z\in \Spe$ and $|\Spe|>1$.  
\qed
\end{proof}

\begin{figure}[tbp]
  \begin{center}
    \includegraphics[width=.8\textwidth]{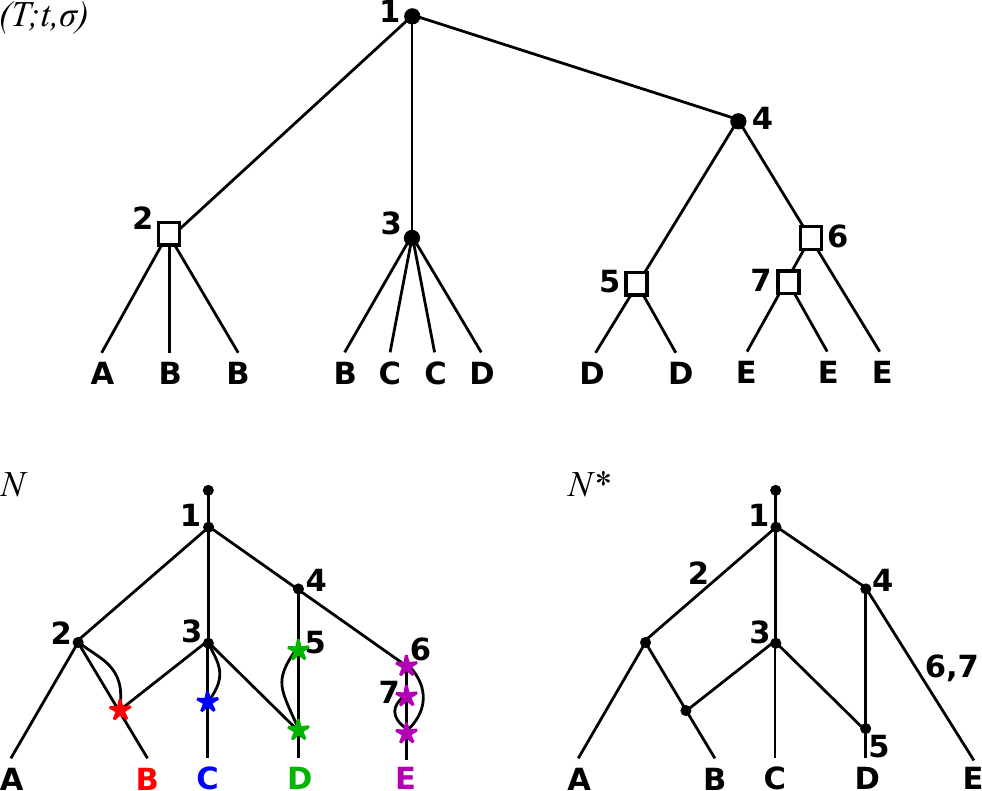}
  \end{center}
  \caption{
In the upper part, an event-labeled gene tree $(T ;t, \sigma )$ is shown where, for simplicity, 
all leaf labels $v\in \Gen$ are replaced by $\sigma(v)\in \Spe =\{A,B,C,D,E\}$.
	\rev{Speciation and duplication events 
	are represented as $\bullet$ and $\square$, respectively. } 
      In the lower-left part, the network $N=(W,F)$
      obtained from the simple subdivision $(M',\chi)$
      of $(M(T ;t, \sigma); \chi)$ as in  Definition \ref{def:MULfoldN}
      (ignoring the ``{\large$\mathbf{\star}$}''-labels 
			of the non-leaf vertices).	
In the lower-right part, the network $N^*$ constructed from $N$ 
as specified in Definition~\ref{def:Nstar}.
Using the vertex labels in $(T;t,\sigma)$, we indicate 
the TreeNet-reconciliation map $\mu^*$ from $(T;t,\sigma)$ to $N^*$
as specified in the proof of Proposition~\ref{prop:reconcNstar}
in terms of \rev{the} $\mu^*$-images on the respective vertices and arcs in $N^*$.
}
	\label{fig:multiarc-free}
\end{figure}

In what follows, we wish to modify a network $N$ with multi-arcs to a
network $N^*$ without multi-arcs.  
To this end, 	we will replace entire subgraphs of $N$ 
by specified arcs or vertices which eventually leads to the multi-arc free network $N^*$.	
We give a formal description of the approach to first construct a DAG
$N^*$ from a given  network $N$. As we shall see in Proposition \ref{prop:multi-arcFree}, 
this DAG $N^*$ is indeed a multi-arc free network.

\begin{defi}
	Let $N=(W,F)$ be a network on $\Spe$. 
	The DAG $N^*$ is obtained from $N$ as follows:
	\begin{description}
	\item  First, for all $x\in \Spe$  with $\W^{N}_x\neq \emptyset$ 
	and $|\V^{N}_x|=1$, remove all vertices in $\W^{N}_x$ and all arcs in $F$ incident to 
	vertices in $\W^{N}_x$ from $N$ and add the arc $(z,x)$ to $N$ with $z\in\V^{N}_x$. 
	\item  Second, for all $x\in \Spe$  with $\W^{N}_x\neq \emptyset$ and  $|\V^{N}_x|>1$, remove all 
				vertices in $\W^{N}_x$ and all arcs in $F$ incident to vertices in $\W^{N}_x$ from $N$
				and add  a new vertex $w_x$ and one arc $(z,w_x)$ for all $z\in \V^{N}_x$  and 
				the arc $(w_x,x)$ to $N$.
	\end{description}
	\label{def:Nstar}
\end{defi}

The sets $\W^N_x$ and $\V^N_x$ as well as the construction of $N^*$ 
are illustrated in Fig.\ \ref{fig:multiarc-free}.

\begin{proposition}
 Let $(T;t,\sigma)$ be an event-labeled gene tree on $\Gen$ and let $N$ be a
 species network on $\Spe$ obtained from the simple subdivision of the MUL-tree
 $(M(T;t,\sigma),\chi)$ as in Definition \ref{def:MULfoldN}.
\rev{Then $N^*$ obtained from $N$ by Def.\ \ref{def:Nstar} 
		 is a multi-arc free species network on $\Spe$ 
		 and the construction of $N^*$ can be done in polynomial time. 
} 
 \label{prop:multi-arcFree}
\end{proposition}
\begin{proof}
  In what follows, let $T=(V,E)$, $N=(W,F)$ and 
	 $M(T;t,\sigma) = (D,U)$.
	To help keep notation at bay, we assume for simplicity that
  $V = D\setminus (\{\rho_{M(T;t,\sigma)}\}\cup \Gen)$.

	Now, we construct the DAG $N^*$ from $N$ as in
        Definition~\ref{def:Nstar}. 
	Clearly, $N^*$ has leaf set $\Spe$. 
	To show that $N^*$ is a multi-arc free network, we first 
	analyze for some $x\in \Spe$ the sets $\W^N_x$ and $\V^{N}_x$. Note, if 
 $\W^N_x \neq \emptyset$, then Lemma \ref{lem:new}(iii) implies that $\V^{N}_x\neq \emptyset$, 
	Thus, we examine the following three mutually exclusive cases:
  \begin{description}
		\item[(I)] 	$\W^N_x = \emptyset$,	
		\item[(II)] $\W^N_x \neq \emptyset$ and $|\V^{N}_x|=1$, or
		\item[(III)] $\W^N_x \neq \emptyset$  and $|\V^{N}_x|>1$.
	\end{description}

   We interrupt the proof of the proof of the proposition to
   illustrate these cases by means of the bottom left network in
   Fig.\ \ref{fig:multiarc-free}. Then the
   set $\W^N_A = \emptyset$ satisfies Case (I).
   For $B,D\in \Spe$, we have $\W^N_B \neq \emptyset$ and
   $\V^N_B = \{2,3\}$ and 	$\W^N_D \neq \emptyset$ and
   $\V^N_D = \{3,4\}$ and hence, Case (III) is satisfied. For
   $C,E\in \Spe$, we have  $\W^N_C \neq \emptyset$ and $\V^N_C = \{3\}$ and 
$\W^N_E\neq \emptyset$ and $\V^N_E = \{4\}$ and hence, Case (II) is satisfied.
   
   We continue with the proof of the proposition by
   taking a closer look at potential multi-arcs in $N$. 
	By construction of $N$ and since $(T;t,\sigma)$ has no multi-arcs, 	
 	there are exactly $\ell\geq 2$ multi-arcs between some vertices $u$ and $v$ in $N$ with $v\prec_N u$
	if and only if there are  exactly $\ell$ arcs $e_1,\dots,e_{\ell}$ incident to $u$ in $(T;t,\sigma)$ with $\sigma(h_T(e_i)) = x$, $1\leq i \leq \ell$.
	Put differently, if there are multi-arcs between two vertices $u$ and $v$ in $N$ with
	$v\prec_N u$, then, we have in $N$ that $v=\parent(x)$ for some
        $x\in\Spe$ and the outdegree of $\parent(x)$ is one. In particular,
        $\parent(x)\in
	\W^N_x$	which implies $\W^N_x\neq \emptyset$. 
	We summarize the latter observation for
	\begin{owndesc}
	\item[\textnormal{\em Case (I):}]
	 Let $x\in \Spe$ such that 	$\W^N_x = \emptyset$. Then the
	 construction of $N$ from $M(T;t,\sigma)$
         implies that there are no multi-arcs in $N$ between the (unique) parent $\parent(x)$ of $x$ 
	and any of the vertices $y$ that are parents of $\parent(x)$.
	Therefore, if $\W^N_x = \emptyset$ for all $x\in \Spe$, then $N$ is a network
	without multi-arcs. In this case, we put $N^*=N$.   
	\end{owndesc}

	\begin{owndesc}
	\item[\textnormal{\em Case (II):}]
	  Let $x\in \Spe$ such that
          $\W^{N}_x = \{w_1,\dots, w_n\}$, $n\geq 1$,
		and  $\V^{N}_x = \{z\}$. According to Definition~\ref{def:Nstar},
		we remove all vertices in $\W^{N}_x$ and all arcs in $F$ incident to vertices in $\W^{N}_x$ from $N$
		and add the arc $(z,x)$ to $N$ to obtain an acyclic digraph $N' = (W',F')$.
		By Lemma \ref{lem:new}(ii),  there is no arc $(w_i,u)$ in $N$ with
		$u\notin \W^{N}_x$ and $u\neq x$. Thus, the ancestor relationship between any vertices in $N$ that are still
		contained in $N'$ (and thus, not contained in  $\W^{N}_x$) is preserved. Put differently, 
		if $u,v\in W\cap W'$ with $u\preceq_{N} v$ then $u\preceq_{N'} v$. 

		We next show that $N'$ is a network and thus,
                satisfies Properties~(N1), (N2) and (N3). 
	  By construction, Property~(N1) is clearly satisfied for $N'$.
		Furthermore, Property~(N2) 
	  clearly holds for all $y\in \Spe\setminus\{x\}$. 
	  Since	 all vertices $\W^{N}_x$ have been removed and only the arc
          $e^x$          
	  has been added, vertex $x$ is an outdegree-0 vertex in $N'$.
		Therefore, $N'$ satisfies Property~(N2).

		To see Property~(N3), observe first that 
		the degrees of the vertices in $N'$ that are not
                incident to vertices in $\W^{N}_x$ remain the same as in $N$.
		Thus, $N'$ satisfies Property (N3) for all such vertices.
		Moreover, since all vertices $\W^{N}_x$ have
                been removed and only the arc $e^x$
		has been added, it remains to analyze the degree of
                the vertex $z:=h_N(e^x)$ in $N'$. In the context of this,
                we claim that $z$ is a tree vertex of $N'$.
		To see this, note that in $N$, vertex $z$ cannot
                have indegree greater than one, 
		as otherwise, the outdegree of $z$  must be one in $N$ and thus,
                $z\in \W^{N}_x$; a contradiction. 
		Thus, the indegree of $z$ is one in $N$ and, by construction, 
		the indegree of $z$ remains one in $N'$. Since $N$ satisfies (N3), 
		the outdegree of $z$ is greater than one in $N$. However,
                not all children of $z$ in $N$
		can be contained in $\W^{N}_x$ as otherwise, $z\in \W^{N}_x$; a contradiction. 
		Thus, there is a child $z'$ of $z$ with $z'\notin \W^{N}_x$. Since 
		the removal of $\W_N^x$ and the respective incident arcs as well as 
		the addition of the arc $e^x$ 
		does not affect any arc between $z$ and and its
                children $z'$ with $z'\not\in\W^{N}_x$, it follows
                that the arc(s) between $z$ and $z'$ remain in $N'$. 
    Thus, $z$ still has outdegree greater than one in $N'$. 
		Moreover, there is no vertex $w_i\in \W^N_x$ with $w_i\succ_N z$, as otherwise, 
		Lemma \ref{lem:new}(ii) implies that	$z\in \W^N_x$; a contradiction. 
		Hence, the unique arc  $(\parent(z),z)$ has not been removed and still
		exists in $N'$. Therefore, $z$ has indegree one in $N'$.
		Consequently, $z$ is a tree vertex in $N'$, as claimed.  
		In summary, $N'$  satisfies (N1)-(N3) and thus, remains a network. 
	
  	        By construction and the latter arguments, $x$ has exactly one
                parent $\parent(x) = z$
		in $N'$ and  $z$ has at least one child $z'$ with
                $z'\not\in\W^{N}_x$
		and this child $z'$ remains in $N'$. Therefore,
                $\W^{N'}_x=\emptyset$ in $N'$. 
		In addition, Lemma \ref{lem:new}(iii) implies that
                $\W^N_a\cap \W^N_b = \emptyset$
		and that there are no arcs between $\W^N_a$ and $\W^N_b$
                for all distinct $a,b\in \Spe$. 
		Therefore,  $\W^{N'}_a = \W^{N}_a$ for all
                $a\in \Spe\setminus \{x\}$ in $N'$. 
		Put differently, the sets $\W^{N}_a$ remain unchanged in $N'$ for all $a\in \Spe\setminus \{x\}$.
\smallskip

	\item[\textnormal{\em Case (III):}]
		Let	$x\in \Spe$ such $\W^{N}_x\neq \emptyset$ and  $|\V^{N}_x|>1$.
		According to Definition~\ref{def:Nstar},
		we remove all vertices in $\W^{N}_x$ and all arcs in $F$ incident to vertices in $\W^{N}_x$ from $N$
		and add a new vertex $w_x$ and one arc $(z,w_x)$ for all $z\in \V^{N}_x$  and 
		the arc $(w_x,x)$ to $N$ to obtain an acyclic digraph $N'$. 
		Note that, as in Case (II), the ancestor relationship
                between any two vertices in $N$ that are also
		contained in $N'$ is preserved.

		We show next that $N'$ is a network. By construction, Properties (N1) and (N2) are satisfied for $N'$. 
		It remains to show that $N'$ satisfies Property (N3). 
		Since all vertices $\W^{N}_x$ have been removed and new arcs have been added only
  	between vertices in $\V^{N}_x$, $w_x$ and $x$, we can conclude that 
		if $z$ is not adjacent with a vertex in $\W^{N}_x$, then the indegree and outdegree of 
		$z$ in $N$ is the same as the indegree and outdegree of $z$ in $N'$. 
		So assume that $z\in \V^{N}_x \cup \{w_x\}$. 
		
                Assume first that $z \in \V^{N}_x$. We claim again
                that $z$ is a tree vertex of $N'$.
		As observed in 	 Case (II),  any  vertex $z \in \V^{N}_x$ must have
		indegree one in $N$ and there is a child $z'$ with $z'\notin \W^{N}_x$. 
		The construction of $N'$ does not affect any arc between $z$ and $z'$. 
		Thus, every arc between $z$ and a child $z'\notin \W^{N}_x$ is also an arc in $N'$.
		Combined with the fact that, by construction, we have added the arc $(z,w_x)$, it follows that
		$z\in \V^{N}_x$ has outdegree greater than one in $N'$.
		Moreover, there is no vertex $w_i\in \W^N_x$ with $w_i\succ_N z$, as otherwise, Lemma \ref{lem:new}(ii) implies that 
		$z\in \W^N_x$; a contradiction. 
		Hence, the unique arc  $(\parent(z),z)$ has not been removed and is also an arc $N'$. 
		Therefore, $z$ has indegree one in $N'$. Thus, $z$ is a tree vertex in $N'$, as claimed. 
	
		Finally, assume that $z=w_x$. Then, by construction, the indegree of $z$ in $N'$ is $|\V^{N}_x|>1$ in $N'$
		and the unique child of $z$ is $x$. Thus, $z$ is a hybrid vertex in $N'$. 
		In summary, $N'$  satisfies (N1)-(N3) and thus, is a network.  		

		It is easy to see that $\W^{N'}_x = \{w_x\}$, where $w_x$ is not contained in any multi-arcs. 
		Note also that, as in Case (II), we have  $\W^{N}_a = \W^{N'}_a$ for all 
		$a\in \Spe\setminus \{x\}$.
	\end{owndesc}

In the latter construction, we modified $N$ in Cases~(II) and (III)
for a specific vertex $x\in \Spe$ to obtain a network $N'$. 

We complete the proof by associating a species network $N^*$ to $N$ as follows. 
Bearing in mind Case~(I), we first apply to all vertices $x\in \Spe$ for
which $\W^N_x \neq \emptyset$ and $|\V^{N}_x|=1$ holds, one after another, the
construction described in Case~(II). This yields a 
network $N'$ such that $\W^{N'}_x=\emptyset$ for all vertices $x$ in $N$
that satisfy Case~(II). Hence, all such ``Case (II)'' vertices of $N$ 
satisfy Case (I) in $N'$. Moreover, in each single modification step,
the sets $\W^{N'}_a$ have remained unchanged for all vertices
$a\in \Spe$ that have not been considered thus far. 
Applying, one after another, the construction described in
Case~(III) to all $x\in \Spe$ for which $x\in\W^{N'}_x\neq \emptyset$
and $|\V^{N'}_x|>1$ holds results in the digraph $N^*$ as constructed
in Definition~\ref{def:Nstar}. 
Arguing for all vertices $x$ that satisfy
the conditions of Case~(III) as in the proof for the ``Case~(II)''
vertices of $N$,
implies that $\W^{N^*}_x=\{w_x\}$ in $N^*$ and \rev{that} $w_x$ is not contained
in any multi-arc.

In summary, none  of the applications
of the constructions described in the proofs of
Cases~(II) and (III), respectively, 
introduces a multi-arc. Moreover, every $x\in \Spe$ satisfies 
$\W^{N^*}_x=\emptyset$ or $\W^{N^*}_x=\{w_x\}$ such that 
$w_x$ is not contained in any multi-arc. This, and the arguments 
preceding the discussion of Case (I) in the proof, imply
that  $N^*$ is a multi-arc free network.

Finally, Corollary~\ref{cor:Nexists} implies that the network $N$ can
be constructed in polynomial-time. 
Moreover, it is \rev{easy} to see that the sets $\W^N_x$ and $\V^N_x$, $x\in\Spe$,
as well as the construction steps carried out
in the Cases~(II) and Case (III) to transform $N$ into $N^*$
can be performed in polynomial-time. 
Hence, $N^*$ can be obtained from $(T;t,\sigma)$ in polynomial-time.
\qed
\end{proof}

Now, let $(T;t,\sigma)$ be an event-labeled gene tree, and let 
$N=(W,F)$ be the network associated to the simple subdivision $M'$ of 
$M=M(T;t,\sigma)$ as detailed in Definition \ref{def:MULfoldN}.
Let $N^*=(W^*,F^*)$ be the network without
multi-arcs obtained from $N$ by the constructions
detailed in Definition~\ref{def:assoMUL} and
Proposition~\ref{prop:multi-arcFree}. 
As argued in the proof of Corollary~\ref{cor:Nexists}, 
there is always a \rev{TreeNet-}reconciliation map $\mu_{\kappa',f}$ from 
 $(T;t,\sigma)$ to $N$. The proof is, in particular, based on the fact that 
there is a folding map $f$ from $M'$ to $N$.
However, such a folding map may not exist for $M'$ and $N^*$. 
Therefore, we will slightly adjust the map
$\mu_{\kappa', f}$ 
to obtain a TreeNet-reconciliation map $\mu^*$ from $(T;t,\sigma)$ to $N^*$.

To this end, we	partition the vertex set of the original network
$N = (W,F)$ as follows. 
\begin{defi}
Let $N = (W,F)$ be a network on $\Spe$ and let $\Spe'$ be the set of all $x\in \Spe$ with $\W^N_x \neq \emptyset$ holding. 
Put
\[W_1 = W\setminus W_2 \text{ and } W_2 = \bigcup_{x\in\Spe'} (\W^N_x \cup \{x\})\]
and 
\[F_1 = \{a\in F \mid h_N(a),t_N(a)\in W_1\}.\]
\label{def:W1}
\end{defi}

Clearly, $W_1$ and $W_2$ form a partition of $W$. 
By Lemma~\ref{lem:new}(iii), there are no
arcs $(u,v)\in F$ with $u\in \W^N_x$ and $v\notin \W^N_x$ 	
and 	$\W^N_x\cap \W^N_y = \emptyset$ for all distinct $x,y\in \Spe$.
As an immediate consequence, we obtain the following

\begin{observe}
The	subgraphs $N[W_1]$ of $N$ and $N^*[W_1]$ of $N^*$ induced by $W_1$
coincide, i.e., $N[W_1] = N^*[W_1]$.
\end{observe}

\begin{defi}
	Let  $(T=(V,E);t,\sigma)$ be an event-labeled gene tree on $\Gen$ 
	and let $N=(W,F)$	be a network on $\Spe$ such that there is a 
	TreeNet-reconciliation map $\mu$ from $(T=(V,E);t,\sigma)$ to $N$. Moreover, let
  $N^*=(W^*,F^*)$ be the 
	multi-arc free network on $\Spe$ as in Definition~\ref{def:Nstar} and
	let $W_1$ and $F_1$ be defined for $N$ as in Definition \ref{def:W1}. 
	The map $\mu^*:V\to W^*\cup F^*$ (w.r.t.\ $\mu$) is defined \rev{for all $v\in V$}
        as follows:
	
\rev{	\begin{itemize}
	\item If $v\in \Gen$, or $t(v)=\bul$, or $t(v)=\squ$ and $\mu(v)\in F_1$, then 
   put $\mu^*(v)\coloneqq \mu(v)$.
	\item  Otherwise, there must exist a leaf 
	$x\in \Spe$ such that either $h_N(\mu(v))\in \W^N_x$ or $h_N(\mu(v))=x$, 
	and we put $\mu^*(v)\coloneqq  (\parent(x),x)$.
	\end{itemize} 
}
	\label{def:mustar}
\end{defi}

Consider  the event-labeled tree $(T;t,\sigma)$ depicted in
Fig.~\ref{fig:multiarc-free} and the network $N^*$ constructed from
$(T;t,\sigma)$ as specified in Proposition~\ref{prop:reconcNstar}. The map
$\mu^*$ is indicated in the bottom right network in
Fig.~\ref{fig:multiarc-free}. In what follows, we show that $\mu^*$ is
well-defined and, in particular, a TreeNet-reconciliation map from $(T;t,\sigma)$ to $N^*$.

\begin{proposition}
  Let  $(T=(V,E);t,\sigma)$ be a \rev{well-behaved} event-labeled gene tree on $\Gen$,
\rev{let}
	$N$ be the network on $\Spe$ associated to the simple subdivision of 
	$(M(T;t,\sigma),\chi)$ as in Definition \ref{def:MULfoldN}, and \rev{let} $N^*$ be the 
	multi-arc free network on $\Spe$ as in Definition~\ref{def:Nstar}.

	Then, \rev{	$\mu^*$ as in Def.\ \ref{def:mustar} } 
	is a TreeNet-reconciliation map from $(T;t,\sigma)$ to $N^*$. 
\label{prop:reconcNstar}  
\end{proposition}
\begin{proof}
	In what follows, let $N=(W,F)$, $N^*=(W^*,F^*)$ and put 
	$M\coloneqq M(T;t,\sigma) = (D,U)$.
	Moreover, let $W_1$ and $F_1$ be defined for $N$ as in Definition \ref{def:W1}. 

	Corollary~\ref{cor:Nexists} implies that there is always a
        \rev{TreeNet-}reconciliation map 
	$\mu\coloneqq \mu_{\kappa',f}$ from $(T;t,\sigma)$ to $N$. In what follows, we show that 
	$\mu^*$ (w.r.t.\ $\mu$) as in Definition \ref{def:mustar} is a
        \rev{TreeNet-}reconciliation map from $(T=(V,E);t,\sigma)$ to $N^*$.

   Note, there may be arcs and vertices in $N$ that have been removed
   to obtain $N^*$
and new vertices and arcs may have been added in the construction of $N^*$. 
Hence, in order to show that $\mu^*$ is well-defined, we must ensure that
either $\mu(v)$ is still contained in $N^*$ and that, otherwise, 
if we assign  $\mu^*(v)=(\parent(x),x)$ the condition $h_N(\mu(v))\in \W^N_x$ or $h_N(\mu(v))=x$
is satisfied for this leaf $x\in \Spe$.

Clearly, if $v\in \Gen$, then $\mu^*(v) = \mu(v) = \sigma(v)$ is well-defined. 

Now suppose that $v\in V^0$ with $t(v)=\bul$. 
We claim that $\mu(v)\in W_1$. To see this, note
that $v$ must have two distinct children $v'$ and $v''$ in $T$ 
for which
$\sigma(L_T(v'))\neq \sigma(L_T(v''))$
\rev{holds as $(T;t,\sigma)$ is well-behaved.}
Thus, $|\sigma(L_T(v))|>1$. 
By Properties~(R1) and (R3) for $\mu$, we have $\mu(v)\succ_N x$
for all $x\in \sigma(L_T(v))$. Hence, $|L_N(\mu(v))|>1$. 
Thus, $\mu(v)\in W$ is an inner vertex and 
there exists no  $x\in \Spe$ such that
$\mu(v)\in \W^N_x$. Thus, $\mu(v)\in W_1$, as claimed.
Since none of the vertices in $W_1$ have been removed from $N$ to obtain $N^*$, 
$\mu(v)$ is still contained in $N^*$ for all speciation vertices
$v$ of $(T;t,\sigma)$
and we can put $\mu^*(v) =  \mu(v)\in W^*$. 

Now suppose that $v\in V^0$ with $t(v)=\squ$.
If $\mu(v)= (u,w)\in F_1$, then $N[W_1] = N^*[W_1]$ implies that 
the arc $(u,w)$ still exists in $N^*$ and we can put $\mu^*(v) = \mu(v)\in F^*$. 
If $\mu(v)= a = (u,w)\notin F_1$, then $u\notin W_1$ or $w\notin W_1$.
If $u\notin W_1$, then $u\in \W^N_x$ for some $x\in \Spe$
and Lemma \ref{lem:new}(ii) implies that either $w\in  \W^N_x$ or $w=x$. 
In either case, $w$ is always contained in $W_1$ and hence, 
$w\in  \W^N_x$ or $w=x$ (and thus, $a=(\parent(x),x)$) for some $x\in \Spe$. 
In this case, we put $\mu^*(v) =	(\parent(x),x) \in F^*$. 

In summary, it follows that $\mu^*$ is well-defined, 

To see that $\mu^*$ is a reconciliation map from $(T;t,\sigma)$
to $N^*$, we need to show that Properties~(R1) - (R3) hold. Since
Property~(R1) clearly holds as
$\mu(v)=\sigma(v)$ for all $v\in \Gen$, it suffices to restrict attention to
Properties~(R2) and (R3).
In what follows, let  $V_{\mu} \subseteq V$ be the set of vertices $v\in V$
such that $\mu^*(v) = \mu(v)$.

We start with establishing Property~(R3).
	We first recap, that the  ancestor relationship between the vertices in $W_1 \cup \Spe$ 
	has not been changed in $N^*$, that is, $w,w'\in W_1 \cup \Spe$ with $w\prec_{N} w'$
        (resp.\ $w=w'$) in $N$
	implies $w\prec_{N^*} w'$ (resp.\ $w=w'$) in $N^*$. The latter also implies that 
  the relative order of the arcs $(a,b)\in F_1$	under $\prec_N$
	has not been changed in $N^*$.

To see Property~(R3), suppose that $u,v\in V$ with $v\prec_T u$.

To see Property (R3.i),  assume that $t(v)=t(u)= \squ$. We need to show
that $\mu^*(v)\preceq_{N^*} \mu^*(u)$. 
Clearly, if $u,v\in  V_{\mu}$ then $\mu^*(v)=\mu(v)\preceq_{N} \mu(u)=\mu^*(u)$
as $\mu$ satisfies Property~(R3.i). Since the
relative order of the arcs in $F_1$ has not been changed by the
construction of $N^*$
it follows that $\mu^*(v) \preceq_{N^*} \mu^*(u)$. Hence,  Property~(R3.i)
holds in this case.

So assume that $u,v\not\in  V_{\mu}$.
Hence, $\mu(u)= a \in F\setminus F_1$ and, as argued above,  
$h_N(a)\in  \W^N_x$ or $a=(\parent(x),x)$ for some $x\in \Spe$.
Note, $\mu(v)\preceq_N \mu(u)=a$ as $\mu$ satisfies Property~(R3).
This combined with Lemma~\ref{lem:new}, implies that 
there cannot be an arc $(w,z)$ in $N$ with 
$w\in \W^{N}_x$ and $z\notin	\W^{N}_x$ or $z\neq x$.
Thus, for  $a'=\mu(v)$ we have $t_N(a')\in \W^N_x$ 
and therefore, $h_N(a')\in \W^N_x$ or $h_N(a')=x$.
By construction of $\mu^*$ we have, therefore, $\mu^*(u) = \mu^*(v) = (\parent(x),x)$.
Hence,  Property~(R3.i) holds in this case.

Assume next that one of $u$ and $v$ is contained in $V_{\mu}$
whereas the other is not. Note that if $u\notin V_{\mu}$, then 
$\mu(u) \in F\setminus F_1$. Similar arguments
as in the latter case imply $\mu(v) \in F\setminus F_1$.
Hence,  $v\notin V_{\mu}$. Therefore, 
$u\in V_{\mu}$ and $v\notin V_{\mu}$ must hold. 
Thus, $\mu(u) = \mu^*(u) \in F_1\subseteq F^*$
and $h_N(\mu(v))\in \W^N_x$ or $\mu(v)=	(\parent(x),x)$ for some
$x\in \Spe$.
Hence,  $\mu^*(v)= (\parent(x),x)$.  
Since $\mu$ satisfies Property (R3.i) in $N$ we have
$x\prec_N\mu(v)\preceq_{N} \mu(u) = \mu^*(u)$ and, therefore, 
that $x\prec_{N^*}\mu^*(v) \preceq_{N^*} \mu^*(u)$. 
In combination, we obtain that Property~(R3.i) is satisfied for all
duplication vertices of $N$.

To see that Property~(R3.ii) also holds, assume that
at least one of $t(u)=\bul$ and $t(v)=\bul$ holds.
If $t(v)=\bul$ then $v\in V_{\mu}$. Again, this implies that 
$u\in V_{\mu}$, since $v\prec_T u$, $\mu$ satisfies Property~(R3.ii)  
and Lemma \ref{lem:new}(ii) holds.
Moreover,   
Thus, $\mu^*(v)=\mu(v)\prec_N \mu(u) =\mu^*(u)$. 
Since $u,v\in V_{\mu}$ and the ancestor relationships in $N$ are preserved in $N^*$
for all vertices in $W_1\cup \Spe$, 
it follows that $\mu^*(v)\prec_{N^*}\mu^*(u)$.

So assume that $t(u)=\bul$. Then $u\in V_{\mu}$. Note that we may
assume w.l.o.g.\ that $t(v)=\squ$ as otherwise $\mu^*(v)=\mu(v)$
and $\mu^*(u)=\mu(u)$ must hold. Thus, similar arguments
as before imply that $\mu^*(v)\prec_{N^*}\mu^*(v)$. Note also that
we may assume that $\mu(v)\not\in F_1$ as otherwise we have again
that $\mu^*(v)=\mu(v)$ and $\mu^*(u)=\mu(u)$ which, in turn, implies
$\mu^*(v)\prec_{N^*}\mu^*(v)$.  Then
$h_N(\mu(v))\in\W_x^N$ or there exists some $x\in \Spe$
such that $\mu(v)=(\parent(x),x)$. Hence, $\mu^*(v)=(\parent(x),x)$.
Thus, $x\prec_{N}\mu(v)\prec_N \mu(u)=\mu^*(u)$.
Since the ancestor relationships in $N$ are preserved in $N^*$
for all vertices in $W_1\cup \Spe$ and since $\mu^*(v)=(\parent(x),x)$
is the lowest possible choice for $\mu^*(v)$, 
it follows that $x\prec_{N^*} \mu^*(v)\prec_{N^*}\mu^*(u)$. This concludes the proof of
Property~(R3.ii) and, thus the proof of Property~(R3).

It remains to show that Property~(R2) holds. 
Clearly, by construction of $\mu^*$,  Property (R2.ii) is satisfied in $N^*$. 

To see Property~(R2.i), assume that $v\in V$ with $t(v)=\bul$ and children
$v_1,\dots,v_k\in V$, $k\geq 2$. 
We need to show that $\mu^*(v)\in Q_{N^*}^2(\mu^*(v_1),\dots, \mu^*(v_k))$.
Since $\mu$ is a \rev{TreeNet-}reconciliation map from $(T;t,\sigma)$ to $N$, we
clearly have 
$\mu(v)\in Q_{N}^2(\mu(v_1),\dots, \mu(v_k))$. 	
Thus, there exist $i,j\in \{1,\ldots, k\}$ such
that $\mu(v_i)$ and $\mu(v_j)$ (note that $\mu(v_i) = \mu(v_j)$
might be possible) are separated by $\mu(v)$ in $N$. 
Therefore, there exists a vertex
$w'\in W$ such that $(\mu(v),w')$ is the first arc
on a directed path from $\mu(v)$ to  $\mu(v_i)$
in $N$. Similarly, there exists a vertex
$w''\in W$ such that $(\mu(v),w'')$ is the first arc
on a directed path from $\mu(v)$ to  $\mu(v_j)$ in $N$.
Note that $w'=w''$ or $w'\neq w''$ might hold. In the first case, we might
have multi-arcs between $\mu(v)$ and $w'$ in $N$. 

Assume first that $w'\neq w''$. 
We distinguish between the cases that $w',w''\in W_1$ (Case (C1)) and that at
least one of $w'$ and $w''$ is not contained in $W_1$ (Case (C2)).
\begin{owndesc}
		\item[\textnormal{\em {Case (C1):}}]
		Since $w', w''\in W_1$, we have for all $x\in\Spe$ that 
	$w',w''\notin \W^N_x$. 
	        By construction of $N^*$, both vertices $w'$ \rev{and}
                $w''$ are also contained in $N^*$. 
	By the definition of $\mu^*$, we have 
	that $\mu^*(v_i) = \mu(v_i)$ in case $v\in \Gen$, or $t(v)=\bul$,
	or $t(v)=\squ$ and $\mu(v_i)\in F_1$, and that, otherwise,
	there exists some $x\in \Spe$ such that
	$\mu^*(v_i) = (\parent(x),x)$. 
	Since the ancestor relationship of all vertices in $W_1\cup \Spe$ is preserved in $N^*$
	it follows that $(\mu(v), w') \succeq_{N^*} \mu^*(v_i)$. Put differently,
	there exists a directed path from $\mu^*(v) = \mu(v)$ to $\mu^*(v_i)$ in $N^*$
	that contains the arc $(\mu^*(v), w')$. Note that $\mu^*(v_i) = \mu(v_i) =
	(\mu^*(v), w')$ might hold. Similarly, there exists a directed path from
	$\mu^*(v)$ to $\mu^*(v_i)$ in $N^*$ that contains the arc $(\mu^*(v), w'')$.
	Hence, $\mu^*(v)$ separates $\mu^*(v_i)$ and $\mu^*(v_j)$ in $N^*$.
	Therefore, $\mu^*(v)\in Q_{N^*}^2(\mu^*(v_1),\dots, \mu^*(v_k))$.
	\smallskip

	\item[\textnormal{\em {Case (C2):}}]
	W.l.o.g.\ assume  that $w'\notin W_1$.
	Thus, either $w' \in \W^N_x$ or $\W^N_x\neq \emptyset$ and $w'=x$, for some $x\in \Spe$. 	
	We claim that, in $N^*$,   	either  $\parent(x) = \mu(v)$ or  $\parent(x) = w_x$ holds, where $w_x$ is the unique vertex
	added by replacing $\W^N_x$ as in Definition~\ref{def:Nstar}.
	To see this, note first that since
        all vertices in $\W^N_x$ and their incident arcs were
        removed from $N$ and either the arc $(\mu(v),x)$ or the
        two arcs $(\mu(v),w_x)$ and $(w_x,x)$
	were added to obtain $N^*$ it follows that 
	either  $\parent(x) = \mu(v)$ or  $\parent(x) = w_x$, as claimed.
	Furthermore, since there is a directed path from
  $\mu(v)$ to  $\mu(v_i)$  in $N$ with arc  $(\mu(v),w')$, 
	we have $\mu(v_i) = (\mu(v),w')$ or $w'\succeq_N\mu(v_i)$.
  Hence, $\mu(v_i)\in \W^N_x$ or $\mu(v_i) = x$. 

	Note that $\mu(v_i)\in \W^N_x$ or $\mu(v_i) = x$ implies in
	particular that, in $N^*$, we have
	$\mu^*(v_i) = (\parent(x),x)$ or $\mu^*(v_i) =x$.

	We next claim that there exists some $1\leq \ell\leq k$ distinct from
	$i$ such that $\mu^*(v_{\ell}) \neq (\parent(x),x)$. To see this claim 
	note first that since $\mu$ satisfies Property~(R3) and
	$\mu(u) =  \sigma(u)\in \Spe$ for all $u\in \Gen$, it
	follows that $\mu(v_i)\succeq_N \sigma(u)$ for all $u\in \Gen$
  with $v_i\succeq_T u$. Thus, if $\mu(v_i)\in \W^N_x$ then
  $\sigma(L_T(v_i)) = \{x\}$.  Since $\sigma(L_T(v_i)) = \{x\}$ also holds in
  case of $\mu(v_i) = x$ our assumption \rev{that $T$ is well-behaved}
  implies that there is some
	$1\leq \ell\leq k$ distinct from $i$ such that
	$\sigma(L_T(v_{\ell}))\neq  \sigma(L_T(v_i))=\{x\}$. 
	Let $y\in \sigma(L_T(v_{\ell}))$ with $y\neq x$. 
	Since $v_{\ell} \succeq_T u$ for some $u\in \Gen$
  with $\mu^*(u) = \sigma(u)=y$
	and $\mu^*$ satisfies Property~(R3) in $N^*$, we have 
	$\mu^*(v_{\ell})\succeq_{N^*} y$. Therefore, 
	$\mu(v_{\ell})$ cannot be contained in $\W^N_x$. 
	Hence, $\mu^*(v_{\ell}) \neq (\parent(x),x)$, as claimed.

	As shown above, $\mu^*(v_i) = (\parent(x),x)$ or $\mu^*(v_i) =x$. 
	In either case, by construction of $N^*$, the vertex $\parent(x)$
	has only one child, namely the vertex $x$. 
	Since $x\neq y$, it follows that if 
	$\mu^*(v_{\ell})\succeq_{N^*} y$ then $\parent(x) \not \succeq_{N^*} \mu^*(v_{\ell})$.   
	The latter combined with $\mu^*(v)\succ_{N^*} \mu^*(v_{\ell})$
  implies that there is an alternative directed path from $\mu^*(v)$
  to $\mu^*(v_{\ell})$ that has only $\mu^*(v)$
	in common with the directed path from $\mu^*(v)$ to $\mu^*(v_{i})$. 
	Thus, $\mu^*(v)$ separates $\mu^*(v_{i})$ and $\mu^*(v_{\ell})$
	and  $\mu^*(v)\in Q_{N^*}^2(\mu^*(v_1),\dots, \mu^*(v_k))$.
\end{owndesc}

	To finish the proof that $\mu^*$ satisfies Property~(R2.i) it remains
	to consider the case that $w'=w''$. Then $(\mu(v),w')$ and
        $(\mu(v),w'')$
	are parallel arcs.  By construction of $N^*$, 
  if there are parallel arcs between two vertices $\mu(v)$ and $w'$ in
  $N$, then $w'=\parent(x)\in\W^N_x$ must hold for some $x\in	\Spe$.
	Hence, $w' \in \W^N_x$. Similar arguments as in the 
	proof of Case (C2) imply that $\mu^*(v)\in Q_{N^*}^2(\mu^*(v_1),\dots, \mu^*(v_k))$.

	In summary,  $\mu^*$ satisfies Property~(R2). This completes the proof
	that $\mu^*$ is a \rev{TreeNet-}reconciliation map from $(T;t,\sigma)$ to $N^*$.
\qed
\end{proof}

As suggested already by Fig.\ref{fig:multiarc-free}, the network
$N^*$ obtained from $N$
as described in Definition~\ref{def:Nstar} displays all triples in
the event-labeled
gene tree in Fig.~\ref{fig:multiarc-free}. That this is not a
coincidence is the purpose of the next result.

\begin{proposition}
 Let $(T;t,\sigma)$ be an event-labeled gene tree on $\Gen$ and let $N$ be a
 species network on $\Spe$ obtained from the simple subdivision of the MUL-tree
 $(M(T;t,\sigma),\chi)$ as in Definition \ref{def:MULfoldN}. 
 Then the multi-arc free network $N^*$ on $\Spe$ displays all triples in $\mathcal{S}(T;t,\sigma)$.  
\label{prop:multi-arc-free-triples}
\end{proposition}
\begin{proof}
Put 
$M\coloneqq M(T;t,\sigma)$ and  
let $N=(W,F)$ and
$N^*=(W^*,F^*)$. 

Suppose that $A,B,C\in \Spe$ such that $AB|C \in \mathcal{S}(T;t,\sigma)$. 
Then there are three elements $a,b,c\in \Gen$ with pairwise distinct
$\sigma(a)=A$,
$\sigma(b)=B$ and $\sigma(c)=C$ such that  
the (undirected) path between $a$
and $b$ in $T$ does not intersect the directed path from $\rho_T$ to $c$ in $T$.
In particular, putting $v \coloneqq \lca_T(a,b,c)$
we have $t(v)=\bul$.  
Thus, $u \coloneqq \lca(a,b)\prec_T v$. \rev{In what follows, we write
$v^W$ for every $v\in V$ that is still contained in $W$.} 
Moreover, we say that two paths  $P(v_1,v_k)$ with $k\geq 1$ and 
$P(w_1,w_{\ell})$ with $\ell\geq 1$ \rev{in a graph}
are {\em internal vertex disjoint}, 
if $P(v_1,v_k)$ and $P(w_1,w_{\ell})$ share at most one vertex from the set
$\{v_1,v_k,w_1,w_{\ell}\}$.
	
Observe that any directed path $P = (v_1,\dots,v_k)$ in $T$ with $k\geq 1$, 
	for which the corresponding vertices $v_{1}^W,\dots,v_{k}^W$
	are contained in $W_1$, also forms  a directed path 
	$P_W\coloneqq (v^W_{1},\dots,v^W_{k})$ in $N$ and thus, in $N[W_1]$. 
	Since $N[W_1]$ and $N^*[W_1]$ coincide, it follows that $P_W$ is
        contained in $N^*$. 
	In particular, if there are internal vertex disjoint directed paths 
	$P=(v_1,\dots,v_k)$ with $k\geq 1$ and 
	$P'=(w_1,\dots,w_{\ell})$ with $\ell\geq 1$ in $T$, 
	then the directed paths $P_W=(v^W_{1},\dots,v^W_{k})$ and $P'_W=(w^W_{1},\dots,w^W_{{\ell}})$ exist in $N$. 
	Moreover,  $P_W$ and $P'_W$
        are internal vertex disjoint, whenever $v^W_{k},w^W_{{\ell}}\in W_1$,
        as this implies, 
	$v^W_{1},\dots,v^W_{k-1},w^W_{1},\dots,w^W_{{\ell}-1}\in W_1$ as well.

	Since  $|\sigma(L_T(v))|>1$ and $|\sigma(L_T(u))|>1$, we have by construction of $N$ that 
	$|L_N(v^W)|>1$ and $|L_N(u^w)|>1$. Thus, $v^W,u^W\in W_1$. 
	Let $P_1$ be the unique path from $v$ to $u$ in $T$. 
	Since $v^W,u^W\in W_1$, \rev{for any vertex $w$ of $P_1$, we have $w^W\in W_1$}. 
	Thus, the (directed) path $P^*_1$ 
	from $v^W$ to $u^W$ exists in $N$ and hence, also in $N^*$.

	Consider now the unique path $P_2 = (u,u_1,\dots u_k,a)$ 	in $T$.
 	For $\sigma(a) = A$, we distinguish between the cases that $|\chi(A)|=1$ and $|\chi(A)|>1$. 
	
	If  $|\chi(A)|=1$, then the arc $e=(u_k,A)$ is contained in the simple subdivision $M'$ of $M$
	and has not been identified 
	with any other arc as part of the construction of $N$. Since $M$ is a MUL-tree and  $|\chi(A)|=1$, it is easy to see
	that $u_k\in W_1$. It follows that $P'_2=(u^W,u^W_{1},\dots u^W_{k},
        \rev{A})$  is entirely contained in 
	$N[W_1]$, and thus, in $N^*[W_1]$. 

		So assume that $|\chi(A)|>1$. Then $e=(u_k,A)$ in $M$ was replaced  by 
	the two arcs $(u_k,v_e)$ and $(v_e,A)$ as part of the construction 
  of $M'$. Moreover, 
	the arc $(v_e,A)$ has been identified with other arcs 
	that have head in $\chi(A)$ to obtain
	the unique arc $(\parent(A),A)$ in $N$. 
	Thus, the path $P''_2 = (u^W_{0}\coloneqq u^W,u^W_{1},\dots u^W_{k},v_e,A)$
	exists in $N$. 
	Let $u^W_{i}$ be the first vertex on $P''_2$
	that is contained in $\W^N_A$. Note, $u^W\in W_1$ implies $u^W\neq u^W_{i}$.
	For the construction of $N^*$ from $N$,
        the subpath $(u^W_{{i-1}},u^W_{i}, \dots, u^W_{k},v_e,
        \rev{A})$ of $P''_2$
	has been replaced by either the arc  $(u^W_{{i-1}},A)$ 
	or by the two arcs $(u^W_{{i-1}},w_A)$ and $(w_A,A)$. 
	Thus, the last arc on the directed path $P'''_2=(u^W,u^W_{1},\dots u^W_{{i-1}},\dots,A)$ in $N^*$
	is either $(u^W_{{i-1}},A)$ or $(w_A,A)$.
	Let us denote by $P^*_2$ the respective path $P'_2$ or $P'''_2$ 
	\rev{from $u^W$ to $A$}  in $N^*$. 

	Analogously to $P_2$,   
 	by repeating the latter arguments for the unique paths $P_3 = (u,\dots,b)$, resp., $P_4 = (v,\dots,c)$ in $T$, we can construct the 
	paths $P^*_3$ \rev{from $u^W$ to $B$ in $N^*$}, resp.,  $P^*_4$ \rev{from $v^W$ to $C$ in $N^*$}. 
 		
	Since all path $P_1$,	$P_2$, $P_3$ and $P_4$ are pairwise internal vertex disjoint
	and since $A,B,C$ are pairwise distinct, 
	the construction of the respective paths $P^*_i$, $1\leq i\leq 4$ implies that
	$P^*_1$, $P^*_2$, $P^*_3$ and $P^*_4$ are pairwise internal vertex disjoint. 
	In particular, $P^*_1$ shares with  $P^*_2$, resp., $P^*_3$ only vertex $u^W$, 
	$P^*_1$ and $P^*_4$ share only vertex $v^W$, 
	$P^*_2$ and $P^*_3$ share only vertex $u^W$, 
	and $P^*_4$ does not share any vertex with 	$P^*_2$ and $P^*_3$. 
	In summary, after the suppression of vertices with in- and outdegree one in the
  minimal subgraph of $N^*$ that contains
	the paths $P^*_1, \dots, P^*_4$, we obtain a binary reduced tree $T'$
	on $\{A,B,C\}$ such that the (undirected) path between
	$A$ and $B$ does not intersect the path from $\rho_{T'}$ 
	to $C$. 
	Therefore, $AB|C$ is displayed in $N^*$. 
\qed \end{proof}

Taken Propositions  \ref{prop:multi-arcFree}, \ref{prop:reconcNstar} and 
\ref{prop:multi-arc-free-triples} together we obtain
\begin{theorem}
  Suppose $(T;t,\sigma)$ is a \rev{well-behaved} event-labeled gene tree.  
Then
there is a multi-arc free network $N^*$ for $(T;t,\sigma)$
that can be constructed in polynomial time and that displays 
all triples in $\mathcal{S}(T;t,\sigma)$. 
\end{theorem}

The next result is well-known \cite{huson_rupp_scornavacca_2010},
but using the last result we can give a simple alternative proof.

\begin{corollary}
	For any set $R$ of triples there is a multi-arc free network that displays each triple in $R$. 
\end{corollary}
\begin{proof}
	If $R=\emptyset $, the statement is trivially satisfied. 
	Hence, let $R = \{r_1, \dots,r_m\}$, $m\geq 1$,
        be a non-empty set of triples. 
	Put $L_R\coloneqq \cup_{r\in R} L(r)$.

	We construct an event-labeled gene tree $(T;t,\sigma)$ as follows:
	To each triple $r_i=x_{i1}x_{i2}|x_{i3} \in R$, $1\leq i \leq m$, we associate a
  triple $T_i=a_{i1}a_{i2}|a_{i3}$ with root $\rho_i$ making sure that 
	$L(T_i)\cap L(T_j) = \emptyset$ for all $1\leq i<j\leq m$. 
	We start with the empty graph $T=\emptyset$ and add first 
	all triples $T_1,\dots T_m$ to $T$. Next, we add a single 
	new vertex $\rho_T$ to $T$ and
	an arc $(\rho_T,\rho_i)$, $1\leq i \leq m$. By construction, 
	$T$ is a reduced phylogenetic tree.
	Note that although $x_{ir}  = x_{js}$ might hold, 
	we always have  $a_{ir}  \neq a_{js}$, $1\leq i<j\leq m$ and $r,s\in\{1,2,3\}$.
	Hence,	$L(T) = \{a_{11},a_{12},a_{13},\dots, a_{k1},a_{k2},a_{k3}\}$.
	Finally, we put $t(v)=\bul$ for all inner vertices $v\neq \rho_T$ 
	and $t(\rho_T)=\squ$ and  define the map $\sigma:L(T)\to L_R$ by putting 
	$\sigma(a_{ij}) = x_{ij}$, $1\leq i\leq m$ and $1\leq j\leq 3$. 

	By construction, $(T;t,\sigma)$ is an event-labeled gene tree
 	and each triple in $(T;t,\sigma)$ that is rooted at a speciation \rev{vertex}
        corresponds
	to one of the triples $T_i$, $1\leq i \leq m$. By the choice of $\sigma$
	we immediately obtain   $\mc{S}(T;t,\sigma)=R$.
        
Next, we construct the network $N$ 
	obtained from a simple subdivision of the MUL-tree
  $(M(T;t,\sigma),\chi)$ as described in Definition \ref{def:MULfoldN}.
	Finally, we construct	$N^*$ as in  Definition~\ref{def:Nstar}
	and apply  Proposition~\ref{prop:multi-arcFree} to conclude that 
	$N^*$ is a multi-arc free network. Moreover,
	Proposition~\ref{prop:multi-arc-free-triples} implies that 
	$N^*$ displays all triples in $R$, which completes
	the proof. 
\qed \end{proof}

\section{The Relationship Between Reconciliations for
  MUL-Trees and Networks}
\label{sec:rec-MUL}

In this section we show that, given
an event-labeled gene tree $(T;t,\sigma)$, and a network $N$, then $N$
is a species network for $(T;t,\sigma)$ if and only if
$U^*(N)$ is a pseudo-MUL-tree for $(T;t,\sigma)$ (see Theorem~\ref{important}). 
Note that the ``if"  direction of this result follows from Theorem~\ref{thm:composed}.

In what follows, let $(T=(V,E);t,\sigma)$ be an event-labeled gene tree, 
let $N=(W,F)$ be a multi-arc free species network on $\Spe$, and let $\mu$ be a 
reconciliation map between from $(T;t,\sigma)$ to $N$. 
Moreover, let $f=(f_V,f_E)$ denote a folding map from a pseudo MUL-tree 
$(M,\chi)= ((D,U),\chi)$ labeled by $\Spe$ to $N$.
By Proposition~\ref{prop:iso-mul-fold}, we may assume  w.l.o.g. that 
$(M,\chi) = (U^*(N), \chi^*)$ and that $\rho_N = \rho_M$.

\rev{In the following, we first define a map 
$\kappa_{\mu,f}:V \to (D\setminus D^1)\cup U$
associated to $\mu$ and $f$, and we then show 
that $\kappa_{\mu,f}$ is a MUL-reconciliation
from $(T;t,\sigma)$ to $(M,\chi)$ (Proposition~\ref{claim2}),
which will immediately imply Theorem~\ref{important}.
We shall define the map $\kappa_{\mu,f}:V \to (D\setminus D^1)\cup U$ in a top-down fashion.
Although the definition is quite long and technical, it follows in
a quite straight-forward fashion from the lifting properties of a
folding map (see Fig.~\ref{fig:kappamuf} for an illustrative example).}

\begin{figure}[tbp]
  \begin{center}
    \includegraphics[width=.9\textwidth]{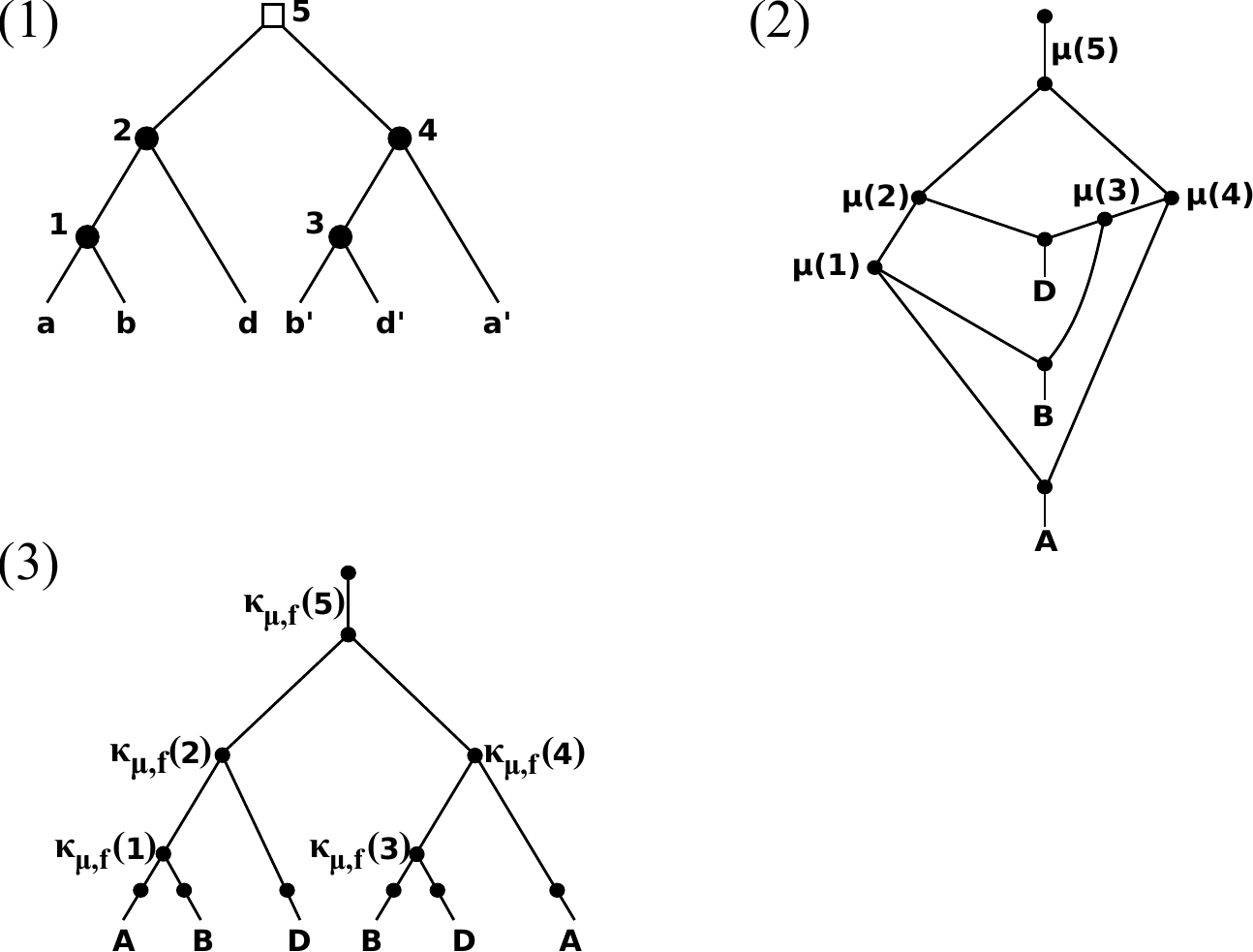}
  \end{center}
	\caption{}
	\rev{Panel (1) shows the event-labeled gene tree $(T;t,\sigma)$ from
          Fig.~\ref{fig:MULsimple-foldN} ignoring the labels of the inner vertices
          for the moment.	Speciation and duplication events 
	are represented as $\bullet$ and $\square$, respectively. } 
	Ignoring the labels of the inner vertices \rev{also in the
          remaining two panels}
        for the moment, we depict in Panel (3) the simple subdivision
  $(M',\chi)$ of $(M(T;t,\sigma),\chi)$, and the
	``fold-up'' of $(M',\chi)$ into $N$ using the folding map $f$ 
	as specified in Lemma \ref{lem:MULfoldN} (Panel (2)). 
Enumerating the inner vertices of $(T;t,\sigma)$
as indicated, the map $\mu$ is given in terms of the labels of the arcs
and vertices of $N$ is a TreeNet-reconciliation $\mu$ from $(T;t,\sigma)$
to $N$ and the map $\kappa_{\mu,f}$ represented in terms of the labels of
the arcs and vertices of $(M',\chi)$ is the MUL-reconciliation map
$\kappa_{\mu,f}$  from  $(T;t,\sigma)$ to $(M',\chi)$ as constructed in
Section~\ref{sec:rec-MUL}.
		\label{fig:kappamuf}
\end{figure}

\rev{In order to define $\kappa_{\mu,f}$ we begin with an observation.}
Let $x,y\in W$ be two distinct vertices with $y\prec_N x$ and
assume that $P(x,y)=(w_0=x,w_1,\ldots, w_k,w_{k+1}=y)$, $k\geq 0$ is
a directed path in $N$ from $x$ to $y$. For all $0\leq i\leq k$, let $a_i$
denote the arc $(w_i, w_{i+1})$ between $w_i$ and $w_{i+1}$ in $N$. Note,
that $a_i$ is well-defined, since $N$ is multi-arc free and thus, there is at most 
one arc between any two vertices of $N$. 
Observe that since $f$ is a folding from $(M,\chi) = (U^*(N), \chi^*)$ into $N$ 
there exists for all $ 0\leq i\leq k$
a unique arc $\widetilde{a_i^{w_i}}\in U$ obtained by lifting
the arc $a_i$ of $N$ at the vertex $w_i$ of $N$.

\rev{Now, suppose $v\in V$. To define $\kappa_{\mu,f}(v)$ in a top-down fashion, 
we begin with the base case  $v=\rho_T$.
Since we assume that $|L(T)|\geq 2$, either {\bf (i)} $t(v)=\bul$ or {\bf (ii)} $t(v)=\squ$.}

\begin{owndesc}
	\item \rev{Case (i):} If $t(v)=\bul$, then $\mu(v)$ is a vertex of $N$. Note that
	$\mu(v)\neq\rho_N$ must hold as $\mu$ satisfies Property~(R2.i). 
	  Hence, the directed path
          $P(\rho_N, \mu(v)) = (w_0= \rho_N,w_1,\ldots, w_k,w_{k+1} = \mu(v))$
	must cross at least one arc of $N$, i.e, $k\geq 0$.
	Since $f$ is a folding of $M$ into $N$ there exists a path
	in $M$ from $\rho_M=\rho_N$ to the head $h_M(\widetilde{a_k^{w_k}})$
	of the arc $\widetilde{a_k^{w_k}}\in U$ obtained by
	lifting $a_k$ at $w_k = \parent(\mu(v))$.
	Then we put $\kappa_{\mu,f}(v)=h_M(\widetilde{a_k^{w_k}})$.

	\smallskip
	\item \rev{Case (ii):}  If $t(v)=\squ$ then $\mu(v)$ is an arc of $N$.
	  If $\rho_N= t_N(\mu(v))$ then we put
          $\kappa_{\mu,f}(v)=\widetilde{a_0^{w_0}}$.
			\rev{In case $\rho_N\neq t_N(\mu(v))$,
	                  the directed path $P(\rho_N, t_N(\mu(v)))
                           = (w_0= \rho_N,w_1,\ldots, w_k,w_{k+1} = \mu(v))$
                          has at least one arc.}
	Since $f$ is a folding from $M$ to $N$ there exists a path
	from $\rho_M$ to the tail $t_M(\widetilde{a_k^{w_k}})$
	of the arc $\widetilde{a_k^{w_k}}\in U$
        obtained by lifting
	$a_k$ at $w_k$. \rev{So} we put $\kappa_{\mu,f}(v)=\widetilde{a_k^{w_k}}$.
\end{owndesc}

\rev{Now, assume} that $v\neq \rho_T$ is such that $\kappa_{\mu,f}(u)$
has already been defined for all vertices $u$ of $T$ that are above
$v$. Let $w\in V$ denote the parent of $v$ which must exist as $v\not=\rho_T$. If
$\mu(v)=\mu(w)$ then we put $\kappa_{\mu,f}(v)=\kappa_{\mu,f}(w)$. So
assume that $\mu(v)\not=\mu(w)$.
Note that $\mu(w)\not=\rho_N$, independent of whether
$t(w)=\bul$ or $t(w)=\squ$. We distinguish between the two cases
that $v$ is a leaf of $T$ and that it is not.

If $v$ is a leaf of $T$ then
$\mu(v)$ must be a leaf of $N$. Consider the directed path $P$ obtained by
extending the directed path
$P(\rho_N, \mu(w))$ of $N$ by the directed path $P(h_N(\mu(w)),\mu(v))$.
Note that $P$ must contain at least two arcs.
Then we put $\kappa_{\mu,f}(v) =h_M(\widetilde{b_q^{u_q}})$ where
 $u_q$ is the last but one vertex on $P$
and $b_q$ is the arc $(u_q, \mu(v))$ of $N$.

So assume that $v$ is not a leaf of $T$. To define $\kappa_{\mu,f}$, 
we need to distinguish between
the cases that {\bf(a)} $t(v)=t(w)=\bul$, {\bf(b)} $t(v)=\bul$ and
$t(w)=\squ$, {\bf(c)}
$t(v)=\squ$ and $t(w)=\bul$, and {\bf(d)} $t(v)=t(w)=\squ$
\rev{(note that since $\mu$ is a TreeNet-reconciliation from $(T;t,\sigma)$ to
$N$,  Property (R3) implies that $\mu(v) \preceq_N \mu(w)$ holds in all
of these cases)}. 

\begin{owndesc}
\item[{\em Case (a):}] If $t(v)=t(w)=\bul$, then $\mu(v)$
  and $\mu(w)$ are vertices of $N$.
  Assume first that none of the children of $w$ have
  obtained an image under $\kappa_{\mu,f}$.
Consider the directed path $P$ obtained by extending a
directed path $P(\rho_N, \mu(w))$ of $N$ 
 by a directed path $P(\mu(w),\mu(v))$. Then the
choice of $v$ combined with the fact that $f$ is a folding from $M$
to $N$ implies that there exists a path in $M$ from
$\rho_M$ via $\kappa_{\mu,f}(w)$ to $h_M(\widetilde{b_q^{u_q}})$
where $u_q$ is the last but one vertex on $P$
and $b_q$ is the arc $(u_q, \mu(v))$ of $N$. Note that $u_q$ must exist
since $\mu(v)\not=\mu(w)$. In this case, we put
$\kappa_{\mu,f}(v)= h_M(\widetilde{b_q^{u_q}})$.

Putting $v=v_1$, we can continue with all other children
$v_2,\dots,v_k$ of $w$. Since we want  $\kappa_{\mu,f}$ to
be a MUL-reconciliation, we need to
ensure that  $\kappa_{\mu,f}$ satisfies (M2.i). 

  Since $\mu$ satisfies Property~(R2.i), there must exist children
  $w_i$ and $w_j$ of $w$ such that  $\mu(w) \in Q_N(\mu(w_i),\mu(w_j))$.
  For $l\in\{i,j\}$ put $P_l:=P(\mu_w, z_l)$ in case $z_l$ is a vertex
  of $N$  and $P_l:=P(\mu_w, h_N(z_l))$ in case $z_l$ is an arc of $N$.
  Then since $\mu(w)$ separates $\mu(w_i)$ and $\mu(w_j)$ it follows
  that regardless of whether or not $\mu(w_i)$ and $\mu(w_j)$ are comparable,
  we may always choose directed paths  $P_i:=P(\mu_w, z_i)$ and
	$P_j:=P(\mu_w, \mu(z_j))$ in such a way that the
  first arc on $P_i$ is distinct from the first arc on $P_j$.

  \rev{For each $l\in \{2,\ldots, k\}$, let $P_l$ denote a path
    that is chosen this way. Assuming that $\kappa_{\mu,f}(v_i)$ has already
    been defined for some $1\leq i<k$, we then put
    $\kappa_{\mu,f}(v_{i+1})=h_M(\widetilde{b_q^{u_q}})$ where $u_q$ is the last vertex
    on $P_{i+1}$and $b_q$ is the arc $(u_q,\mu(v_{i+1}))$ of $N$.  }

	\smallskip

      \item[{\em Case (b):}] If  $t(v)=\bul$ and $t(w)=\squ$, then   $\mu(v)$
        is a vertex of $N$ and $\mu(w)$
is an arc of $N$. If  $h_N(\mu(w))=\mu(v)$
then we put $\kappa_{\mu,f}(v)=h_M(\kappa_{\mu,f}(w))$.
So assume that $h_N(\mu(w))\not =\mu(v)$.
Consider the directed path $P$ obtained by extending
a directed path $P(\rho_N, h_N(\mu(w)))$ of $N$ 
by a directed path $P(h_N(\mu(w)),\mu(v))$. 
Then the choice of $v$ combined with the fact that $f$ is a folding
of $M$ into $N$ implies that there exists a path from $\rho_M$
via $h_M(\kappa_{\mu,f}(w))$ to $h_M(\widetilde{b_q^{u_q}})$ where $u_q$
and $b_q$ are as in Case (a). Then we put
$\kappa_{\mu,f}(v)= h_M(\widetilde{b_q^{u_q}})$.
\smallskip

\item[{\em Case (c):}] If $t(v)=\squ$ and $t(w)=\bul$,
  then $\mu(v)$ is an arc of $N$ and $\mu(w)$
  is a vertex of $N$. Again, we assume first that none of the
  children of $w$ have been assigned a value under 
 $\kappa_{\mu,f}$. If  $t_N(\mu(v))=\mu(w)$
then we define $\kappa_{\mu,f}(v)=\widetilde{\mu(v)^{\mu(w)}}$.
So assume  $t_N(\mu(v))\neq \mu(w)$. 
Consider the directed path $P$ obtained by extending
a directed directed $P(\rho_N, \mu(w))$ of $N$ by a directed path $P(\mu(w),h_N(\mu(v)))$. Then the
choice of $v$ combined with the fact that $f$ is a folding
from $M$ to $N$ implies that there exists a directed path from $\rho_M$
via $\kappa_{\mu,f}(w)$ to $h_M(\widetilde{b_q^{u_q}})$
where $u_q$ is the last but one vertex of $P$
and $b_q$ is the arc $(u_q, h_N(\mu(v)))$ of $N$. Then we put
$\kappa_{\mu,f}(v)=\widetilde{b_q^{u_q}}$.
Putting again $v=v_1$, we can continue with all other children $v_2,\dots,v_k$ of $w$
	 in the same way as in Case (a). Employing analogous arguments,  
	it follows that $\kappa_{\mu,f}$ satisfies (M2.i) for Case~(c). 
\smallskip

\item[{\em Case (d):}] If $t(v)=t(w)=\squ$, then both $\mu(v)$ and $\mu(w)$
are arcs of $N$. Consider the directed path $P$
obtained by extending a directed path $P(\rho_N, h_N(\mu(w)))$ of $N$ 
by a directed path $P(h_N(\mu(w)),h_N(\mu(v)))$. Then the
choice of $v$ combined with the fact that $f$ is a folding
from $M$ into $N$ implies that there exists a directed path from $\rho_M$
via  $t_N(\kappa_{\mu,f}(w))$, $h_N(\kappa_{\mu,f}(w))$ and
$t_M(\widetilde{b_q^{u_q}})$
(note that $t_M(\widetilde{b_q^{u_q}})$ and $h_N(\kappa_{\mu,f}(w))$
might coincide)
to $h_M(\widetilde{b_q^{u_q}})$
where $u_q$ and $b_q$ are as in Case (c). Then we 
put $\kappa_{\mu,f}(v)=\widetilde{b_q^{u_q}}$.
\end{owndesc}

\rev{This completes the definition of $\kappa_{\mu,f}$.
Note that $\kappa_{\mu,f}$ can be regarded as a ``lifting" of
the map $\mu$ along $f$ in the sense discussed in \cite[Section 5]{huber2016folding}.
We now show that $\kappa_{\mu,f}$ is indeed
a MUL-reconciliation map from  $(T;t,\sigma)$ to $(M,\chi)$.}

\begin{proposition}\label{claim2}
Let $(T;t,\sigma)$ be an event-labeled gene tree, 
let $N$ be a species network on $\Spe$, and let $\mu$ be
a TreeNet-reconciliation map from $(T;t,\sigma)$ to $N$. 
Moreover, let $(M,\chi)$ be a pseudo MUL-tree that can be folded 
   via a folding map $f=(f_V,f_E)$ to $N$.
Then, $\kappa_{\mu,f}$ is
a MUL-reconciliation map from  $(T;t,\sigma)$ to $(M,\chi)$.
\end{proposition}
\begin{proof} 
	Let $T=(V,E)$, let $N=(W,F)$ and let $M=(D,U)$.
	Note that in view of Proposition~\ref{prop:iso-mul-fold},
        we may assume w.l.o.g. that 
		$(M,\chi) = (U^*(N), \chi^*)$. Thus, $\rho_N = \rho_M$.

  Clearly the map $\kappa_{\mu,f}$ is well-defined. Also, Property~(M1) holds
  in view of $\mu$ satisfying Property~(R1). 

  To see that $\kappa_{\mu,f}$ satisfies
  Property~(M2) assume that $v\in V\backslash L(T)$. 
	If $t(v)=\squ$ then 
  $\mu(v)$ is an arc of $N$. Hence,
        $\kappa_{\mu,f}(v)\in U$, by the definition of $\kappa_{\mu,f}(v)$.
        So Property~(M2.ii)
  must hold.
	If $t(v)=\bul$, then  Cases (a) and (c) already 
	imply that Property~(M2.i) is satisfied. 

  To see that $\kappa_{\mu,f}$ satisfies Property~(M3), assume that $x,y\in V$
  such that $x\prec_T y$. We start with establishing Property~(M3.i).
  Assume that $ t(x)=t(y)=\squ$. Then  $\mu(x)$ and $\mu(y)$ are arcs of $N$.
  Since $\mu(x)\preceq_N\mu(y)$ as $\mu$
  satisfies Property~(R3.i) it follows that there exists a directed path $P$
  in $N$ from $\rho_N$ to $\mu(x)$ that contains $\mu(y)$. Since $M=U^*(N)$, the
	definition of  $\kappa_{\mu,f}$ implies that $\kappa_{\mu,f}(x)\preceq_M\kappa_{\mu,f}(y)$.
  Thus,  Property~(M3.i) must hold.

  To see that Property~(M3.ii) holds assume that at least
  one of $t(x)$ and $t(y)$
  equals $\bul$. Then $\mu(x)\prec_N \mu(y)$ as $\mu$ satisfies Property~(R3.ii).
  We first consider the case that $t(x)=\bul$. Then $\mu(x)$ is a vertex
  of $N$.  Hence, $\mu(x)\preceq_N h_N(\mu(y))$ in case $\mu(y)\in F$
	 and $\mu(x)\prec_N \mu(y)$ in case $\mu(y)\in W$.
  In either case, there
  exists a directed path in $N$ from $\rho_N$ to $\mu(x)$
  that contains $\mu(y)$.
	Since $M=U^*(N)$, the
	definition of  $\kappa_{\mu,f}$ implies that
  $\kappa_{\mu,f}(x)\prec \kappa_{\mu,f}(y)$. Thus, Property~(M3.ii) holds.

  Finally, assume that  $t(x)=\squ$. Then $\mu(x)$ is an arc of $N$
  and $t(y)=\bul$. Hence, $\mu(y)\in W$. Consequently, there
  exists a directed path in $N$ from $\rho_N$ to $\mu(x)$ that crosses $\mu(y)$.
	Since $M=U^*(N)$, the
	definition of  $\kappa_{\mu,f}$ implies that
  $\kappa_{\mu,f}(x)\prec \kappa_{\mu,f}(y)$. Thus,  Property~(M3.ii) must hold
  in this final case too.
 \qed \end{proof}

As immediate consequences of Theorems~\ref{universal-folding},
\ref{thm:composed} and
Proposition~\ref{claim2} we obtain the result that we promised above.

\begin{theorem}\label{important}
  Suppose $(T;t,\sigma)$ is an event-labeled gene tree and
  $N$ is a  multi-arc free network. Then $N$ is a species network for
  $(T;t,\sigma)$ if and only if $U^*(N)$ is a pseudo-MUL-tree for $(T;t,\sigma)$. 
\label{thm:NM}
\end{theorem}

Note that the existence of reconciliations  from an 
event-labeled gene tree to both a MUL-tree $M$ and a network \rev{$N^*$}, does 
not necessarily imply that there exists a folding map from 
$M$ to \rev{$N^*$} -- see Fig.~\ref{fig:no-folding}. This issue is related to the fact 
that Theorem~\ref{important} is stated in terms of  $U^*(N)$ and not $U(N)$, 
and that a network $N$ is not always a folding 
of $U(N)$ (cf. Theorem~\ref{universal-folding}). 

\begin{figure}[tbp]
  \begin{center}
    \includegraphics[width=.95\textwidth]{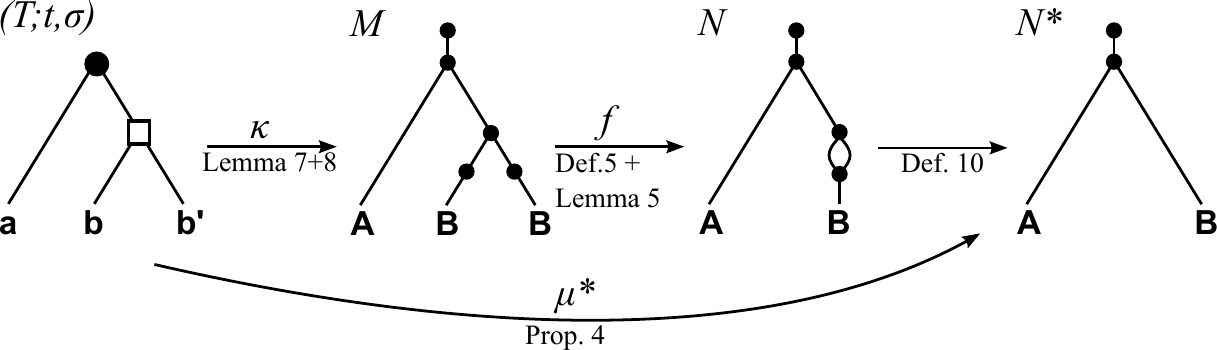}
  \end{center}
	\caption{An event-labeled gene tree $(T;t,\sigma)$ \rev{with 
	speciation events $\bullet$ and duplication events $\square$, }
	the simple
	subdivision $M$ of the  MUL-tree $M(T;t,\sigma)$, the network $N$
	constructed in Definition \ref{def:MULfoldN}, and the network
	$N^*$ associated to $N$ as specified in Definition~\ref{def:Nstar}.
	\rev{Lemma \ref{trivial} and \ref{lem:reconcMUL-pseudoMUL} imply
	that there is a MUL-tree reconciliation map $\kappa$ from 
	from $(T;t,\sigma)$ to $M$. As outlined in the proof of 
	Lemma \ref{lem:MULfoldN}, there is a folding map $f$
	from $M$ to $N$. In addition, Proposition \ref{prop:reconcNstar}  
	implies that there is is a TreeNet-reconciliation map $\mu^*$
	from  $(T;t,\sigma)$ to $N^*$. However,} 
	there does not exist a folding map from $M$ to $N^*$.
    } 
	\label{fig:no-folding}
\end{figure}

\section{Outlook}

In this paper, we have introduced the concept of  a reconciliation map
between an arbitrary event-labeled gene tree and a species network, called a
\emph{TreeNet-reconciliation map}. 
In particular, we have shown that for every \rev{well-behaved} event-labeled gene
tree \rev{$(T;t,\sigma)$} there is always a (multi-arc free) species network $N$ \rev{together with}
a TreeNet-reconciliation map $\mu$ \rev{from $(T;t,\sigma)$ to $N$}
  such that $N$ displays all informative triples in $\mathcal{S}(T;t,\sigma)$. 
Both the network and the reconciliation map can be constructed in polynomial-time. 

These results open up several interesting new avenues for research. For example,  
although we know that for every  event-labeled gene tree \rev{$(T;t,\sigma)$} there
is always a species network $N$ \rev{for that tree}, we do not know much about
the computational complexity of determining whether $(T;t,\sigma)$ can
be reconciled with an arbitrary given 
species network $N$. Clearly, if the unfolded network $U^*(N)$ is
isomorphic to  \emph{some} subdivision
of the MUL-tree $(M(T;t,\sigma),\chi)$, then we can \rev{re-use} the latter
results based on the \rev{MUL-}reconciliation map
$\kappa$ and the folding $f$ from $(M(T;t,\sigma),\chi)$ to $N$ to conclude that $N$ is a 
network for $(T;t,\sigma)$. However, not all TreeNet-reconciliation maps $\mu$
can be expressed in terms of $\kappa$ and $f$; see Fig.\  \ref{fig:no-folding}. 
The results and proof techniques in \cite[Section 7]{huber2016folding}
may offer an  avenue to solving these problems. \rev{In this regards it could also be 
interesting  to understand whether the notion of a folding map could
be adapted so that it more directly captures the relation between MUL-trees and networks}. 

In another direction, the structure of the species network 
obtained from an event-labeled gene tree is heavily dependent on our construction via
MUL-trees and foldings, and so it is not clear which
properties these networks will enjoy. By way of example, 
the bi-connected component in the network in Fig.\ \ref{fig:multiarc-free}
contains two hybrid vertices and it is not clear 
whether or not a less complex network for the given gene tree may exist
(e.g. with fewer hybrid vertices).  It would be interesting to investigate 
properties of the species network, especially how far
they are from being minimal under various criteria. 

Finally, we have only considered 
speciation and duplication events. 
It would be interesting to develop a theory to cope with other events such as 
horizontal gene transfer.  Note that several 
results have been derived for
accommodating gene transfer in reconciliation models 
between gene trees and species trees, both for 
unlabeled \cite{hallett2001efficient} and 
event-labeled gene trees \cite{Hellmuth2017,Nojgaard2018}.

\section*{Acknowledgements}
MH would like to thank the School of Computing Sciences, University of East Anglia, and 
KH and VM would like to thank the Institute of Mathematics and Computer Science, University of Greifswald,  
for helping to make two visits possible during which this work was conceived and developed.

\bibliographystyle{plain}       
\bibliography{biblio}   

\begin{thebibliography}{10}

\bibitem{Altenhoff:09}
A~M Altenhoff and C.~Dessimoz.
\newblock Phylogenetic and functional assessment of orthologs inference
  projects and methods.
\newblock {\em PLoS Comput Biol.}, 5:e1000262, 2009.

\bibitem{AGGD:13}
A~M Altenhoff, M~Gil, G~H Gonnet, and C~Dessimoz.
\newblock Inferring hierarchical orthologous groups from orthologous gene
  pairs.
\newblock {\em PLoS ONE}, 8(1):e53786, 2013.

\bibitem{ASGD:11}
A~M Altenhoff~\emph{et al.}
\newblock The {O}{M}{A} orthology database in 2015: function predictions,
  better plant support, synteny view and other improvements.
\newblock {\em Nucleic Acids Res}, 43(D1):D240--D249, 2015.

\bibitem{bapteste2013networks}
Eric Bapteste, Leo van Iersel, Axel Janke, Scot Kelchner, Steven Kelk, James~O
  McInerney, David~A Morrison, Luay Nakhleh, Mike Steel, Leen Stougie, et~al.
\newblock Networks: expanding evolutionary thinking.
\newblock {\em Trends in Genetics}, 29(8):439--441, 2013.

\bibitem{CMSR:06}
F~Chen, A~J Mackey, C~J Stoeckert, and D~S Roos.
\newblock Ortho{M}{C}{L}-db: querying a comprehensive multi-species collection
  of ortholog groups.
\newblock {\em Nucleic Acids Res}, 34(S1):D363--D368, 2006.

\bibitem{cui2012polynomial}
Yun Cui, Jesper Jansson, and Wing-Kin Sung.
\newblock Polynomial-time algorithms for building a consensus mul-tree.
\newblock {\em Journal of Computational Biology}, 19(9):1073--1088, 2012.

\bibitem{czabarka2013generating}
{\'E}va Czabarka, PL~Erd{\H{o}}s, V~Johnson, and V~Moulton.
\newblock Generating functions for multi-labeled trees.
\newblock {\em Discrete Applied Mathematics}, 161(1-2):107--117, 2013.

\bibitem{DEML:16}
Riccardo Dondi, Nadia El-Mabrouk, and Manuel Lafond.
\newblock Correction of weighted orthology and paralogy relations-complexity
  and algorithmic results.
\newblock In {\em International Workshop on Algorithms in Bioinformatics},
  pages 121--136. Springer, 2016.

\bibitem{DES:14}
Riccardo Dondi, Nadia El-Mabrouk, and Krister~M. Swenson.
\newblock Gene tree correction for reconciliation and species tree inference:
  Complexity and algorithms.
\newblock {\em Journal of Discrete Algorithms}, 25:51 -- 65, 2014.
\newblock 23rd Annual Symposium on Combinatorial Pattern Matching.

\bibitem{dondi2017approximating}
Riccardo Dondi, Manuel Lafond, and Nadia El-Mabrouk.
\newblock Approximating the correction of weighted and unweighted orthology and
  paralogy relations.
\newblock {\em Algorithms for Molecular Biology}, 12(1):4, 2017.

\bibitem{DONDI17}
Riccardo Dondi, Giancarlo Mauri, and Italo Zoppis.
\newblock Orthology correction for gene tree reconstruction: Theoretical and
  experimental results.
\newblock {\em Procedia Computer Science}, 108:1115 -- 1124, 2017.
\newblock International Conference on Computational Science, ICCS 2017, 12-14
  June 2017, Zurich, Switzerland.

\bibitem{DCH:09}
Jean-Philippe Doyon, Cedric Chauve, and Sylvie Hamel.
\newblock Space of gene/species trees reconciliations and parsimonious models.
\newblock {\em Journal of Computational Biology}, 16(10):1399--1418, 2009.

\bibitem{DRDB11}
Jean-Philippe Doyon, Vincent Ranwez, Vincent Daubin, and Vincent Berry.
\newblock Models, algorithms and programs for phylogeny reconciliation.
\newblock {\em Briefings in bioinformatics}, 12(5):392--400, 2011.

\bibitem{Doyon2010}
Jean-Philippe Doyon, Celine Scornavacca, K.~Yu. Gorbunov, Gergely~J.
  Sz{\"o}ll{\H{o}}si, Vincent Ranwez, and Vincent Berry.
\newblock {\em An Efficient Algorithm for Gene/Species Trees Parsimonious
  Reconciliation with Losses, Duplications and Transfers}, pages 93--108.
\newblock Springer Berlin Heidelberg, Berlin, Heidelberg, 2010.

\bibitem{EHL10}
Oliver Eulenstein, Snehalata Huzurbazar, and David~A Liberles.
\newblock Reconciling phylogenetic trees.
\newblock {\em Evolution after gene duplication}, pages 185--206, 2010.

\bibitem{Fitch:70}
W~M Fitch.
\newblock Distinguishing homologous from analogous proteins.
\newblock {\em Syst Zool}, 19:99--113, 1970.

\bibitem{Fitch:00}
Walter~M. Fitch.
\newblock Homology: a personal view on some of the problems.
\newblock {\em Trends Genet.}, 16:227--231, 2000.

\bibitem{gontier2015reticulate}
Nathalie Gontier.
\newblock Reticulate evolution everywhere.
\newblock In {\em Reticulate evolution}, pages 1--40. Springer, 2015.

\bibitem{GCMRM:79}
Morris Goodman, John Czelusniak, G.~William Moore, A.~E. Romero-Herrera, and
  Genji Matsuda.
\newblock Fitting the gene lineage into its species lineage, a parsimony
  strategy illustrated by cladograms constructed from globin sequences.
\newblock {\em Systematic Biology}, 28(2):132--163, 1979.

\bibitem{GJ:06}
P.~Gorecki and J.~Tiuryn.
\newblock {DLS}-trees: A model of evolutionary scenarios.
\newblock {\em Theoretical Computer Science}, 359(1):378 -- 399, 2006.

\bibitem{gregg2017gene}
WC~Thomas Gregg, S~Hussain Ather, and Matthew~W Hahn.
\newblock Gene-tree reconciliation with mul-trees to resolve polyploidy events.
\newblock {\em Systematic biology}, 66(6):1007--1018, 2017.

\bibitem{hallett2001efficient}
Michael~T Hallett and Jens Lagergren.
\newblock Efficient algorithms for lateral gene transfer problems.
\newblock In {\em Proceedings of the fifth annual international conference on
  Computational biology}, pages 149--156. ACM, 2001.

\bibitem{HES:14}
Reza Hassanzadeh, Changiz Eslahchi, and Wing-Kin Sung.
\newblock Do triplets have enough information to construct the multi-labeled
  phylogenetic tree?
\newblock {\em PLOS ONE}, 9(7):1--10, 07 2014.

\bibitem{HHH+13}
M.~Hellmuth, M.~Hernandez-Rosales, K.~T. Huber, V.~Moulton, P.~F. Stadler, and
  N.~Wieseke.
\newblock Orthology relations, symbolic ultrametrics, and cographs.
\newblock {\em J. Math. Biology}, 66(1-2):399--420, 2013.

\bibitem{HW:16b}
M.~Hellmuth and N.~Wieseke.
\newblock From sequence data including orthologs, paralogs, and xenologs to
  gene and species trees.
\newblock In P.~Pontarotti, editor, {\em Evolutionary Biology: Convergent
  Evolution, Evolution of Complex Traits, Concepts and Methods}, pages
  373--392, Cham, 2016. Springer.

\bibitem{Hellmuth2017}
Marc Hellmuth.
\newblock Biologically feasible gene trees, reconciliation maps and informative
  triples.
\newblock {\em Algorithms for Molecular Biology}, 12(1):23, 2017.

\bibitem{HW:16}
Marc Hellmuth, Peter~F Stadler, and Nicolas Wieseke.
\newblock The mathematics of xenology: Di-cographs, symbolic ultrametrics,
  2-structures and tree-representable systems of binary relations.
\newblock {\em Journal of Mathematical Biology}, 75(1):199--237, 2017.

\bibitem{HHH+12}
M.~Hernandez-Rosales, M.~Hellmuth, N.~Wieseke, K.~T. Huber, V.~Moulton, and
  P.~F. Stadler.
\newblock From event-labeled gene trees to species trees.
\newblock {\em BMC Bioinformatics}, 13(Suppl 19):S6, 2012.

\bibitem{HM:06}
K~T Huber and V.~Moulton.
\newblock Phylogenetic networks from multi-labelled trees.
\newblock {\em Journal of Mathematical Biology}, 52(5):613--632, 2006.

\bibitem{huber2016folding}
Katharina~T Huber, Vincent Moulton, Mike Steel, and Taoyang Wu.
\newblock Folding and unfolding phylogenetic trees and networks.
\newblock {\em Journal of Mathematical Biology}, 73(6-7):1761--1780, 2016.

\bibitem{huber2006reconstructing}
Katharina~T Huber, Bengt Oxelman, Martin Lott, and Vincent Moulton.
\newblock Reconstructing the evolutionary history of polyploids from
  multilabeled trees.
\newblock {\em Molecular Biology and Evolution}, 23(9):1784--1791, 2006.

\bibitem{huson_rupp_scornavacca_2010}
Daniel~H. Huson, Regula Rupp, and Celine Scornavacca.
\newblock {\em Phylogenetic Networks: Concepts, Algorithms and Applications}.
\newblock Cambridge University Press, 2010.

\bibitem{huson2011survey}
Daniel~H Huson and Celine Scornavacca.
\newblock A survey of combinatorial methods for phylogenetic networks.
\newblock {\em Genome biology and evolution}, 3:23--35, 2011.

\bibitem{LDEM:16}
Manuel Lafond, Riccardo Dondi, and Nadia El-Mabrouk.
\newblock The link between orthology relations and gene trees: a correction
  perspective.
\newblock {\em Algorithms for Molecular Biology}, 11(1):1, 2016.

\bibitem{Lafond2014}
Manuel Lafond and Nadia El-Mabrouk.
\newblock Orthology and paralogy constraints: satisfiability and consistency.
\newblock {\em BMC Genomics}, 15(6):S12, 2014.

\bibitem{lafond2015orthology}
Manuel Lafond and Nadia El-Mabrouk.
\newblock Orthology relation and gene tree correction: complexity results.
\newblock In {\em International Workshop on Algorithms in Bioinformatics},
  pages 66--79. Springer, 2015.

\bibitem{lafond2012optimal}
Manuel Lafond, Krister~M Swenson, and Nadia El-Mabrouk.
\newblock An optimal reconciliation algorithm for gene trees with polytomies.
\newblock In {\em International Workshop on Algorithms in Bioinformatics},
  pages 106--122. Springer, 2012.

\bibitem{Lechner:11a}
M~Lechner, S~Findei{\ss}, L~Steiner, M~Marz, P~F Stadler, and S~J Prohaska.
\newblock \texttt{Proteinortho:} detection of (co-)orthologs in large-scale
  analysis.
\newblock {\em BMC Bioinformatics}, 12:124, 2011.

\bibitem{Lechner:14}
Marcus Lechner, Maribel Hernandez-Rosales, D.~Doerr, N.~Wiesecke, A.~Thevenin,
  J.~Stoye, Roland~K. Hartmann, Sonja~J. Prohaska, and Peter~F. Stadler.
\newblock Orthology detection combining clustering and synteny for very large
  datasets.
\newblock {\em PLoS ONE}, 9(8):e105015, 08 2014.

\bibitem{lott2009inferring}
Martin Lott, Andreas Spillner, Katharina~T Huber, Anna Petri, Bengt Oxelman,
  and Vincent Moulton.
\newblock Inferring polyploid phylogenies from multiply-labeled gene trees.
\newblock {\em BMC evolutionary biology}, 9(1):216, 2009.

\bibitem{BLZ:00}
B.~Ma, M.~Li, and L.~Zhang.
\newblock From gene trees to species trees.
\newblock {\em SIAM Journal on Computing}, 30(3):729--752, 2000.

\bibitem{MMZ:00}
Bin Ma, Ming Li, and Louxin Zhang.
\newblock From gene trees to species trees.
\newblock {\em SIAM Journal on Computing}, 30(3):729--752, 2000.

\bibitem{Nojgaard2018}
Nikolai N{\o}jgaard, Manuela Gei{\ss}, Daniel Merkle, Peter~F. Stadler, Nicolas
  Wieseke, and Marc Hellmuth.
\newblock Time-consistent reconciliation maps and forbidden time travel.
\newblock {\em Algorithms for Molecular Biology}, 13(1):2, 2018.

\bibitem{page1998genetree}
RD~Page.
\newblock Genetree: comparing gene and species phylogenies using reconciled
  trees.
\newblock {\em Bioinformatics}, 14(9):819--820, 1998.

\bibitem{posada2016phylogenomics}
David Posada.
\newblock Phylogenomics for systematic biology.
\newblock {\em Systematic biology}, 65(3):353--356, 2016.

\bibitem{RLG+14}
L~Yu Rusin, EV~Lyubetskaya, K~Yu Gorbunov, and VA~Lyubetsky.
\newblock Reconciliation of gene and species trees.
\newblock {\em BioMed research international}, 2014, 2014.

\bibitem{scornavacca2009gene}
Celine Scornavacca, Vincent Berry, and Vincent Ranwez.
\newblock From gene trees to species trees through a supertree approach.
\newblock In {\em International Conference on Language and Automata Theory and
  Applications}, pages 702--714. Springer, 2009.

\bibitem{SMC:17}
Celine Scornavacca, Joan Carles~Pons Mayol, and Gabriel Cardona.
\newblock Fast algorithm for the reconciliation of gene trees and {LGT}
  networks.
\newblock {\em Journal of Theoretical Biology}, 418:129 -- 137, 2017.

\bibitem{inparanoid:10}
E.L.L. Sonnhammer and G.~{\"O}stlund.
\newblock Inparanoid 8: orthology analysis between 273 proteomes, mostly
  eukaryotic.
\newblock {\em Nucleic Acids Research}, 43(D1):D234--D239, 2015.

\bibitem{steel2016phylogeny}
Mike Steel.
\newblock {\em Phylogeny: discrete and random processes in evolution}.
\newblock SIAM, 2016.

\bibitem{SLX+12}
Maureen Stolzer, Han Lai, Minli Xu, Deepa Sathaye, Benjamin Vernot, and Dannie
  Durand.
\newblock Inferring duplications, losses, transfers and incomplete lineage
  sorting with nonbinary species trees.
\newblock {\em Bioinformatics}, 28(18):i409, 2012.

\bibitem{SD:12}
Gergely~J. Sz{\"o}ll{\H{o}}si and Vincent Daubin.
\newblock Modeling gene family evolution and reconciling phylogenetic discord.
\newblock In Maria Anisimova, editor, {\em Evolutionary Genomics: Statistical
  and Computational Methods, Volume 2}, pages 29--51, Totowa, NJ, 2012. Humana
  Press.

\bibitem{STDB:15}
Gergely~J. Sz{\"o}ll{\H{o}}si, Eric Tannier, Vincent Daubin, and Bastien
  Boussau.
\newblock The inference of gene trees with species trees.
\newblock {\em Systematic Biology}, 64(1):e42, 2015.

\bibitem{TG+00}
R~L Tatusov, M~Y Galperin, D~A Natale, and E~V Koonin.
\newblock The {C}{O}{G} database: a tool for genome-scale analysis of protein
  functions and evolution.
\newblock {\em Nucleic Acids Research}, 28(1):33--36, 2000.

\bibitem{To2015}
Thu-Hien To and Celine Scornavacca.
\newblock Efficient algorithms for reconciling gene trees and species networks
  via duplication and loss events.
\newblock {\em BMC Genomics}, 16(10):S6, Oct 2015.

\bibitem{Tofigh2011}
A.~Tofigh, M.~Hallett, and J.~Lagergren.
\newblock Simultaneous identification of duplications and lateral gene
  transfers.
\newblock {\em IEEE/ACM Trans Comput Biol Bioinf}, 8, 2011.

\bibitem{T+11}
K~Trachana, T~A Larsson, S~Powell, W-H Chen, T~Doerks, J~Muller, and P~Bork.
\newblock Orthology prediction methods: A quality assessment using curated
  protein families.
\newblock {\em BioEssays}, 33(10):769--780, 2011.

\bibitem{VSGD:08}
Benjamin Vernot, Maureen Stolzer, Aiton Goldman, and Dannie Durand.
\newblock Reconciliation with non-binary species trees.
\newblock {\em Journal of Computational Biology}, 15(8):981--1006, 2008.

\end{thebibliography}

\end{document}